\documentclass[a4paper,10pt,reqno]{amsart}
\usepackage[T1]{fontenc}
\usepackage[utf8]{inputenc}
\usepackage{lmodern}
\usepackage{amsmath} 
\usepackage{amsfonts}
\usepackage{amscd} 
\usepackage{amssymb} 
\usepackage{enumerate}
\usepackage{mathrsfs}
\usepackage[all]{xy}
\makeatletter
\@addtoreset{equation}{subsection}

\makeatother
\newcommand{\termin}[1]{{\em #1}}
\newcommand{\bA}{\mathbf{A}}\newcommand{\bC}{\mathbf{C}}
\newcommand{\bF}{\mathbf{F}}\newcommand{\bG}{\mathbf{G}}
\newcommand{\bL}{\mathbf{L}}
\newcommand{\bN}{\mathbf{N}}\newcommand{\bP}{\mathbf{P}}
\newcommand{\bQ}{\mathbf{Q}}\newcommand{\bR}{\mathbf{R}}

\newcommand{\bZ}{\mathbf{Z}}

\newcommand{\bd}{\boldsymbol{d}}
\newcommand{\be}{\boldsymbol{e}}\newcommand{\bbf}{\boldsymbol{f}}

\newcommand{\bm}{\boldsymbol{m}}\newcommand{\bn}{\boldsymbol{n}}
\newcommand{\bt}{\boldsymbol{t}}

\newcommand{\cA}{{\mathcal A}}\newcommand{\cC}{{\mathcal C}}\newcommand{\cD}{{\mathcal D}}
\newcommand{\cE}{{\mathcal E}}
\newcommand{\cI}{{\mathcal I}}\newcommand{\cL}{{\mathcal L}}
\newcommand{\cM}{{\mathcal M}}\newcommand{\cP}{{\mathcal P}}
\newcommand{\cT}{{\mathcal T}}
\newcommand{\cX}{{\mathcal X}}

\DeclareMathAlphabet{\eulercal}{U}{eus}{m}{n}

\newcommand{\ecD}{{\eulercal D}}
\newcommand{\ecF}{{\eulercal F}}
\newcommand{\ecH}{{\eulercal H}}

\newcommand{\ecM}{{\eulercal M}}
\newcommand{\ecO}{{\eulercal O}}

\newcommand{\ecT}{{\eulercal T}}

\DeclareMathAlphabet{\beulercal}{U}{eus}{b}{n}

\newcommand{\becD}{{\beulercal D}}
\newcommand{\becE}{{\beulercal E}}

\newcommand{\mfS}{\mathfrak{S}}

\newcommand{\scC}{\mathscr{C}}
\newcommand{\scF}{\mathscr{F}}
\newcommand{\scI}{\mathscr{I}}
\newcommand{\scN}{\mathscr{N}}\newcommand{\scO}{\mathscr{O}}
\newcommand{\scT}{\mathscr{T}}
\newcommand{\scX}{\mathscr{X}}
\newcommand{\scZ}{\mathscr{Z}}

\newcommand{\longto}{\longrightarrow}
\newcommand{\longlto}{\longleftarrow}
\newcommand{\inject}{\hookrightarrow}

\newcommand{\isom}{\overset{\sim}{\to}}
\newcommand{\longisom}{\overset{\sim}{\longto}}
\DeclareMathOperator{\End}{End}

\DeclareMathOperator{\Pic}{Pic}

\DeclareMathOperator{\Div}{Div}

\DeclareMathOperator{\Aut}{Aut}
\DeclareMathOperator{\Sym}{Sym}
\DeclareMathOperator{\dist}{dist}

\DeclareMathOperator{\pr}{pr}
\DeclareMathOperator{\rk}{rk}

\DeclareMathOperator{\Hom}{Hom}
\DeclareMathOperator{\DEG}{\textbf{deg}}

\DeclareMathOperator{\Sup}{Sup}
\DeclareMathOperator{\Min}{Min}
\DeclareMathOperator{\Max}{Max}
\DeclareMathOperator{\Inf}{Inf}

\DeclareMathOperator{\Spec}{Spec}
\DeclareMathOperator{\SPEC}{\textbf{Spec}}
\DeclareMathOperator{\Proj}{Proj}

\DeclareMathOperator{\Gal}{Gal}

\DeclareMathOperator{\NS}{NS}
\DeclareMathOperator{\Id}{Id}

\DeclareMathOperator{\Irr}{Irr}

\DeclareMathOperator{\ord}{ord}

\DeclareMathOperator{\Frac}{Frac}

\DeclareMathOperator{\inde}{ind}

\DeclareMathOperator{\TNS}{T_{\text{NS}}}

\let\leq\leqslant
\let\geq\geqslant
\newcommand{\sumu}[1]{\underset{#1}{\sum}}
\newcommand{\produ}[1]{\underset{#1}{\prod}}
\newcommand{\cupu}[1]{\underset{#1}{\cup}}
\newcommand{\capu}[1]{\underset{#1}{\cap}}

\newcommand{\bigcupu}[1]{\underset{#1}{\bigcup}}

\newcommand{\oplusu}[1]{\underset{#1}{\oplus}}
\newcommand{\bigoplusu}[1]{\underset{#1}{\bigoplus}}

\newcommand{\otimesu}[1]{\underset{#1}{\otimes}}

\newcommand{\Maxu}[1]{\underset{#1}{\Max}}
\newcommand{\Minu}[1]{\underset{#1}{\Min}}
\newcommand{\simu}[1]{\underset{#1}{\sim}}

\newcommand{\Supu}[1]{\underset{#1}{\Sup}}

\newcommand{\eps}{\varepsilon}
\newcommand{\vide}{\varnothing}

\newcommand{\eqdef}{\overset{\text{{\tiny{déf}}}}{=}}
\newcommand{\grsym}{\mathfrak{S}} 

\newcommand{\ind}{\mathbf{1}}

\newcommand{\wt}{\widetilde}

\newcommand{\sym}{\mathfrak{S}}
\newcommand{\acc}[2]{\left\langle #1\, ,\,#2 \right\rangle}

\newcommand{\abs}[1]{\left| #1 \right|}

\newcommand{\map}[4]{
\begin{array}{rcl}
#1 & \longto & #2 \\
#3 & \longmapsto & #4 
\end{array} 
}
 
\newcommand{\bigou}[2]{\underset{{#1}}{\ecO}\left(#2\right)} 
\newcommand{\inv}{\times} 
\newcommand{\ceff}{C_{\text{\textnormal{\tiny{eff}}}}}
\newcommand{\sep}[1]{#1^{\,\text{\textnormal{\tiny{sép}}}}}

\newcommand{\str}[1]{\scO_{#1}}
\newcommand{\cf}{{\it cf.}\  }
\newcommand{\ie}{{\it i.e.}\ }
\newcommand{\eg}{{\it e.g.}\ }
\newcommand{\opcit}{{\it op.cit.}}

\newcommand{\Ql}{\bQ_{\ell}}
\newtheorem{thm}{Theorem}[section]
\newtheorem{prop}[thm]{Proposition}
\newtheorem{lemma}[thm]{Lemma}

\newtheorem{question}[thm]{Question}
\newtheorem{questions}[thm]{Questions}

\newtheorem{hyps}[thm]{Hypotheses}

\newtheorem{heur}[thm]{Heuristic}
\theoremstyle{definition}
\newtheorem{defi}[thm]{Definition}

\newtheorem{notas}[thm]{Notations}

\theoremstyle{remark}

\newtheorem{ex}[thm]{Example}
\newtheorem{exs}[thm]{Examples}
\newtheorem{rem}[thm]{Remark}

\renewcommand{\eqdef}{\overset{\text{{\tiny{def}}}}{=}}
\renewcommand{\ceff}{\text{\textnormal{Eff}}}
\renewcommand{\sep}[1]{#1^{\,\text{\textnormal{\tiny{sep}}}}}
\DeclareMathOperator{\HOM}{\textbf{Mor}}
\usepackage{url}
\newcommand{\courbe}{\scC}
\newcommand{\diveff}[1] {\Div_{\text{eff}}(#1)}

\renewcommand{\TNS}[1]{T_{\text{NS(#1)}}}
\DeclareMathOperator{\Poinc}{Poinc}
\DeclareMathOperator{\Var}{Var}

\DeclareMathOperator{\Mot}{Mot}

\DeclareMathOperator{\HCR}{HCR}
\DeclareMathOperator{\Coll}{Coll}
\newcommand{\tors}{\scT}
\newcommand{\Gammaf}{\Gamma_{\!\!\bbf}}

\newcommand{\classe}[1]{\left[#1\right]}

\newcommand{\Homog}{\ecH}
\newcommand{\kpff} {K_0(\text{PFF}_k)}
 
\newcommand{\scan}[1]{s_{#1}} 
\newcommand{\Homogs} {\ecH^{\bullet}}

\renewcommand{\div}{\text{div}}

\newcommand{\antican}[1]{\omega_{#1}^{-1}}
\newcommand{\carac}[1]{\cX\left(#1\right)}
\newcommand{\class}[1]{\left[#1\right]}
\newcommand{\mk}{\cM_k}

\newcommand{\mhat} {\widehat{\cM_k}}

\newcommand{\mum} {\mu^{\text{mot}}}
\newcommand{\muxm} {\mu^{\text{mot}}_{X}}
\newcommand{\muxmt} {\wt{\mu^{\text{mot}}_{X}}}
\newcommand{\mux} {\mu_{X}}

\newcommand{\fil}{\ecF}

\renewcommand{\sym}[2]{\Sym^{#2}\left(#1\right)}
\newcommand{\chm}[1]{\Mot_{#1}}
\newcommand{\kochm}[1]{K_0\left(\chm{#1}\right)}
\newcommand{\kochmk}{\kochm{k}}

\newcommand{\symb}[1]{\left[#1\right]}

\newcommand{\chif}{\chi_{{}_{\text{\textnormal{\scriptsize form}}}}}

\newcommand{\chim}{\chi_{{}_{\text{\textnormal{\scriptsize mot}}}}}

\newcommand{\vark}{\Var_k}

\newcommand{\kovark}{K_0(\vark)}
\newcommand{\hmk}{\widehat{\cM_k}}

\newcommand{\ZHW}{Z_{\text{HW}}}
\newcommand{\ZHWm}{Z_{\text{HW},\text{\textnormal{\scriptsize{mot}}}}}

\newcommand{\phinm}{\Phi_{n,\text{\textnormal{\scriptsize{mot}}}}}
\newcommand{\psinm}[1]{\Psi_{n,\text{\textnormal{\scriptsize{mot}}}}(#1)}
\newcommand{\psin}[1]{\Psi_{n}(#1)}

\newcommand{\cone}{\scC}
\title{Asymptotic behaviour of rational curves}
\author{David Bourqui}
\address{IRMAR\\Université de Rennes 1\\Campus de Beaulieu\\35042
  Rennes Cedex\\France}
\email{david.bourqui@univ-rennes1.fr}
\begin{document}
\maketitle
\begin{abstract}
We investigate the asympotic behaviour of the moduli space of morphisms
from the rational curve to a given variety when the degree becomes large.
One of the crucial tools is the homogeneous coordinate ring of the variey. First we
explain in details what happens in the toric case. Then we examine the
general case.

This is a revised and slightly expanded version of notes for a course delivered during the
summer school on rational curves held in June 2010 at Institut Fourier, Grenoble. 
\end{abstract}

\section{Introduction}

\subsection{The problem}

There are several natural questions that one may raise about rational
curves on an algebraic variety $X$: is there a rational curve on $X$?
are there infinitely many? are there 'a lot' of rational
curves on $X$, that is to say, for example, do the rational curves on
$X$ cover an open dense subset? 
Here we will be concerned with the following question: 
given an algebraic variety $X$ 
possessing a lot of rational curves (for example, a rational variety)
is it possible to give a quantitative estimate of the number of
rational curves on it? We expect of course an answer slighly less vague
than: the number is infinite.

To give a more precise meaning to the above question,
let us assume from now that $X$ is projective
and fix a projective embedding $\iota\,:\,X\subset \bP^n$ (or, if you
prefer and which amounts almost to the same,  
an ample line bundle $\cL$ on $X$). Then being given a morphism
$\varphi\,:\,\bP^1\to X$ we 
define its degree (with respect to $\iota$) by
\begin{equation}
\deg_{\iota}(\varphi)\eqdef\deg((\iota\circ \varphi)^{\ast}\str{\bP^n}(1))
\end{equation}
(or $\deg_{\cL}(\varphi)\eqdef\deg(\varphi^{\ast} \cL)$). This is a nonnegative
integer. 
We know from the work of Grothendieck (\cf \cite{Gro:Hilb,Deb:higher}) that for any
nonnegative integer $d$ there exists a quasi-projective variety 
$\HOM(\bP^1,X,\iota,d)$ (or $\HOM(\bP^1,X,\cL,d)$) parametrizing the set of morphisms
$\bP^1\to X$ of $\iota$-degree $d$. 
Assuming that $X$ is defined over a field $k$, 
recall that this mean in particular that for every $k$-extension $L$ 
there is a natural $1$-to-$1$ correspondence between the
set of $L$-points of $\HOM(\bP^1,X,\iota,d)$ and the set of morphisms
$\bP^1_{L}\to X\times_k L$ of $\iota$-degree $d$.

Thus we obtain a sequence of
quasi-projective varieties $\{\HOM(\bP^1,X,\iota,d)\}_{d\in \bN}$ and we can raise the
(still rather vague) question: what can be said about the behaviour
of this sequence? Note that one way to understand this question is 
to 'specialize' the latter sequence to a numeric one, and consider
the behaviour of the specialization. There are several natural
examples of such numeric specializations. 
For instance we can consider
the sequence $\{\dim(\HOM(\bP^1,X,\iota,d))\}$ obtained by taking the dimension, or, 
if $k$ is a subfield of the field of complex numbers $\bC$, the
sequence $\{\chi_c(\HOM(\bP^1,X,\iota,d))\}$, where $\chi_c$ designates
the Euler-Poincar\'e characteristic with compact support; if $k$ is finite, one can also
look at the number of rational points, \ie the sequence $\{\# \HOM(\bP^1,X,\iota,d)(k)\}$. 

The study of the latter
sequence is a particular facet of a problem raised by
Manin and his collaborators in the late 1980's, namely the
understanding of the asymptotic behaviour of the number of rational
points of bounded height on varieties defined over a global field
(see \eg \cite{Pey:bordeaux,Pey:ober}).
The degree of the morphism $x\,:\,\bP^1\to X$ may be interpreted as the
logarithmic height of the point of $X(k(\bP^1))$ determined by $x$.

The sequence $\{\HOM(\bP^1,X,\iota,d)\}_{d\in \bN}$ depends on the
choice of $\iota$ (or $\cL$), nevertheless there is a simple way to get rid
of this dependency: let us introduce indeed the \termin{intrisic degree}
$\DEG(\varphi)$ of a morphism $\varphi\,:\,\bP^1\to X$ 
as 
the element of the dual $\NS(X)$ of the Neron-Severi group 
defined by 
\begin{equation}
\forall x\in \NS(X),\quad \acc{\DEG(\varphi)}{x}\eqdef\deg(\varphi^{\ast}x).
\end{equation}
Then for every $y\in \NS(X)^{\vee}$ there exists a quasi-projective variety 
$\HOM(\bP^1,X,y)$  parametrizing the set of morphisms
$\bP^1\to X$ of intrisic degree $y$. 
For every ample line bundle $\cL$ one
has a finite decomposition
\begin{equation}\label{eq:finite:decomposition}
\HOM(\bP^1,X,\cL,d)
=
\bigsqcup_{
\substack{y\in \NS(X)^{\vee}
\\
\acc{y}{\cL}=d
}
}
\HOM(\bP^1,X,y).
\end{equation} 
Now instead of considering the asympotic behaviour of the sequence
$\{\HOM(\bP^1,X,\cL,d)\}_{d\in \bN}$ for a particular choice of $\cL$,
one could study the behaviour of $\{\HOM(\bP^1,X,y)\}_{y\in \NS(X)^{\vee}}$
when '$y$ becomes large', the latter condition needing of course to be
more precisely stated.

Before explaining the expected behaviour of the
previously introduced sequences, we make a few remarks about possible
generalizations of the problem. None of them will be considered in
these notes.

First it is possible to raise analogous questions for varieties defined over $k(t)$, not only over $k$.
In more geometric words, 
instead of considering only constant families $X\times_k \bP^1_k\to \bP^1_k$,
one could look at families $\scX\to U$ where $U$ is a non empty open
subset of $\bP^1$, and rational sections of them.

Another natural generalization would be of course to replace $\bP^1$ by a
curve of higher genus. Let us stress that most of the results described
in these notes extend without much difficulty to the higher genus case.

It is also possible to consider higher-dimensional generalizations of
the problem, see \eg \cite{Wan}.

\subsection{Batyrev's heuristic}\label{subsec:Bat} 

I thank Ana-Maria Castravet for interesting remarks and comments about
the content of this section.
We retain all the notations introduced in the previous section.
When the base field $k$ is finite, Batyrev, Manin, their collaborators and
subsequent authors made precise predictions about the asymptotic
behaviour of the sequence $\{\# \HOM(\bP^1,X,\cL,d)(k)\}_{d\in \bN}$.
Let us explain how Batyrev links these predictions to some heuristic
insights on the asymptotic geometric properties of the varieties 
$\{\HOM(\bP^1,X,y)\}_{y\in \NS(X)^{\vee}}$ (over an arbitrary field $k$).
We will restrict ourselves to varieties $X$ for which the following
hypotheses hold:
\begin{hyps}\label{hyps:X}
$X$ is a smooth projective variety whose anticanonical bundle $\antican{X}$
is ample, in other words $X$ is a Fano variety. The geometric Picard
group of $X$ is free of finite rank and the geometric effective cone of $X$ is
generated by a finite number of class of effective divisors\footnote{
When the characteristic of $k$ is zero, it is true, that the hypotheses on the Picard
group and on the effective cone automatically holds for a Fano
variety, the latter property being highly non
trivial.
}.
\end{hyps}
Recall that the effective cone is the cone generated
by the classes of effective divisors.
We will be mostly interested in the case where $\cL=\antican{X}$.
For the sake of simplicity, we
will assume in this section that the class $\class{\antican{X}}$ has index one in $\Pic(X)$,
that is, $\Min\{d,\,\class{\antican{X}}\in d\,\Pic(X)\}=1$.
Then a naïve version of the predictions of Manin et al.
is the asymptotic
\begin{equation}\label{eq:predman}
\# \HOM(\bP^1,X,\antican{X},d)(k)\simu{d\to +\infty} c\,d^{\,\rk(\Pic(X))-1}\,(\# k)^d.
\end{equation}
Here and elsewhere $c$ will always designate a positive constant
(whose value may vary according to the places where it appears).
There is also a version when $\antican{X}$ is replaced by any line
bundle $\cL$ whose class lies in the interior of the effective cone
(in other words, a so-called \termin{big} line bundle),
about which we will say a few words below.

We call it a naïve prediction since it was clear from the very beginning that 
\eqref{eq:predman} could certainly not always hold because of 
the phenomenon of accumulating subvarieties. One of the simplest
relevant examples is the exceptional divisor of the
projective plane blown-up at one point. One can check that with
respect to the anticanonical degree 'most' of
the morphisms $x\,:\,\bP^1\to X$ factor through the exceptional
divisor\footnote{One geometric incarnation of the predominance of
  the morphisms factoring through the exceptional divisor  $E$ is the
  following fact:  the components of $\HOM(\bP^1,X,\antican{X},d)$ of
  maximal dimension contain only morphisms which factor through $E$; this can be easily seen using
the toric structure of $X$ and the results described in the next part
of this text.
}.
Thus one is led 
to consider in fact the
sequence $\{\HOM_{U}(\bP^1,X,\antican{X},d)\}$ where $U$ is a dense
open subset of $X$ and 
$\HOM_{U}(\bP^1,X,\antican{X},d)$
designates the open subvariety of 
$\HOM(\bP^1,X,\antican{X},d)$
parametrizing those morphisms $\bP^1\to X$ of anticanonical
degree $d$ which do not factor through $X\setminus U$.
Similarly, one defines $\HOM_U(\bP^1,X,y)$ for every $y\in \Pic(X)^{\vee}$.

Now the 'correct' prediction should be that \eqref{eq:predman} holds for 
$\# \HOM_U(\bP^1,X,\antican{X},d)(k)$
if $U$ is a sufficiently small open dense subset of $X$\footnote{
One may (and will) also consider the case where
  the anticanonical bundle of $X$ is not necessarily ample, but only big,
  namely only assumed to lie in the interior of the effective cone; in this case
  $\HOM(\bP^1,X,\omega_X^{-1},d)$ is not always a quasi-projective
  variety, but $\HOM_{U}(\bP^1,X,\omega_X^{-1},d)$ is for any sufficiently
  small dense open set $U$, thus the refined prediction still makes
  sense in this context.

One must also stress that even with this refinement, the prediction has
already been shown to fail for certain Fano varieties (see
\cite{BaTs:conicbundles}; the proof is over a number field but may be
adapted 
to our setting). Nevertheless, the class of Fano varieties
for which the refined prediction holds might be expected to be quite
large; in particular one might still hope that it holds for every del
Pezzo surface; especially in the arithmetic setting, the analogous refined
prediction was shown to be true for a large number of instances of
Fano variety;  here is a (far from exhaustive) list of related work in
the arithmetic setting:
\cite{BaTs:manconj},
\cite{dlB:duke},
\cite{dlBBD},
\cite{dlBBP:Cha},
\cite{BrFou},
\cite{CLT_vect3},
\cite{FMT},
\cite{ShaTaTs:comp:semisimple},
\cite{Spe:Manin},
\cite{Sal:tammes},
\cite{StTs:twist},
\cite{Thu:Schu},
\cite{Thu:flag},
\cite{Pey:duke}.
}.

In order to `explain geometrically' the prediction \eqref{eq:predman}, Batyrev makes use of the following
heuristic: 
\begin{heur}\label{heur:bat}
A geometrically irreducible $d$-dimensional variety defined over a
finite field $k$ has approximatively $(\# k)^d$ rational points defined over $k$.
\end{heur}
Of course there is the implicit assumption that the error terms
deriving from this approximation will be negligible regarding our asympotic
counting problem. This heuristic may be viewed as a very crude
estimate deduced from the Grothendieck-Lefschetz trace formula
expressing the number of $k$-points of $X$ as an alternating sum of traces
of the Frobenius acting on the cohomology groups. 
It is also used by Ellenberg and Venkatesh in a somewhat
different 
counting problem, see \cite{EllVen:count}.

The next crucial ingredient of Batyrev's heuristic is the classical
result from deformation theory of morphisms $\bP^1\to X$ saying that 
every component of $\HOM_{U}(\bP^1,X,y)$
has dimension greater than or equal to 
the
`expected dimension' 
$\dim(X)+\acc{y}{\antican{X}}$
(see \eg \cite[chapter 2]{Deb:higher}). 

Let us choose a finite family of effective divisors of $X$ whose classes in $\Pic(X)$
generate the effective cone of $X$ and let $U$ be a dense open set
of $X$ contained in the complement of the union of the support of these divisors. 
Then any morphism $\bP^1\to X$ which does not
factor through $X\setminus U$ has an intrisic degree $y$ such that
$\acc{y}{D}\geq 0$ for every effective class $D$, in other words $y$
belongs to the dual $\ceff(X)^{\vee}$ of the effective cone.

For any algebraic variety $Y$, let us denote by $\scN(Y)$ the number
of its geometrically irreducible components of dimension $\dim(Y)$.
Assuming that $\scN(\HOM_{U}(\bP^1,X,y))$ is asymptotically constant,
that the dimension of $\HOM_{U}(\bP^1,X,y)$
coincides with the expected dimension, 
and that the above heuristic applies, the number of $k$-points of
$\HOM_{U}(\bP^1,X,\antican{X},d)$ can be approximated by
\begin{equation}
c.
\# \{y\in \ceff(X)^{\vee}\cap \Pic(X)^{\vee}, 
\acc{y}{\antican{X}}=d\}
.
(\# k)^{d+\dim(X)}
\end{equation}
and we will see in section \ref{subsec:sens} that we have the asymptotic
\begin{equation}
\# \{
y\in \ceff(X)^{\vee}\cap \Pic(X)^{\vee},
\acc{y}{\antican{X}}=d\}\simu{d\to +\infty} c.d^{\,\rk(\Pic(X))-1}.
\end{equation}
Thus the above geometric assumptions on the varieties $\HOM_{U}(\bP^1,X,y)$
together with the adopted heuristic are 'compatible' with Manin's
prediction.
Thus, as pointed out by Batyrev, it should be interesting to study
the asympotic behaviour of $\HOM(\bP^1,X,y)$ in terms of dimension and
number of components.
For example, one may raise the following
questions.
\begin{question}\label{ques:bat}
\begin{enumerate}
\item
Does there exist a dense open subset $U$ of $X$ such that
for any $y\in \ceff(X)^{\vee}\cap \Pic(X)^{\vee}$ with
$\acc{y}{\antican{X}}$ large enough, the dimension of $\HOM_{U}(\bP^1,X,y)$
is equal to $\acc{y}{\antican{X}}+\dim(X)$?
\item
Does there exist a dense open subset $U$ of $X$ such that
for any $y\in \ceff(X)^{\vee}\cap \Pic(X)^{\vee}$ with
$\acc{y}{\antican{X}}$ large enough, $\scN(\HOM_{U}(\bP^1,X,y))$ is
constant?
\end{enumerate}
\end{question}

Note that the condition $\acc{y}{\antican{X}}$ large enough may be
replaced by the condition $\acc{y}{\cL}$ large enough for any big line
bundle $\cL$. Note also that if the answer to
the above questions is positive, for any big line bundle $\cL$ 
one should have following \ref{heur:bat} the
heuristic estimation
\begin{equation}\label{eq:estim:hom:cl}
\# \HOM_{U}(\bP^1,X,\cL,d)(k)
\underset{d\to +\infty}{\approx}
c.
\sum_{\substack{y\in \ceff(X)^{\vee}\cap \Pic(X)^{\vee}, 
\\
\acc{y}{\cL}=d}
} (\# k)^{\acc{y}{\antican{X}}+\dim(X)}.
\end{equation}
Moreover, as we will explain in section \ref{subsec:sens}, the RHS of \eqref{eq:estim:hom:cl} is 
equivalent as $d\to \infty$ to
\begin{equation}
c.d^{\,b(\cL)-1}.(\# k)^{\,a(\cL).d}
\end{equation}
where $a(\cL)\eqdef \Inf\{a\in \bR,\quad a.\cL-\antican{X}\in \ceff(X)\}$
and $b(\cL)$ is the codimension of the minimal face of $\ceff(X)$
containing $a.\cL-\antican{X}$; note that $a(\antican{X})=1$ and $b(\antican{X})=\rk(\Pic(X))$.
Thus one obtains an heuristic prediction for the asympotic behaviour of 
$\# \HOM_{U}(\bP^1,X,\cL,d)(k)$, which is in fact the general
version (\ie not limited to the case of the anticanonical sheaf) 
of the prediction of Manin et al. alluded to above.
\begin{rem}
I thank Ana-Maria Castravet for pointing out to me the following.
Let $M$ be an irreducible component of $\HOM(\bP^1,X)$. 
By \cite[4.10]{Deb:higher}, if the evaluation map
$\text{ev}\,:\,\bP^1\times M\to X$ is dominant then $M$ has the
expected dimension. Hence it is clear that for any degree
$y$, there is a dense open subset $U$ of $X$ such that every component
of $\HOM_U(\bP^1,X,y)$ has the expected dimension. But $U$ depends a
priori on $y$. 

Nevertheless, when $X$ has a dense open subset $U$ isomorphic to a
homogeneous variety, by using the
group action one sees immediatly that  
for every component $M$ of $\HOM_{U}(\bP^1,X,y)$ the evaluation morphism $\bP^1\times M\to X$ is
dominant. Hence the answer to the first question is affirmative in
this case, provided  that $\HOM_{U}(\bP^1,X,y)$ is non empty. 
More generally, this holds as soon as the subset $X^{\text{free}}$
defined in [\opcit, Proposition 4.14] is a dense open subset of $X$.

In particular, the first question has an affirmative answer for toric
varieties, provided that $\HOM_{U}(\bP^1,X,y)$ is non empty. 
In the next section, we will prove that 
the answers to \ref{ques:bat} are positive for toric varieties.
The proof does not rely on deformation theory but on the so-called homogeneous
coordinate ring. 
\end{rem}
\begin{rem}\label{rem:ans:ques:bat}
As indicated just before, later in these notes we will see that 
the answer to \ref{ques:bat} is affirmative
for toric varieties.
It is also known to be affirmative for some
other particular classes of varieties.

If $X$ is a homogeneous variety, then for every $y\in
\Pic(X)^{\vee}\cap \ceff(X)^{\vee}$, the moduli space $\HOM(\bP^1,X,y)$ is irreducible of
the expected dimension (independent works of Thomsen,
Kim-Pandharipande, and Perrin, see \cite{Tho:irr,KimPan:conn,Per:rathom}). 
This is also the case for general hypersurfaces of low degree (Harris,
Roth and Starr, see \cite{HarRotSta:rathyp}).

If $X$ is a blowing-up of a product of projective spaces and $U$ is
the complement of the exceptional divisors, it is shown 
by Kim, Lee and Oh in \cite{KimLeeOh:rat}
that, under suitable extra numerical assumptions on the degree $y$,  
$\HOM_U(\bP^1,X,y)$ is irreducible of the expected dimension.

But now let $X$ be the moduli space of stable rank two vector bundles on a
curve, with fixed determinant of degree $1$. This is a Fano variety
with Picard group of rank 1. Castravet's results in
\cite{Cas:rat} imply that the answer to the second part of question 
\ref{ques:bat} is negative for $X$ if the genus $g$ of the curve is even: 
for any sufficiently small open set $U$, $\HOM_U(\bP^1,X,d)$ has two
components if $g-1$ divides the degree $d$, and one otherwise. 
It is perhaps worth noting 
that the morphisms in the extra component appearing for degrees which
are multiple of $g-1$ 
are generically free, but not very free.

For the counter-example of Batyrev and Tschinkel to Manin's conjecture
\cite{BaTs:conicbundles}, which is a fibration in cubic surfaces, it is likely that the answer to
question \ref{ques:bat} is also negative.
\end{rem}

\subsection{A generating series: the degree zeta function}\label{subsec:genser}

In the previous sections, some predictions were formulated
about the asymptotic behaviour of some particular specializations of
the sequence $\{\HOM_{U}(\bP^1,X,\antican{X},d)\}$, 
namely the ones obtained by considering the dimension, the number of
geometrically irreducible components of maximal dimension and, in case
$k$ is finite, the number of $k$-points.
One may of course wonder whether there exist predictions for
other specializations, for instance the one deriving from the
topological Euler-Poincaré characteristic with compact support.
Concerning the latter, note that it has at least one common feature with the
specialization 'number of points over a finite field': they are both
examples of maps from the set of isomorphism class of algebraic
varieties to a commutative ring, which are additive in the sense that  
the relation $f(X)=f(X\setminus F)+f(F)$ holds whenever $X$ is a
variety and $F$ is a closed
subvariety of $X$, and satisfying moreover the relation
$f(X\times Y)=f(X)\,f(Y)$. 
We call such maps \termin{generalized Euler-Poincaré
characteristic}, abbreviated in GEPC in the following . 
We are naturally led to consider the universal
target ring for GEPC: as a group it is generated by symbols $\class{X}$
where $X$ is a variety modulo the relations $\class{X}=\class{Y}$
whenever $X\isom Y$ and $\class{X}=\class{F}+\class{X\setminus F}$
whenever $F$ is a closed subvariety of $X$ (the latter  are
often called \termin{scissors relations}). We endow it with a ring
structure by setting $\class{X}\class{Y}\eqdef\class{X\times Y}$. The
resulting ring is called the \termin{Grothendieck ring of varieties}\footnote{
This ring, already
considered by Grothendieck in the sixties (see \cite{corrGrotSer}),
has attracted a huge renewal of interest since Kontsevich used it
fifteen years ago as a key ingredient of his theory of motivic integration.
Its structure turns out to be quite difficult to understand.
Let us just cite a celebrated open question, which has connections
with the Zariski simplification problem: is the class of the affine
line in the Grothendieck ring a zero divisor? 
}
and denoted
by $K_0(\Var_k)$. Thus the datum of a GEPC with value in a commutative
ring $A$ is equivalent to the datum of a ring morphism
$K_0(\Var_k)\to A$.

For an algebraic variety $V$ we denote by $\class{V}$ its class in the
Grothendieck ring.
Now a way to handle  `all-in-one' every possible
specialization of the family 
$\{\HOM_{U}(\bP^1,X,y)\}_{y\in\ceff(X)^{\vee}}$ 
deriving from a GEPC is
to look at the family $\{\class{\HOM_{U}(\bP^1,X,y)}\}_{y\in \ceff(X)^{\vee}}$
which is thus a family with value in
the ring $K_0(\Var_k)$. Let us stress that although this is not obvious at first
sight, the knowledge of the class $\class{Y}$ of an algebraic variety
$Y$ allows also to
recover $\dim(Y)$ and $\scN(Y)$ (though $\dim$ and $\scN$ are
certainly not GEPC), see below.

A classical and useful tool when dealing with a sequence of complex numbers
$\{a_n\}$ is the associated generating series $\sum a_n\,t^n$.
Inded, it is often possible to get informations about the
analytic behaviour of the meromorphic function defined by the series, which in
turn yields by Tauberian theorems informations about the asymptotical
behaviour of the sequence itself.

We can try a similar approach in our context by forming, 
for every big line bundle $\cL$ and every sufficiently small dense open subset $U$,
the generating series
\begin{equation}
Z_U(X,\cL,t)\eqdef \sum_{d\geq 0}
\class{\HOM_{U}(\bP^1,X,\cL,d)}\,t^d
\in K_0(\Var_k)[[t]]
\end{equation}
whose coefficients lie in the Grothendieck ring of varieties.
We call it the  
\termin{geometric $\cL$-degree zeta function}. 
Applying a GEPC $\chi\,:\,K_0(\Var_k)\to A$ 
to its coefficients yields a specialized degree zeta function with
coefficients in $A$, denoted by $Z^{\chi}_U(X,\cL,t)$.
If $k$ is finite and the GEPC is $\#_k$ (that is, the morphism
`number of $k$-points'), we recover the generating
series associated to the counting of points of bounded
$\cL$-degree/height, which we will name the \termin{classical $\cL$-degree
zeta function}.

It will also be interesting to consider the intrisic degree zeta
function, which is a generating series keeping track of the
decomposition \eqref{eq:finite:decomposition}, from 
which  the various $\cL$-degree zeta functions may be recovered by specialization.
First we need some preliminaries about monoïd algebras.

Let $N$ be a $\bZ$-module of finite rank and $\cone$ be a rational polyedral cone 
of $N$, that is, $\cone$ is a convex cone in $N\otimes \bR$ generated by a finite
number of elements of $N$. We moreover assume that $\cone$ is strictly
convex, \ie $\cone \cap -\cone=\{0\}$.
Let $A$ be a commutative ring. Recall  how the $A$-algebra
$A[\cone \cap N]$ may be defined: it is the set of 
families $(a_y)\in A^{\cone\cap N}$ endowed with
the componentwise addition and the multiplication
defined by $(a.b)_y\eqdef \sum_{y_1+y_2=y}a_{y_1}b_{y_2}$; the point
is that 
since
$\cone$ is strictly convex there is only a finite number of
pair $(y_1,y_2)\in (\cone\cap N)^2$ such that $y_1+y_2=y$. 
The element $(a_y)$ will be written $\sum a_y\,t^y$.
If $x_0$ is an element of the interior of $\cone^{\vee}$ there is a
well defined morphism $\text{sp}_{x_0}\,:A[\cone \cap N]\to A[[t]]$
sending $t^y$ to $t^{\acc{y}{x_0}}$.
The point is that the level sets $\{y\in \cone\cap N,\,\acc{y}{x}=d\}_{d\in \bN}$
are finite.

Now we can define the intrisic geometric degree zeta function
\begin{equation}
Z_U(X,t)\eqdef \sum_{y\in \ceff(X)^{\vee}\cap \Pic(X)^{\vee}}
\class{\HOM_{U}(\bP^1,X,y)}\,t^y
\in K_0(\Var_k)[\ceff(X)^{\vee}\cap \Pic(X)^{\vee}].
\end{equation}
For a line bundle $\cL$ whose class is big, applying $\text{sp}_{\cL}$ to $Z_U(X,t)$, one recovers the  geometric
$\cL$-degree zeta funtion $Z_U(X,\cL,t)$. We can also specialize $Z_U(X,t)$ through various GEPC;
note that these specializations commute with $\text{sp}_{\cL}$.

\subsection{Some more examples of GEPC}

So far we have given only two examples of GEPC, the topological Euler Poincaré
characteristic with support compact and the number of $k$-rational points
when $k$ is a finite field. Both of them have of course a
cohomological flavour. It turns out that cohomology theories 
are a natural reservoir of GEPC.
Let us content ourselves to describe one particular example: fix a prime $\ell$ distinct from the characteristic of $k$,
and a separable closure $\sep{k}$ of $k$.
To every variety $X$ defined over $k$ are attached 
its $\ell$-adic cohomology groups, which form a sequence of $\Ql$-vector spaces
$\{H^n(X_{\sep{k}},\Ql)\}_{n\in \bN}$ equipped with a continuous
action of the absolute Galois group $\Gal(\sep{k}/k)$. 
If $X$ is proper, the $H^n(X_{\sep{k}},\Ql)$ are finite dimensional
 and vanish for $n>2\,\dim(X)$.
When $X$ is proper and smooth, one defines its 
\termin{$\ell$-adic Poincaré polynomial} by
\begin{equation}
\Poinc_{\ell}(X)\eqdef\sum_{n\geq 0}\dim(H^n(X_{\sep{k}},\Ql))t^n.
\end{equation}
One can show that there is a ring morphism $\Poinc_{\ell}\,:\,K_0(\Var_k)\to \bZ[t]$
extending $\Poinc_{\ell}$ (in characteristic zero one may use the
fact, proven by F.Bittner, that the class
of smooth projective varieties, modulo the relations derived from
blowing up along a smooth subvariety, form a presentation of
$K_0(\Var_k)$; when $k$ is finitely generated, one uses the weight
filtration on the $\ell$-adic cohomology groups with compact support;
in the general case one reduces to the latter by a limiting process).
For every algebraic variety $X$, we have
$\deg(\Poinc_{\ell}(\class{X})=2\,\dim(X)$ thus the knowledge of
$\Poinc_{\ell}$ allows to recover the dimension.
In case $k$ is a subfield of $\bC$, comparison theorems between
$\ell$-adic cohomology and Betti cohomology show that the topological
Euler-Poincaré characteristic factors through $\Poinc_{\ell}$.

In fact one can even define a refined $\ell$-adic Poincaré polynomial 
$\Poinc^{\text{ref}}_{\ell}\,:\,K_0(\Var_k)\to K_0(\Gal(\sep{k}/k)-\Ql)$
which satisfies for $X$ smooth and proper the relation
\begin{equation}
\Poinc^{\text{ref}}_{\ell}(X)=\sum_{n\geq 0}\class{H^n(X_{\sep{k}},\Ql)}\,t^n.
\end{equation}
Here $K_0(\Gal(\sep{k}/k)-\Ql)$ stands for the Grothendieck ring of the category
of finite dimensional $\Ql$-vector spaces equipped with a
continuous action of the absolute Galois group.
If $k$ is finite,  one can recover from this refined Poincaré
polynomial the GEPC $\#_k$ by
applying the trace of the Frobenius and evaluating at $t=-1$.
In general, one can recover the number of geometrically irreducible components
of maximal dimension from the refined Poincaré polynomial: indeed, for any
algebraic variety $X$, $\scN(X)$ is the dimension of
$(a_{2\,\dim(X)})^{\Gal(\sep{k}/k)}$, where $a_{2\,\dim(X)}$ is the
leading coefficient of $\Poinc^{\text{ref}}_{\ell}(X)$.

If the characteristic of $k$ is zero, there exists by the work of
Gillet, Soul\'e et al. a universal `cohomological' GEPC $\chi_{mot}$ whose target is
the Grothendieck ring of the category of pure motives.
Recalling the construction and the basic properties of this category is beyond the
scope of these notes (see \cite{And:mot} for  a nice introduction). Let us simply stress that one of the guiding
lines of the theory of motives is that it should be a kind of universal
cohomological theory for algebraic varieties, which would allow to
recover any classical cohomological theory by
specialization. Unfortunately, later in these notes, we will be obliged
to work with the specialization $Z^{\chi_{mot}}_U(X,t)$ rather than with the initial
geometric degree function. Though this is certainly inacurrate in many
senses, the reader unaware of motives may think of the Grothendieck
ring of motives as if it was the Grothendieck ring of varieties
(localized at the class of the affine line, see below).

\subsection{Completion of the Grothendieck ring of varieties}

We will  now define a topology on (a localization of) the Grothendieck 
ring of algebraic varieties. 
This seems
necessary if we want to
talk about the `analytic behaviour' of the geometric zeta function.
The topology we will consider is the one proposed by Kontsevich for
his construction of motivic integration. 
We denote by $\bL$ the class of the affine line $\bA^1$ in the
Grothendieck ring of varieties\footnote{The letter L stands for
  Lefschetz. This is because the image of $\class{\bA^1}$ by the
  morphism $\chi_{\text{mot}}$ alluded to above coincides with the
  class of the so-called Lefschetz motive.}.
We denote by $\mk$ the localization of $\kovark$ 
with respect to $\bL$ (recall that it is not known whether the
localization morphism $\kovark\to \mk$ is injective).

Intuitively, the idea behind the definition given below might be understood as follows: if $k$ is finite with cardinality
$q$, the image of $\bL$ by the `number of $k$-points' morphism is $q$;
since the series $\sum_{n\geq 0} q^{-n}$ converges, we would like by
analogy the series $\sum_{n\geq 0} \bL^{-n}$ to be convergent too.
Let us stress that this is really a loose analogy here, since the
`number of $k$-points' morphism will not be continuous with the respect to
the topology we will define, and thus will not extend to the completion
of the Grothendieck ring with respect to this topology.

We filter the elements in  $\mk$ by their `virtual
dimension': 
for $n\in \bZ$, let $\fil^{n}\mk$ be the
 subgroup of $\mk$ generated by those elements which may be written as 
$\bL^{-i}\symb{X}$, where $i\in \bZ$ and $X$ is a $k$-variety satisfying $i-\dim(X)
\geq n$ (elements whose virtual dimension is less than or equal to $-n$). Thus
$\fil^{\bullet}$ is a decreasing filtration, and $\cupu{n\in \bZ} \fil^{n}=\mk$.

Let $\hmk$ be the completion of $\mk$ with respect to the topology
defined by the dimension filtration (that is, the topology for which
$\{\fil^n\mk\}$ is a fondamental system of neighboroods of the
origin). In other words we
have 
\begin{equation}
\hmk=\underset{\longlto}{\lim}\quad \mk/\fil^n \mk.
\end{equation} 
Thus an element of $\hmk$ may
be represented as an element $(x_n)\in \prod_{n\in \bZ} \mk/\fil^n \mk$ such
that for every integers $n$ and $m$ satisfying $m\geq n$ we have
$\pi_{m}^n(x_m)=x_n$, where $\pi_{m}^n$ is the natural projection
$\mk/\fil^m\mk\to \mk/\fil^n\mk$. We have the natural
completion morphism $\mk\to \hmk$ and a natural filtration on $\hmk$ coming
from the filtation $\fil^{\bullet}$.

A priori $\hmk$ inherits only the group structure of the ring $\mk$. 
Now we define a product.
Let $x=(x_n)$ and $y=(y_n)$ be two
elements in $\hmk$ and $M$ be an integer such that $x,y\in \fil^M\mk$
(that is, we have $x_n=y_n=0$ for $n\leq M$). Let $n$ be an integer
and $\wt{x_{n-M}}$, $\wt{y_{n-M}}$ be liftings of $x_{n-M}$ and
$y_{n-M}$ to $\mk$ respectively. Define $(x.y)_n$ as the class of $\wt{x_{n-M}}.\wt{y_{n-M}}$
modulo $\fil^n\mk$. The inclusions $\fil^n\mk.\fil^m\mk\subset \fil^{n+m}\mk$
show that this does not depend on the made choices and that this
endows $\hmk$ with a ring structure compatible with the completion morphism.

For an element $x\in \mk$ (respectively $x\in \hmk$), define 
\begin{equation}
\dim(x)=-\frac{1}{2}\Sup\{n,\quad x\in \fil^n\mk\}.
\end{equation}
Using the $\ell$-adic Poincaré  polynomial, one may check that if $X$
is a $k$-variety then we have indeed $\dim(\class{X})=\dim(X)$. 
Note that for every integer $n\in \bZ$ one has $\dim(\bL^{n})=n$.
One may wonder whether there are nonzero elements in $\mk$ with
dimension $-\infty$, in other words whether the completion morphism is
injective: this is an open question.

Note that a series $\sum_{n\geq N} x_n$ whose terms belong to $\hmk$
converges in $\hmk$ if and only if
$\dim(x_n)$ goes to $-\infty$. For example $\sum_{n\geq 0}\bL^{n}$ 
converges, and one checks that its limit is the inverse of $1-\bL$ in $\hmk$.

Note also that if $k$ is finite with cardinality $q$ the morphism $\#_k \,:\,\mk\to \bZ[q^{-1}]\subset \bR$ is not
continuous when we endow $\bR$ with the usual topology; 
for example, for {\em any} sequence of integers $\{c_n\}$, the sequence
$c_n\,\bL^{-n}$ converges to zero with respect to our topology. 
Thus there is no hope to extend $\#_k$ to a morphism $\hmk\to \bR$.

By contrast, the morphism $\Poinc_{\ell}\,:\,\mk\to \bZ[t,t^{-1}]$ is continuous when
$\bZ[t,t^{-1}]$ is endowed with the topology associated to the filtration
by the degree, and thus extends  to a morphism $\hmk\to \bZ[[t^{-1}]]_{(t)}$.
\begin{rem}\label{rem:L:trans}
Let $A$ be $\kovark$, $\mk$ or $\hmk$.
Using the Poincaré polynomial, one sees 
that the morphism $\bZ\to A$ mapping $1$ to $1$ is an injection
and that $\bL\in A$ is transcendent over $\bZ$. 
\end{rem}

\subsection{Some  questions about the analytic behaviour of
  the degree zeta function}\label{subsec:sens}
Let $N$ be a $\bZ$-module of finite rank $d$ and $\cone$ be a rational
polyedral strictly convex cone of $N$.
We set
\begin{equation}
Z(N,\cone,t)\eqdef\sum_{y\in N\cap \cone}t^y \in \bZ[\cone\cap N].
\end{equation}
When $\cone$ is regular, that is, generated by a subset $\{y_1,\dots,y_d\}$ of a basis of $N$,
a straightforward computation shows the relation
\begin{equation}\label{eq:zcone:regular}
Z(N,\cone,t)=\prod_{i=1}^d\frac{1}{1-t^{\,y_i}}.
\end{equation}
In general, it is known that $\cone$ can be
written as an `almost disjoint' union of regular cones
(more precisely as the support of a regular fan, see below)
thus $Z(N,\cone,t)$ will equal a finite sum
of expressions of the type \eqref{eq:zcone:regular}. 
For any element $x\in N^{\vee}$ lying in the
relative interior of $\cone^{\vee}$, 
we have
\begin{equation}
\text{sp}_xZ(N,\cone,t)=\sum_{y\in N\cap \cone}t^{\acc{y}{x}} \in \bZ[[t]].
\end{equation}
From the above decomposition, we deduce that $\text{sp}_xZ(N,\cone)$ is a rational
function of $t$, with a pole of order $\dim(\cone)$ at $t=1$, and whose
other poles are roots of unity. 
For $x$ in $N^{\vee}$, define the index of $x$ in $N^{\vee}$ by
\begin{equation}
\inde(N^{\vee},x)\eqdef\Max\{d\in \bN,\,x\in d\,N^{\vee}\}.
\end{equation}
If  $\inde_{N^{\vee}}(x)=1$, the order of any pole of $\text{sp}_xZ(N,\cone)$ 
distinct from $1$ is less than $\dim(\cone)$. In general, a
similar statement holds for the series $\text{sp}_xZ(N,\cone)\left(t^{\frac{1}{\inde(N^{\vee},x)}}\right)$.

Letting $\alpha(N,\cone,x)$ be the leading term of $\text{sp}_xZ(N,\cone)$ at the
critical point $t=1$, 
we obtain by
Cauchy estimates 
\begin{equation}\label{eq:first:asymp}
\#\{y\in N\cap \cone,\quad
\acc{y}{x}=\,
\inde(N^{\vee},x)
\,d\}
\simu{d\to +\infty} 
\alpha(N,\cone,x)\,
[
\inde(N^{\vee},x)\,
d]^{\dim(\cone)-1}.
\end{equation}
If $x_0$ is an element of $N$, one may also consider
\begin{equation}
Z(N,\cone,x_0,t)\eqdef\sum_{y\in N\cap \cone}\rho^{\acc{y}{x_0}}t^y \in \bZ[\rho][\cone\cap N].
\end{equation}
Assume that $x_0$ lies in the interior of $\cone^{\vee}$. 
For every
element $x$ lying in the interior of $\cone^{\vee}$, 
let $a(x_0,x)\eqdef \Inf\{a,\quad a.x-x_0\in \cone^{\vee}\}$ and $b(x_0,x)$ be
the codimension of the minimal face of $\cone^{\vee}$ containing $a(x_0,x).x-x_0$.
Using similar arguments as above, one checks,
letting $\alpha(N,\cone,x_0,x)$ be the leading term of $\text{sp}_xZ(N,\cone,x_0)$ at the
critical point $t=\rho^{-a(x_0,x)}$, that the following generalisation
of \eqref{eq:first:asymp} holds:
\begin{equation}
\sum_{\substack{
y\in N\cap \cone,
\\
\acc{y}{x}=\,\inde(N^{\vee},x)\,d
}}
\rho^{\acc{y}{x_0-a(x_0,x).x}}
\simu{d\to +\infty} 
\alpha(N,\cone,x_0,x)\,[\inde(N^{\vee},x)\,d]^{b(x,x_0)-1}.
\end{equation}

\begin{defi}\label{defi:controlled}
Let $Z(t)\in \bC[[t]]$, $\rho$ a positive real number and $d$ a nonnegative integer.
We say that $Z(t)$ is \termin{strongly $(\rho,d)$ controlled} if $Z(t)$
converges absolutely in the open disc $\abs{t}<\rho$ and the associated
holomorphic function extends to a meromorphic funtion on the open disc $\abs{t}<\rho+\eps$
for a certain $\eps>0$, whose poles on the circle $\abs{t}=\rho$
have order bounded by $d$. We say that $Z(t)$ is
\termin{$(\rho,d)$-controlled} if it is bounded by a strongly
$(\rho,d)$-controlled series (we say that $\sum a_n t^n$ is bounded by
$\sum b_n t^n$ if $\abs{a_n}\leq \abs{b_n}$ for all $n$). 
\end{defi}
Note that by Cauchy estimates, if $d\geq 1$ then
$\sum a_nt^n$ is $(\rho,d)$-controlled if and only if
the sequence $a_n.n^{1-d}.\rho^{n}$ is bounded.

We are now in position to state a question about the analytic
behaviour of the classical degree zeta function. For the sake of
simplicity, we will restrict ourselves to the case of the
anticanonical degree. The following may be seen as a
version of a refinement by Peyre of a question raised by Manin.

\begin{question}\label{ques:fin}
Let $k$ be a finite field of cardinality $q$. Let $X$ be a $k$-variety 
satisfying hypotheses \ref{hyps:X}.
Does there exist a positive real number $c$ and a dense open subset
$U$ such that the series
\begin{equation}\label{eq:diff}
\#_k\text{sp}_{\antican{X}}Z_U(X,t)
-c.\text{sp}_{\antican{X}}Z(\Pic(X)^{\vee},\ceff(X)^{\vee})(q\,t)
\end{equation}
is $(q^{-1},\rk(\Pic(X))-1)$-controlled (respectively strongly 
$(q^{-1},\rk(\Pic(X))-1)$-controlled)?
\end{question}
Of course the question may be refined by asking whether the result
holds for every sufficiently small dense open subset.

Note that an affirmative answer yields the estimate 
\begin{multline}
\# \HOM_{U}(\bP^1,X,\antican{X},\inde(\Pic(X),\antican{X})\,d)(k)
\\
\simu{d\to +\infty}
c.\alpha(\Pic(X)^{\vee},\ceff(X)^{\vee},\class{\antican{X}})\,
[\inde(\Pic(X),\antican{X})\,d]^{\rk(\Pic(X))-1}\,q^{\inde(\Pic(X),\antican{X})\,d}.
\end{multline}
Of course, 
in case \eqref{eq:diff}
is strongly
$(q^{-1},\rk(\Pic(X)-1)$-controlled, we get a more
precise asymptotic expansion.

Let us add that 
there exists a precise description of the expected value of the
constant $c$ (see at the end of section \ref{subsec:dzf:tv}).

Now we turn to the search for a geometric analog of the previous question.
We adopt the following definition.
\begin{defi}\label{defi:cont:mot}
Let $Z(t)\in \hmk[[t]]$, $r\in \bZ$ and $d$ a nonnegative integer.
We say that $Z(t)$ is  \termin{$(\bL^{-r},d)$ controlled} if it
may be written as a finite sum $\sum_{i\in I}Z_i(t)$ such that for
every $i\in I$, there exist $d_i\leq d$ and $d_i$ positive integers
$a_{i,1},\dots,a_{i,d_i}$ such that the series
\begin{equation}
\prod_{1\leq e\leq d_i} (1-\bL^{\,r\,a_{i,e}}\,t^{a_{i,e}})\,Z_i(t)
\end{equation}
converges at $t=\bL^{-k}$.
\end{defi}
This definition is to be thought as a loose analog of definition \ref{defi:controlled}. 
\begin{question}\label{ques:mot}
Let $k$ be a field and $X$ be a $k$-variety 
satisfying hypotheses \ref{hyps:X}.
Does there exist
a nonzero element $c\in \hmk$
and a dense open subset
$U$ such that the series
\begin{equation}\label{eq:series}
\text{sp}_{\antican{X}}.Z_U(X,t)-
c.
\text{sp}_{\antican{X}}Z(\Pic(X)^{\vee},\ceff(X)^{\vee})(\bL\,t)
\end{equation}
is $(\bL^{-1},\rk(\Pic(X))-1)$-controlled?

Does the constant $c$ have an interpretation analogous to the one in the classical case?
\end{question}
Regarding tauberian statements, it is worth noting that unfortunately the situation is not as
comfortable as in the case of a finite field. One might want for
example to deduce from an affirmative answer to the latter question
informations about the asymptotic behaviour of the dimension and the
number of irreducible components of maximal dimension of
$\HOM_{d,U}(\bP^1,X)$, but one may check that the only statement one
is able to derive is the less precise inequality
\begin{equation}
\overline{\lim}\,\,\frac{\dim(\HOM_{U}(\bP^1,X,\antican{X},d))}{d}\leq 1.
\end{equation} 
When studying the case of a toric variety $X$, we
will in fact be able to answer affirmatively to questions
\ref{ques:bat} long before we are in position to do so for question \ref{ques:mot}.

To partially solve this issue, one may consider a variant of
question \ref{ques:mot}, suggested by Peyre and keeping track of the
intrisic degree. Indeed, following Batyrev's heuristic, in case $k$ is
finite the quantity
\begin{equation}
\# \HOM_{U}(\bP^1,X,y)\,q^{-\acc{y}{\antican{X}}}
\end{equation}
should have a positive limit when $\acc{y}{\antican{X}}\to +\infty$.
Peyre pointed out that it seemed sensible to raise the following questions.
\begin{questions}\label{ques:fin:bis}
\begin{enumerate}
\item
Let $k$ be a finite field of cardinality $q$. Let $X$ be a $k$-variety 
satisfying hypotheses \ref{hyps:X}.
Does there exist a positive real number $c$ and a dense open subset
$U$ such that 
\begin{equation}\label{eq:diff:bis}
\lim_{
\substack{
~\\
y\in \Pic(X)^{\vee}\cap \ceff(X)^{\vee}
\\
\dist(y,\partial \ceff(X)^{\vee})\to +\infty
}
}
\# \HOM_{U}(\bP^1,X,y)(k)\,q^{-\acc{y}{\antican{X}}}
=c\quad ?
\end{equation}
\item
Let $k$ be any field and $X$ be a $k$-variety 
satisfying hypotheses \ref{hyps:X}.
Does there exist a nonzero element $c\in \hmk$
and a dense open subset
$U$ such that 
\begin{equation}\label{eq:series:bis}
\lim_{
\substack{
~
\\
y\in \Pic(X)^{\vee}\cap \ceff(X)^{\vee}
\\
\dist(y,\partial \ceff(X)^{\vee})\to +\infty
}
}
\class{\HOM_{U}(\bP^1,X,y)}\,\bL^{-\acc{y}{\antican{X}}}=c\quad ?
\end{equation}
\end{enumerate}
\end{questions}
\begin{rem}
If \eqref{eq:series:bis} holds, one can check using the
Poincaré polynomial that the answer to \eqref{ques:bat} is positive
provided that $\dist(y,\partial \ceff(X)^{\vee})\to +\infty$.
The latter condition if of course stronger than the mere condition 
$\acc{y}{\antican{X}}\to +\infty$.

We will see in the case of toric varieties why the answers to 
questions \ref{ques:fin:bis} can not be expected in general to be affirmative
under the mere assumption $\acc{y}{\antican{X}}\to +\infty$. 

Anyway, Castravet's results (see remark
\ref{rem:ans:ques:bat}) imply that 
\eqref{eq:series:bis} can not hold in general.
\end{rem}
\begin{rem}
Ellenberg raised the question whether the existence of the limits
in \eqref{eq:diff:bis} or \eqref{eq:series:bis} could be explained by a
phenomenon of homological stability akin to the one established in 
\cite{EVW:Hurwitz}.
\end{rem}

\section{The case of toric varieties}

In this section we explain how one can deal with the previously introduced problem in the case of toric varieties.
The crucial tool will be the so-called homogeneous
coordinate ring introduced by Cox in the toric case and, as we will
see in the last section, generalized by subsequent authors to other varieties.

But first let us take a little moment to explain very informally to
what extent the homogeneous coordinate ring will be helpful.
Basically, what can be done if one wants to describe the variety
$
\HOM(\bP^1,X,\cL,d)
$
for $\cL$ a very ample line bundle and $d$ an integer~? 
One of the most natural approach is probably to fix an embedding $\iota\,:\,X\hookrightarrow \bP^n$ 
such that $\iota^{\ast}\str{\bP^n}(1)\isom \cL$, thus identifying $X$
with a closed subvariety of $\bP^n$ described by homogeneous equations $\{F_1=\dots=F_r=0\}$.
The points of $\HOM(\bP^1,X,\cL,d)$ are then in one-to-one correspondence
with the $(n+1)$-uples $(P_0,\dots,P_n)$ of homogeneous polynomials
with two variables and degree $d$ (modulo multiplication by a nonzero
scalar), having no common root in an
algebraic closure of $k$, and satisfying the equations 
\begin{equation}\label{eq:naive:equations}
F_1(P_0,\dots,P_n)=\dots=F_r(P_0,\dots,P_{n})=0.
\end{equation}
This allows to  describe explicitely $\HOM(\bP^1,X,\cL,d)$ as a locally
closed subset of $\bP^{(d+1)(n+1)-1}$.

This elementary approach has at least two drawbacks:
\begin{itemize}
\item
even if it turns out to be fruitful for a particular choice of $\cL$,
for another choice of line bundle the equations of the embedding will change and
everything will have to be redone
\item
the equations \eqref{eq:naive:equations} will be a priori rather
complicated, and thus probably  not very helpful to understand the geometry of 
$\HOM(\bP^1,X,\cL,d)$; in particular, one hardly imagine how the decomposition \eqref{eq:finite:decomposition}
with respect to the intrisic degree could be recovered naturally
from these equations. 
\end{itemize}
The homogeneous coordinate ring of $X$ will, in some sense, solve
completely the first issue and the second part of the second issue.
The loose idea is that this ring will contain all the informations
about every possible embeddings of $X$ in a projective space, which 
 will allow us, roughly speaking, to treat all of them simultaneously.

On the other hand, the first part of the second issue will in general
still cause some trouble. We will still have to face equations of
the shape \eqref{eq:naive:equations}, which may be rather hard to deal
with (though now these equations are 'independent of the choice of the
line bundle'). Nevertheless, we will see that in the toric case, the
situation simplifies dramatically, since there are 'no equation' for
the homogeneous coordinate ring. 

\subsection{Toric geometry}\label{subsec:tor:gemo}
Here we recall some basic facts about toric geometry. 
Proofs will be omitted or very roughly sketched, and are
easily accessible in the classical references on the topic
(\cite{Ful:toric,Oda:conv,Ew}).

A (split) algebraic torus is a group variety isomorphic to a product
of copies of the multiplicativ groupe $\bG_m$. 
A (split) toric variety is a normal equivariant (partial) compactification of
an algebraic torus. In other words, it is a normal algebraic variety
endowed with an algebraic action of an algebraic torus $T$ and possessing
an open dense subset $U$ isomorphic to $T$ in such a way that the action
of $T$ on $U$ identifies with the action of $T$ on itself by translations.

\begin{exs}\label{exs:toric:vtes}
$\bA^n$ on which $\bG_m^n$ acts diagonally, $\bP^n$ on which $\bG_m^n$
acts by 
\begin{equation}
(\lambda_1,\dots,\lambda_n)(x_0:\dots:x_n)=(x_0:\lambda_1\,x_1:\dots:\lambda_n\,x_n).
\end{equation}
Now blow up $\bP^2$ at a $\bG_m^2$-invariant point, for example
$(0:0:1)$, yielding a variety $X$; then the $\bG_m^2$-action on
$\bP^2$ extends to $X$, making $X$ a compactification of $\bG_m^2$.
\end{exs}

\begin{rem}
A non necessarily split algebraic torus is a group variety which
becomes isomorphic to a split torus over an algebraic closure of the
base field. 
Though the
case of nonsplit toric varieties, that is, compactifications of
non necessarily split tori, certainly deserves consideration in the
context of our problem, we will stick in these notes to the case
of split toric varieties.
\end{rem}

Let $T\isom \bG_m^d$ be a split torus of dimension $d$. 
The group $\carac{T}$ of algebraic characters of $T$, that is, of
algebraic group morphism $T\to \bG_m$, is a free module of finite rank
$d$. Note that the natural morphism $\carac{T}\to k[T]^{\inv}/k^{\inv}$ is
an isomorphism.

\begin{prop}
Let $X$ be a smooth projective equivariant compactification of $T$,
and $U$ its open orbit. Then $X\setminus U$ is the union of a finite number $\{D_i\}_{i\in I}$
of $T$-invariant irreducible divisors defined over $k$. We call the $D_i$'s the \termin{boundary divisors}.
The map $D_i\mapsto \str{X}(D_i)$ induces an exact sequence
\begin{equation}\label{eq:exsq}
0\to k[T]^{\inv}/k^{\inv}\to \oplusu{i\in I}\bZ D_i\to \Pic(X)\to 0.
\end{equation}
\end{prop}
\begin{proof}(sketch)
A key tool is Sumihiro's lemma (\cite{Sum:equiv,Sum:equiv.II}), which tells us that since the
$T$-variety $X$ is normal, it may be covered by $T$-invariant affine open subsets. From this
one easily concludes that $X\setminus U$ is of pure codimension $1$
(note that each affine open subset of the covering must contain $U$,
which is itself affine).
Let $\{D_i\}_{i\in I}$ be the finite set of geometric irreducible
components of $X\setminus U$.
Since $T$ acts on $\{D_i\}_{i\in I}$ and is connected,
each $D_i$ is $T$-invariant. It induces a valuation 
\begin{equation}
v_{D_i}\,:\,\bar{k}[T]=\oplusu{m\in \carac{T}}\bar{k}.\chi^m\to \bN
\end{equation}
which, by $T$-invariance, satisfies
\begin{equation}\label{eq:formula:vdi}
v_{D_i}(\sum_{m\in \carac{T}}a_m\,\chi^m)=\Minu{a_m\neq 0}v_{D_i}(\chi^m)
\end{equation}
Hence $D_i$ is defined over $k$.

Since $k[T]\isom k[t_1,t_1^{-1},\dots,t_d,t_d^{-1}]$ is a UFD, the
Picard group of $U\isom T$ is trivial, hence \eqref{eq:exsq}.
\end{proof}
By dualizing the exact sequence \eqref{eq:exsq}, we obtain
\begin{equation}\label{eq:exsq:dual}
0\to \Pic(X)^{\vee}\to \oplusu{i\in I}\bZ D_i^{\vee}\to \carac{T}^{\vee}\to 0.
\end{equation}
Let $\rho_i$ denote the image of $D_i^{\vee}$ in
$\carac{T}^{\vee}$.
Let $\Sigma_X$ be the set of cones
generated by $\{\rho_i\}_{i\in J}$ 
for those $J\subset I$ such that $\cap_{i\in J}D_i\neq \vide$. 
Then $\Sigma_X$ is a fan of $\carac{T}^{\vee}$, in the following sense:
\begin{defi}
A \termin{fan} $\Sigma$ of $\carac{T}^{\vee}$ is a finite family
$\{\sigma\}_{\sigma\in \Sigma}$ of polyedral rational strictly convex cones (see 
section \ref{subsec:genser})
 of $\carac{T}^{\vee}\otimes \bR$ such that:
\begin{enumerate}
\item whenever $\sigma$ and $\sigma'$ belong to $\Sigma$, $\sigma\cap
  \sigma'$ is a face of $\sigma$ and $\sigma'$
\item whenever $\sigma$ belongs to $\Sigma$, every face of $\sigma$ belongs to $\Sigma$
\end{enumerate}
\end{defi}

One of the most striking feature of the theory of toric varieties
is that the fan $\Sigma_X$ defined above allows to recover $X$ (thus
the geometry of $X$ may be described in terms of combinatorial objects
coming from convex geometry). In fact, starting from any fan in $\carac{T}^{\vee}$ one may
construct a normal (partial) compactification of $T$ by glueing
together the affine $T$-varieties $V_{\sigma}\eqdef \Spec(k[\sigma^{\vee}\cap
\carac{T}])$ along the $V_{\sigma\cap \sigma'}$, and one can show that every normal compactification of $T$ is obtained in
this way.

In our case, since we assumed $X$ to be projective, the fan
$\Sigma_X$ is complete (that is, the union of its cone is the whole
space) and since it was assumed to be smooth, the cones of $\Sigma_X$
are regular, that is, each one of them is generated by a part of a $\bZ$-basis
of $\carac{T}^{\vee}$. Note that this implies that $X$ is covered by
affine varieties isomorphic to affine spaces. In fact in case $\sigma$
is a maximal cone of the fan, $V_{\sigma}$ is naturally isomorphic to $\bA^{I_{\sigma}}$
where $I_{\sigma}=\{i\in I,\quad \rho_i\in \sigma\}$. And for any
$i\in I$, a local equation of the divisor $D_i$ in $V_{\sigma}$ is
$x_i$ if $i\in I_{\sigma}$ and $1$ otherwise.
\begin{exs}
In case $X=\bP^n$, equipped with the toric structure described in 
\eqref{exs:toric:vtes}, the boundary divisors are the coordinate
hyperplanes $D_i=\{x_i=0\}$ for $i=0,\dots,n$; the rays $\{\rho_i\}_{i=0,\dots,n-1}$
form a $\bZ$-basis of $\carac{T}^{\vee}$ and we have
$\rho_{n}=-\sum_{i=0}^{n-1}\rho_i$; the maximal cones of $\Sigma_X$
are those cones generated by $\{\rho_i\}_{i\in J}$ for $J$ running
over the subsets of $\{0,\dots,n\}$
with $n$ elements.

In case $X$ is $\bP^2$ blown-up at $(0:0:1)$, the boundary divisors are the
exceptional divisor $E$ and the strict transforms $D_0$, $D_1$, $D_2$
of the coordinate hyperplanes; the rays $\{\rho_0,\rho_1\}$ form 
a $\bZ$-basis of $\carac{T}^{\vee}$ and we have 
$\rho_E=\rho_0+\rho_1$
and
$\rho_2=-\rho_E$; the maximal cones of $\Sigma_X$ are the four cones
generated respectively by 
$\{\rho_0,\rho_E\}$,
$\{\rho_E,\rho_1\}$,
$\{\rho_1,\rho_2\}$,
and $\{\rho_2,\rho_0\}$.
\end{exs}

\begin{rem}\label{rem:antican:toric}
One can show that the image of $\sum_{i\in I} D_i$ in $\Pic(X)$
coincides with the class of the anticanonical line bundle
$\class{\antican{X}}$: indeed one checks that the form $\wedge_{i\in I_{\sigma}}\frac{dx_i}{x_i}$
on $V_{\sigma}$ glue to a rational section of the canonical bundle.
\end{rem}
\subsection{Homogeneous coordinates on toric varieties}\label{subsec:hom:cor:tor}

When dealing with (say, projective) varieties, it may be useful
to have coordinates on it, for instance to do some computations. One basic way to do
this is to embed $X$ into a projective space $\bP^n$: the homogeneous
coordinates on $\bP^n$ yields coordinates on $X$. As already pointed
out, one drawback of this approach is that there are a lot of available embeddings $X\inject \bP^n$,
and thus no canonical choice for such coordinates.

A different approach was proposed by Cox for toric varieties. 
The basic idea is to observe that the
homogeneous coordinates on $\bP^n$ correspond to the quotient of the
affine space $\bA^{n+1}$ minus the origin by the diagonal action of
$\bG_m$. Let us denote by $\pi$ the quotient map $\bA^{n+1}\setminus
\{0\}\to \bP^n$. If we view $\bP^n$ as a toric variety in the usual
way, the pull back by $\pi$ of a boundary divisor
is the trace of a coordinate hyperplane on $\bA^{n+1}\setminus
\{0\}$.

This construction can be generalized to any smooth
projective
toric variety $X$ as follows: let $\{D_i\}_{i\in I}$ be the finite
set of boundary divisors. 
Let $\TNS{X}$ be the torus whose character group is $\Pic{X}$, that is,
the torus $\Hom(\Pic(X),\bG_m)$ (it is 
called the \termin{N\'eron-Severi torus} of $X$; in our setting the Picard
group and the N\'eron-Severi group coincide).

The morphism $\bZ^I\to \Pic(X)$ extracted from the exact sequence 
\eqref{eq:exsq}
yields by duality an algebraic group morphism $\TNS{X}\to \bG_m^I$. 
Composing with the coordinatewise action of $\bG_m^I$ on $\bA^I$, we get an
action of $\TNS{X}$ on $\bA^I$. 
If $X=\bP^n$, one has $\Pic(X)\isom\bZ$ and 
the action of $\TNS{X}\isom \bG_m$ on $\bA^{n+1}$ is the diagonal one.

We set
\begin{equation}\label{eq:def:torsX:toric}
\tors_X\eqdef \bA^I \setminus 
\bigcupu{\substack{J\subset I \\ \capu{i\in J}D_i=\vide}} \capu{i\in J}\{x_i=0\}.
\end{equation}
Recall that the condition $\cap_{i\in J}D_i=\vide$ may be expressed in
terms of the fan $\Sigma_X$ by  saying
that the $\{\rho_i\}_{i\in J}$ are not the rays of a
cone of the fan. For $X=\bP^n$ the only subset of
$\{0,\dots,n\}$ satisfying the condition is $\{0,\dots,n\}$ itself. 
For $X=\bP^2$ blown-up at $(0:0:1)$, the minimal subsets 
of $\{0,1,2,E\}$
satisfying the
conditions are $\{0,1\}$ and $\{2,E\}$.

One checks immediatly that the action of $\TNS{X}$ on $\bA^I$ leaves
$\tors_X$ invariant. Now we define a morphism $\pi\,:\,\tors_X\to
X$. Recall that for a cone $\sigma$ of the fan, we have set $I_{\sigma}=\{i\in I,\,\rho_i\in\sigma\}$.

First we notice that the open subsets of $\tors_X$
\begin{equation}
\tors_{X,\sigma}
=
\{\prod_{i\in I\setminus I_{\sigma}}x_i\neq 0\} 
=\bA^{I_{\sigma}}\times \bG_m^{I\setminus I_{\sigma}}
\end{equation}
are $\TNS{X}$-invariant and form a covering of $\tors_X$ when $\sigma$ varies along the
maximal cones of $\Sigma_X$. 
Now let $\sigma$ be such a cone.
Then
$\{\rho_i\}_{i\in I_{\sigma}}$ is a $\bZ$-base of $\carac{T}^{\vee}$
(recall that since $X$ is smooth, the fan $\Sigma_X$ is regular), 
thus the classes of the divisors $\{D_i\}_{i\notin I_{\sigma}}$ in $\Pic(X)$ form
    a $\bZ$-basis of it, and therefore determine isomorphisms $\Pic(X)\isom
    \bZ^{I\setminus I_{\sigma}}$ and 
$\TNS{X}\isom \bG_m^{I\setminus  I_{\sigma}}$. 
Now we define 
\begin{equation}
\pi_{\sigma}\,:\,\tors_{X,\sigma}\to V_{\sigma}=\Spec(k[\sigma^{\vee}\cap \carac{T}])
\end{equation}
by composing the natural isomorphism $\bA^{I_{\sigma}}\isom
V_{\sigma}$, the first projection $\bA^{I_{\sigma}}\times
\bG_m^{I\setminus I_{\sigma}}\to \bA^{I_{\sigma}}$ and the morphism
\begin{equation}
\map{\bA^{I_{\sigma}}\times
\bG_m^{I\setminus I_{\sigma}}}{\bA^{I_{\sigma}}\times
\bG_m^{I\setminus I_{\sigma}}}{(x,t)}{(t^{-1}.x,t)}.
\end{equation}

We leave to the reader the task of verifying that
 the morphisms $\pi_{\sigma}$ glue to a morphism
$\pi\,:\,\tors_X\to X$. It follows immediatly from the construction
that $\pi$ is  a $\TNS{X}$-torsor over $X$; 
here, since $\TNS{X}$ is a split torus, it simply means that there is
a Zariski-open covering
$(X_{\alpha})$ of $X$  and isomorphisms $\varphi_{\alpha}\,:\,U_{\alpha}\times \TNS{X}\isom
\pi^{-1}(U_{\alpha})$ such that $\pi\circ
\varphi_{\alpha}=\pr_{U_{\alpha}}$ and the action of $\TNS{X}$ on 
$U_{\alpha}\times \TNS{X}$ induced by $\varphi_{\alpha}^{-1}$ is by
translations on the second factor; in our case, the open covering is
given by the $V_{\sigma}$'s; when dealing with nonsplit tori, one
has to replace the Zariski topology by the \'etale topology.

\begin{rem}\label{rem:pullback}
One checks, using the covering $\{V_{\sigma}\}$, that the divisor $\pi^{\ast}D_i$ is the trace of the
coordinate hyperplane  $\{x_i=0\}$ on $\tors_X\subset \bA^I$.
\end{rem}
\begin{rem}
In the construction of $\pi\,:\,\tors_X\to X$ we did not use the fact
that $X$ was projective, and indeed the construction is valid for any smooth toric variety.
For generalization to other toric varieties and some applications we
refer to Cox's paper \cite{Cox:hom_coo_ring}.
\end{rem}
\begin{rem}\label{rem:hcr:toric}
There is a natural $\Pic(X)$-graduation on the polynomial ring
$k[x_i]_{i\in I}$, which yields the $\TNS{X}$-action on $\bA^I$ used above: we set $\deg(x^{\bd})=\class{\sum d_iD_i}$.
Now let $D=\sum a_i\,D_i$ be an integral combination of the $D_i$'s. It
is known that the set 
\begin{equation}\carac{T}_D=\{m\in \carac{T},\,\forall i\in
I,\,\acc{m}{\rho_i}+a_i\geq 0\}
\end{equation} is a basis of
\begin{equation}
H^0(X,\str{X}(D))=\{f\in k(X)^{\inv},\,\div(f)+D\geq 0\}\cup \{0\}.
\end{equation} 
But the
map $m\mapsto \prod x_i^{\acc{m}{\rho_i}+a_i}$ is clearly a
bijection from $\carac{T}_D$ onto the set of monomials of degree
$\class{D}$, thus the degree $\class{D}$ part of $k[x_i]_{i\in I}$
may be identified with the vector space of global sections $H^0(X,\str{X}(D))$.
\end{rem}

\subsection{Application to the description of the functor of points of
  a toric variety}

Now we explain the application of homogeneous coordinate
rings to the description of the functor of points of a smooth
projective toric variety $X$ defined over $k$, that is, the functor which maps
a $k$-scheme $S$ to the set $\Hom_k(S,X)$. This is due to Cox (\cite{Cox:funct}). 

Here again the case of $\bP^n$ may serve as
a basic guiding example. In fact what we will seek to generalize in a minute
 is the following well-known property: a morphism $S\to \bP^n$
is determined by the datum of a line bundle on $S$ and $n+1$ global
sections of this line bundle which do not vanish simultaneously.

One can slightly restate the previous property by saying that 
a morphism $S\to \bP^n$ is determined by the datum of $n+1$ line
bundles $\cL_0,\dots,\cL_n$ on $S$, a global section $s_i$ of
$\cL_i$ for each $i$ such that the $s_i$ do not vanish simultaneously
and a collection of  isomorphisms
$\varphi_{i,j}\,:\,\cL_i\to \cL_j$ 
which are compatible in the sense that
$\varphi_{j,k}\circ \varphi_{i,j}=\varphi_{i,k}$.
Note that the datum does not consist simply of pairwise isomorphic line
bundles $\{\cL_i\}$, the isomorphisms are also part of it. 
One intuitive way to understand
this is the following: the morphism corresponding to the above datum
is given by
\begin{equation}\label{eq:map:S:to:Pn}
\map{S}{\bP^n}{x}{(s_0(x):\dots:s_n(x))}.
\end{equation}
The value of $s_i(x)$ is defined only modulo the choice of a local
trivialization of $\cL_i$ around $x$; changing the trivialization
multiplies it by a nonzero scalar; but if we change the
trivializations of the various $\cL_i$'s 'independently', the scalar
will not be the same for each $i$, and the map \eqref{eq:map:S:to:Pn} will not be well defined.
Fixing the isomorphisms
$\varphi_{i,j}$ forces us to do only 'compatible' change of trivializations.

Now let $X$ be a smooth projective toric variety.
Recall that we have the exact sequence
\begin{equation}
0\to \carac{T}\to \oplusu{i\in I} \bZ D_i\to \Pic(X)\to 0.
\end{equation}
This means in particular that for every $m\in \carac{T}$ we have 
\begin{equation}
\div(m)=\sum_{i\in I}\acc{m}{\rho_i}D_i.
\end{equation}
Therefore $m\in \carac{T}$ determines an isomorphism
$c_m\,:\,\otimesu{i\in I} \str{X}(D_i)^{\otimes \acc{m}{\rho_i}}\isom \str{X}$. It is clear that
$c_m\otimes c_{m'}=c_{m+m'}$.

Let $f\,:\,S\to X$ be a morphism from a $k$-scheme $S$ to our toric
variety $X$. Let $\cL_i\eqdef f^{\ast} \str{X}(D_i)$, $u_i\eqdef f^{\ast}\scan{D_i}$
(where $\scan{D_i}$ denote the canonical section of $D_i$) and, for
$m\in \carac{T}$, let $d_m\eqdef f^{\ast} c_m$.
Then the datum $(\{(\cL_i,u_i)\}_{i\in I},\{d_m\}_{m\in \carac{T}})$ is an
\termin{$X$-collection on $S$} in the following sense:
\begin{defi}\label{defi:coll}
An $X$-collection on a $k$-scheme $S$ is the datum of:
\begin{enumerate}
\item a family of pairs $\{(\cL_i,u_i)\}_{i\in I}$ where $\cL_i$ is a line bundle
on $S$  and $u_i$ a global section of $\cL_i$
such that for every $J\subset I$ satisfying $\cap_{i\in J}D_i=\vide$ 
the sections $\{u_i\}_{i\in J}$ do not vanish simultaneously
(non-degeneracy condition);
\item
a family $\{d_m\}_{m\in M}$ of isomorphisms $d_m\,:\,\otimes \cL_i^{\otimes \acc{m}{\rho_i}}\isom \str{S}$
such that $d_m\otimes d_{m'}=d_{m+m'}$.
\end{enumerate}
\end{defi}
We have an obvious notion of isomorphism of $X$-collections on $S$ and
we denote by $\Coll_{X,S}$ the set of isomorphism classes of
$X$-collections on $S$. Note that $\Coll_{X,S}$ is clearly fonctorial
in $S$. We denote by $C_X$ the $X$-collection on $X$ given by 
$(\{(\str{X}(D_i),\scan{D_i})\}, \{c_m\}\}$.
Using remark
\ref{rem:pullback}, one checks that the collections $\pi^{\ast}C_X$
and $(\{(\str{\tors_X},x_i)\},\{1\})$ are isomorphic.

In \cite{Cox:funct}, Cox proves that the  maps
\begin{equation}
\map{\Hom(S,X)}{\Coll_{X,S}}{f}{f^{\ast}C_X}
\end{equation}
define an isomorphism between the functor of points of $X$ and the
functor which associates to a $k$-scheme $S$ the set $\Coll_{X,S}$.

Let us explain the proof.
First we describe a map $\Coll_{X,S}\to \Hom(S,\bP^n)$.  
Let $(\{(\cL_i,u_i)\},\{d_m\})$ be a representative 
of an element $C$ of $\Coll_{X,S}$.
First assume that the $\cL_i$'s are trivial. Thus $C$
has a representative of the form $(\{(\str{S},u_i)\},\{d_m\})$. In
this case
the datum of $\{d_m\}$
is equivalent to the datum of a group morphism 
\begin{equation}
\carac{T}\to
\Aut(\str{S})=H^0(S,\str{S})^{\inv},
\end{equation} 
that is, an element of $T(S)$.
Moreover for $t,t'\in T(S)$ the two $X$-collections $(\{(\str{S},u_i)\},t)$
and $(\{(\str{S},u'_i)\},t')$ are isomorphic
if and only if there is an element 
$\lambda\in \bG_m^I(S)=H^0(S,\str{S})^{\inv}$ such that $\lambda.t=t'$ 
(recall the
exact sequence of tori $1\to \TNS{X}\to \bG_m^I\to T\to 1$) and $\lambda_i.u_i=u'_i$.
In particular we may choose a representative of $C$ of the form $(\{(\str{S},u_i)\},1))$.
Then the $u_i$'s  define a morphism $S\to \bA^I$, whose image lies in
$\tors_X$ thanks to the non degeneracy condition satisfied by the $u_i$'s.
By composition with $\pi\,:\,\tors_X\to X$ we obtain a morphism $S\to X$. By the
previous observation, the morphism $S\to
\tors_X$ depends on the choice of the representative $(\{(\str{S},u_i)\},1)$
but the induced morphism $S\to X$ does not because any other
representative of this shape differ by the action of an element of
$\bG_m^I$ whose image in $T$ is trivial, that is, an element of $\TNS{X}$.
Let us denote by $f_C$ the above constructed morphism $S\to X$. The
construction is clearly fonctorial: for any morphism $\varphi\,:\,T\to S$ one has
$f_{\varphi^{\ast}C}=f_C \circ \varphi$.

If the $\cL_i$'s are not trivial, cover $S$ by open subset
trivializing them, and apply the previous construction ; by
fonctoriality the corresponding morphisms agree on the intersections,
hence can be glued to a morphism $f_C\,:\,S\to X$. To
check that $f_C^{\ast} C_X$ and $C$ are isomorphic, again reduce to the
case where the $\cL_i$'s are trivial and use the isomorphism 
$\pi^{\ast}C_X\isom (\{(\str{\tors_X},x_i)\},1)$.

It remains to check that if $f^{\ast}C_X$ and $C$ are isomorphic then $f=f_C$.
This is easy if $f$ factors through $\pi$ and we reduce to the latter
case by considering the morphisms $f^{-1}V_{\sigma}\to
V_{\sigma}$  induced by $f$ and the fact that over $V_{\sigma}$, $\pi$
is a trivial torsor, hence has a section.
\begin{rem}
One obtain an analogous description of the functor assigning to a
$k$-scheme $S$ the set of morphisms $\Hom(S,\bP^1)$ which do not
factor through the boundary $\cup\,D_i$: by remark \ref{rem:pullback}
and the above construction they correspond to
those $X$-collections $(\{(\cL_i,u_i)\},\,\{d_m\})$ for which no one of
the $u_i$ is the zero section. We call such collections \termin{non degenerate $X$-collections}.
\end{rem}

\subsection{Description of $\HOM(\bP^1,X)$ for $X$ toric}
Now we are ready to give a useful description of the scheme
$\HOM(\bP^1,X)$ where $X$ is a smooth projective toric variety.

More precisely, for every $\bd\in \bZ^{I}$, we will describe 
the variety  parametrizing the set of morphisms $\bP^1\to X$
such that for $i\in I$ we have $\deg(f^{\ast} D_i)=d_i$, and which do not factor through
the boundary\footnote{Thus we will only describe an open subset of 
$\HOM(\bP^1,X,\bd)$; this is mainly for the sake of simplicity, since
  the full variety could be described using very similar arguments.} $\cup D_i$ (recall that $U=X\setminus \cup D_i$ is the
open orbit). Note that this variety will be empty if $\bd$ does not belong
to the image of $\Pic(X)^{\vee}$ in $\bZ^I$ (recall the exact sequence
\eqref{eq:exsq:dual}); and if $\bd\in \Pic(X)^{\vee}$, 
this is exactly the variety denoted previously by 
$\HOM_{U}(\bP^1,X,\bd)$; in accordance with previously introduced
notations we will use the symbol $y$ to denote such a $\bd$. 
Recall that the injection $\Pic(X)^{\vee}\hookrightarrow \bZ^I$ is
given by $y\mapsto (\acc{y}{D_i})$.
Recall also that since the $D_i$'s
generate the effective cone of $X$,
$\HOM_{U}(\bP^1,X,y)$ will be empty if $y$ does not belong to
$\ceff(X)^{\vee}\cap \Pic(X)^{\vee}$, or equivalently, to
$\Pic(X)^{\vee}\cap \bN^I$.

Let $L$ be a $k$-extension.
Let $y\in \bN^{I}\cap \Pic(X)^{\vee}$. By the previous
section, a element $f\in \HOM_{U}(\bP^1,X,y)(L)$ is entirely
determined by an isomorphism class of
$X$-collection $(\{(\cL_i,u_i)\},\{d_m\})$ on $\bP^1_L$, with
$\deg(\cL_i)=\acc{y}{D_i}=d_i$, and no one of the $u_i$'s is the zero section.
We may assume that $\cL_i$ is $\str{\bP^1_L}(d_i)$,
thus $u_i$ may be identified with a nonzero homogeneous polynomial in
two variables of degree $d_i$, denoted by $P_i$. 
As explained above, the datum of $\{d_m\}$ is equivalent to the datum
of a point of $T(\bP^1_L)=T(H^0(\bP^1_L,\str{\bP^1_L}))=T(L)$ and two
collections $(\{(\str{\bP^1_L}(d_i),P_i)\},t)$ and $(\{(\str{\bP^1_L}(d_i),P'_i)\},t')$
are isomorphic if and only if there exists $\lambda=(\lambda_i)\in \bG_m^I(L)$ such
that $\lambda.t=t'$ and $\lambda_i.P_i=P'_i$.

For any nonnegative integer $d$, we denote by $\Homogs_{d}$ the
variety $\bA^{d+1}\setminus \{0\}$ (this is only to stress that we wiew
a point of the latter as the coefficients of a nonzero homogeneous polynomial
in two variables of degree $d$). For $\bd\in \bN^I$, set $\Homogs_{\bd}\eqdef\prod_i \Homogs_{d_i}$.

Elimination theory 
shows that there exists a dense open subset $U_{\bd}$
of $\Homogs_{\bd}$ such that for every field $L$ we have that $(P_i)$
lies in $U_{\bd}(L)$ if and only 
if the $P_i$'s do not have a common nontrivial root in an algebraic
closure of $L$. Thus there exists an open dense subset $\Homogs_{y,X}$ 
of $\Homogs_{y}$ such that $(P_i)$ lies in $\Homogs_{y,X}(L)$ if
and only if the $P_i$'s satisfy the non degeneracy condition of
definition \ref{defi:coll}.

It follows that the map wich associates to the (non degenerate) collection
$(\{(\str{\bP^1}(d_i),P_i)\},t)$ the element $(P_i)\in \Homogs_{y,\Sigma(X)}(L)$
induces a bijection between the isomorphism classes of non degenerate collections and
the set
\begin{equation}
\Homogs_{y,X}(L)/\TNS{X}(L)
=
(\Homogs_{y,X}/\TNS{X})(L).
\end{equation}
The equality holds even if $L$ is not algebraically closed because
the torsor
$\Homogs_{y,\Sigma(X)}\to \Homogs_{y,\Sigma(X)}/\TNS{X}$
is locally trivial for the Zariski topology.

The previous reasoning suggests that the variety
$\Homogs_{y,X}/\TNS{X}$ 
should be isomorphic to $\HOM_{U}(\bP^1,X,y)$. It does not prove it, 
since we
only looked at the level of points with value in a field, but with little extra work one can  show that this is indeed
the case.

Note in particular that for every $y\in \Pic(X)^{\vee}\cap\ceff(X)^{\vee}$
the variety $\HOM_{U}(\bP^1,X,y)$ is geometrically irreducible of dimension
\begin{equation}
\sum \acc{y}{D_i}+\# I-\rk(\Pic(X))
=\sum \acc{y}{D_i}+\dim(X)
=\acc{\antican{X}}{y}+\dim(X)
\end{equation}
the last equality coming from remark \ref{rem:antican:toric}. Thus
questions \ref{ques:bat} have an affirmative answer
for toric varieties.

\subsection{Application to the degree zeta function}\label{subsec:app}
For $\bd\in \bN^I$, set $\bP^{\bd}\eqdef \prod \bP^{d_i}$ and for
$y\in \Pic(X)^{\vee}\cap \ceff(X)^{\vee}$ let $\bP^{y}_{X}$ be the
image of $\Homogs_{y,X}$ in $\bP^{y}$. 
Since 
$\bP^{y}_X=\Homogs_{y,X}/\bG_m^I$, the morphism
\begin{equation}
\Homogs_{y,X}/\TNS{X}\to \bP^{y}_X
\end{equation} 
is a $\bG_m^I/\TNS{X}=T$-torsor and we have 
\begin{equation}
\class{\Homogs_{y,X}/\TNS{X}}
=
\class{T}\,\class{\bP^{y}_X}=(\bL-1)^{\dim(X)}\,\class{\bP^{y}_X}.
\end{equation}
To evaluate the class of  $\bP^{y}_X$ in the Grothendieck ring
we will use a classical
tool to `get rid' of coprimality conditions,
namely, we will perform a kind of M\"obius inversion. 
This will allow us to reduce to the case of $\bP^{\bd}$, whose class is readily computed as
$\prod_{i\in I}\frac{\bL^{d_i+1}-1}{\bL-1}$; note that in order to
give a rigorous meaning to the previous expression we have to
work in the completed Grothendieck ring, or at least in a suitable
localization, since we do not know whether
$\bL-1$ is a zero divisor in $\kovark$.

We claim that there is a unique fonction
$\muxm\,:\,\bN^I\to \kovark$ satisfying:
\begin{equation}
\forall \,\bd\in \bN^I,\quad \class{\bP^{\bd}_{X}}
=\sum_{0\leq \bd'\leq  \bd}
\muxm(\bd') 
\class{\bP^{\bd-\bd'}}.
\end{equation}
The claim follows immediately from an induction on the lenght $\abs{\bd}=\sum d_i$ of $\bd$.

Now, for $y\in \Pic(X)^{\vee}\cap \ceff(X)^{\vee}$, we can write:
\begin{align}
\class{\HOM_{U}(\bP^1,X,y)}
&=
\class{\Homogs_{y,X}/\TNS{X}}
\\
&=
(\bL-1)^{\dim(X)}
\sum_{0\leq \bd\leq  y}
\muxm(\bd) 
\bP^{y-\bd}
\\
&
=
(\bL-1)^{\dim(X)-\# I}
\sum_{0\leq \bd\leq  y}
\muxm(\bd) 
\prod_{i\in I}(\bL^{\acc{y}{D_i}-d_i+1}-1)
\label{eq:expr:gen}
\\
&
=
\frac{\bL^{\#I+\sum_i \acc{y}{D_i}}}{(\bL-1)^{\rk(\Pic(X))}}
\sum_{0\leq \bd\leq  y}
\muxm(\bd) 
\,
\bL^{-\abs{\bd}}
\,
\prod_{i\in I}(1-\bL^{-\acc{y}{D_i}+d_i-1})
\\
&
=
\frac{\bL^{\dim(X)+\acc{y}{\antican{X}}}}
{(1-\bL^{-1})^{\rk(\Pic(X))}}
\sum_{0\leq \bd\leq  y}
\muxm(\bd) 
\,
\bL^{-\abs{\bd}}
\,
\prod_{i\in I}(1-\bL^{-\acc{y}{D_i}-d_i-1}).
\label{eq:expr:laccyantican:fin}
\end{align}
Let us explain very sketchly
how we will proceed with the asymptotic estimation of 
$\bL^{-\acc{y}{\antican{X}}}\class{\HOM_{U}(\bP^1,X,y)}$. 
We will have to show that the dominant term is given by approximating 
in \eqref{eq:expr:laccyantican:fin}
the quantity $\prod_{i\in I}(1-\bL^{-d_i-d'_i-1})$
by $1$, proving in particular that the series
\begin{equation}\label{eq:ser:mu}
\sum_{\bd\in \bN^I}\muxm(\bd) \bL^{-\abs{\bd}}
\end{equation}
converges in the completed Grothendieck ring and that its limit is nonzero.
Thus we will have established that the answer to question \ref{ques:fin:bis}(2) is
affirmative for $X$. 

Concerning the anticanonical degree zeta function, 
we will see that the dominant term is obtained by using
in \eqref{eq:expr:laccyantican:fin} the same approximation as before
and  by replacing moreover, for $\bd\in \bN^I$,  the summation over 
$\sum_{0\leq \bd\leq  y}$ (\ie a summation over a truncation of $\ceff(X)^{\vee}$)
by a summation over the whole dual of the effective cone.
Thus the main term will be
\begin{equation}
\frac{\bL^{\dim(X)}}{(1-\bL^{-1})^{\rk(\Pic(X))}}
\left(\sum_{\bd\in \bN^I}\muxm(\bd)\bL^{-\abs{\bd}}\right)
\left(
\sum_{\bd\in \Pic(X)^{\vee}\cap \ceff(X)^{\vee}}
(\bL\,t)^{\acc{\bd}{\antican{X}}}
\right).
\end{equation}
which corresponds indeed to the second term appearing in \eqref{eq:series}.

The first job we will be occupied with is to settle the convergence of
\eqref{eq:ser:mu} (and more precisely to get a good control on the
behaviour of the M\"obius fonction).
We will first describe what happens over a finite field after
specializing by the morphism 'number of $k$-points'.
In this case the multiplicativity property of the
M\"obius function allows to  settle easily the convergence of 
the 'specialized' version (more rigorously the analogous of) \eqref{eq:ser:mu}, by
decomposing it as an Euler product. Then we will explain how this
approach may be `mimicked' to study the motivic series \eqref{eq:ser:mu}.

Finally, we will see how to show that the `approximations' described above are valid, that is
to say that the error terms resulting from these approximations are
suitably controlled.

\subsection{The leading term of the classical degree zeta function of a toric variety}\label{subsec:dzf:tv}

In this section we assume that $k$ is a finite field with $q$
elements, 
and we will study the convergence  of the 'specialization' of \eqref{eq:ser:mu}
under $\#_k$, that is,  the series
\begin{equation}\label{eq:ser:mu:fin}
\sum_{\bd\in \bN^I}\#_k[\muxm(\bd)] q^{-\abs{\bd}}.
\end{equation}
Let us recall thet the world 'specialization' is to be taken in a loose sense, 
since  the morphism $\#_k$ does not extend to completed Grothendieck ring; in particular the 
convergence of \eqref{eq:ser:mu:fin} would not follow from the
convergence of \eqref{eq:ser:mu}.

A key fact in the present setting is that the specialized function
\begin{equation}
\#_k \muxm\,:\,\bN^I\to \bZ
\end{equation}
can be refined to a
function 
\begin{equation}
\mux\,:\,\bigsqcup_{\bd\in \bN^I} \bP^{\bd}(k)\to \bZ,
\end{equation} 
in the sense that for all $\bd$
we will have 
\begin{equation}\label{eq:cardkmuxm}
\#_k\muxm(\bd)=\sumu{\becD\in \bP^{\bd}(k)}\mux(\becD).
\end{equation}
Indeed, define $\mux$ by the relation 
\begin{equation}
\forall \bd\in \bN^I, \quad \forall \becD\in \bP^{\bd}(k),\quad
\sum_{\becD'\leq \becD}\mux(\becD')=\ind_{\bP^{\bd}_{X}}(\becD).
\end{equation}
Here we identify $\bP^{\bd}(k)$  with the set of $I$-uples
of effective $k$-divisors $\becD$ on $\bP^1$ of degree $\bd$: this
gives a sense to the expression $\becD'\leq \becD$.

The basic properties of $\mux$ are listed in the following proposition.
The reader may check them as an easy exercise.
\begin{prop}\label{prop:mu}
\begin{enumerate}
\item
$\mux$ is a multiplicative function: whenever $\becD$ and $\becD'$ are
such that $\cD_i$
and $\cD'_i$ are coprime (that is, have disjoint supports) for each
$i$, 
we have $\mux(\becD+\becD')=\mux(\becD)\,\mux(\becD')$.
\item
There exists a unique map $\mux^0\,:\,\bN^{\,I}\to \bZ$ such that for all $\bn\in \bN^I$
and every closed point $\cP$ of $\bP^1_k$ we have
\begin{equation}
\mux((n_i\,\cP))=\mux^0(\bn)
\end{equation}
\item
We have
\begin{equation}
\forall \bn\in \{0,1\}^I,\quad \sum_{0\leq \bn'\leq \bn}\mux^0(\bn')=
\left\{
\begin{array}{rl}
1&\text{ if }\capu{i\in I,\,n_i=1}D_i\neq \vide\text{ or }\bn=0\\
0&\text{ otherwise.}
\end{array}
\right.
\end{equation}
\item\label{item:conv}
We have $\mux^0(\bn)=0$ if $\sum n_i=1$ or if there exists $i$ such
that $n_i\geq 2$.
\item
Denoting by $\{0,1\}^I_X$ the set of elements $\bn$ of $\{0,1\}^I$
such that $\Minu{\sigma\in \Sigma_X}\sum_{i\notin \sigma(1)}n_i>0$
and by $\{0,1\}^I_{X,\text{min}}$ the set of the minimal elements of $\{0,1\}^I_X$,
we have
\begin{equation}
\forall \bn\in \{0,1\}^I,\quad \sum_{0\leq \bn'\leq \bn}\mux^0(\bn')=
\left\{
\begin{array}{rl}
1&\text{ if }\bn=0\\
0&\text{ if }\bn\neq 0\text{ and }\bn\notin \{0,1\}^I_X\\
(-1)^{\# \{\bn'\in  \{0,1\}^I_{X,\text{min}},\,\bn'\leq \bn\}}&\text{ if
}\bn\in \{0,1\}^I_X
\end{array}
\right.
\end{equation}
\end{enumerate}
\end{prop}

Using the classical fact that for $\eps>0$ the Euler product
\begin{equation}
\prod_{\cP\in (\bP^1_k)^{(0)}} 1+
\bigou{\deg(\cP)\to +\infty}{q^{-(1+\eps)\,\deg(\cP)}}
\end{equation}
(where $(\bP^1_k)^{(0)}$ denotes the set of closed points of $\bP^1_k$)
converges and thanks to the previous proposition, we obtain that the series
\begin{equation}\label{eq:sumcardkmuxm}
\sum_{\bd\in \bN^I}
\#_k[\muxm(\bd)]\,q^{-\abs{\bd}}
=
\prod_{\cP\in (\bP^1_k)^{(0)}}
\sum_{\bn\in  \{0,1\}^I}
\mux^0(\bn)\,q^{\,-\deg(\cP)\,\abs{\bn}}
\end{equation}
is  absolutely convergent. 
The following proposition will yield a nice
interpretation
of the right hand side of  \eqref{eq:sumcardkmuxm}.
\begin{prop}\label{prop:rel:tor}
Let $L$ be a finite extension of $k$.
We have the relation:
\begin{equation}
\sum_{\bn\in  \{0,1\}^I}
\mux^0(n_i)\,(\# L)^{-\sum_i n_i}
=
(1-\# L)^{\rk(\Pic(X))}\# X(L)/(\# L)^{-\dim(X)}
\end{equation}
\end{prop}
\begin{proof}
We will in fact prove the following relation in the Grothendieck ring
of varieties (valid over any field): 
\begin{equation}\label{eq:mot:rel:toric}
\sum_{\bn\in  \{0,1\}^I}
\mux^0(\bn)\,\bL^{\# I-\sum_i n_i}
=
(\bL-1)^{\rk(\Pic(X))}\,\class{X}.
\end{equation}
The desired relation follows immediatly by applying the realization
morphism $\#_{L}$ and the relation $\#I=\dim(X)+\rk(\Pic(X))$.

Since the morphism $\tors_X\to X$ is a torsor under a split torus of
dimension $\rk(\Pic(X))$, we have
\begin{equation}
(\bL-1)^{\rk(\Pic(X))}\,\class{X}=\class{\tors_X}.
\end{equation}
Now for $\bn\in \{0,1\}^I$ let $\bA^I_{\bn}\eqdef
\capu{i,\,n_i=1}\{x_i=0\}$. Reminding the definition of $\tors_X$, we have 
(we refer to proposition \ref{prop:mu} for the definition of $\{0,1\}^I_X$)
\begin{equation}\label{eq:torsx:eq}
\tors_X
=\bA^I\setminus \cupu{\bn\in \{0,1\}^I_X}\bA^I_{\bn}
=\bA^I\setminus \cupu{\bn\in \{0,1\}^I_{X,\text{min}}}\bA^I_{\bn}
\end{equation}
The inclusion-exclusion principle and the scissor relations now yield
\begin{align}
\class{\cupu{\bn\in \{0,1\}^I_{X,\text{min}}}\bA^I_{\bn}}
&=\sum_{\vide\neq A\subset \{0,1\}^I_{X,\text{min}}}(-1)^{1+\# A}\class{
\capu{\bn\in A}\bA^I_{\bn}}
\\
&=\sum_{\vide\neq A\subset \{0,1\}^I_{X,\text{min}}}(-1)^{1+\# A}
\class{\bA^I_{\Maxu{\bn\in A}(\bn)}}.
\end{align}
Note that the map which associates to a non empty subset $A$ of $\{0,1\}^I_{X,\text{min}}$
the element $\Maxu{\bn\in A}\,(\bn)$ is a bijection from
$\cP(\{0,1\}^I_{X,\text{min}})\setminus \vide$ onto $\{0,1\}^I_X$, whose inverse
is the map associating to $\bn\in\{0,1\}^I_X$ the subset $\{\bn'\in
\{0,1\}^I_{X,\text{min}},\,\bn'\leq \bn\}$. Hence the above equality may be rewritten as
\begin{equation}
\class{ \cupu{\bn\in \{0,1\}^I_{X,\text{min}}}\bA^I_{\bn}}
=
\sum_{\bn \{0,1\}^I_{X}}
(-1)^{1+
\# \{\bn'\in    \{0,1\}^I_{X,\text{min}},\,\bn'\leq \bn\}
}
\,
\bL^{\,\#I -\abs{\bn}}
\end{equation}
Thus we have by proposition \ref{prop:mu}
\begin{equation}
\class{ \cupu{\bn\in \{0,1\}^I_{X,\text{min}}}\bA^I_{\bn}}
=\bL^{\,\# I}-\sum_{\bn \{0,1\}^I}\mux^0(\bn)\,\bL^{\,\#I-\abs{\bn}}
\end{equation}
From this and \eqref{eq:torsx:eq},  the desired relation follows immediatly.
\end{proof}

Later on, the previous results on the M\"obius fonction will allow us to show that
the answers to questions \ref{ques:fin} and \ref{ques:fin:bis}(1) are
affirmative for a toric variety $X$,
with a constant $c$ which may be written as
\begin{align}
\label{eq:def:cfinX}
c_{\text{fin}}(X)
&
\eqdef
\frac{q^{\dim(X)}}{(1-q^{-1})^{\rk{\Pic(X)}}}
\sum_{\bd\in \bN^I}
\#_k[\muxm(\bd)]\,q^{-\abs{\bd}} 
\\
&=\frac{q^{\dim(X)}}{(1-q^{-1})^{\rk{\Pic(X)}}}
\prod_{\cP\in (\bP^1_k)^{(0)}}
(1-q^{-\deg(\cP)})^{\rk(\Pic(X))}\,
\frac
{\# X(\kappa_{\cP})}
{q^{\,\deg(\cP)\,\dim(X)}}
\label{eq:tam}
\end{align}
where $\kappa_{\cP}$ is the residue field at the closed point $\cP$
(the second equality follows from proposition \ref{prop:rel:tor}).

Now remark that, disregarding convergence issues, the expression
\eqref{eq:tam} makes sense for any variety $X$ satisfying hypotheses
\ref{hyps:X}, not only the toric ones. Under suitable extra hypotheses
on $X$, 
Peyre showed that the Euler product in \eqref{eq:tam} is indeed convergent
and predicted that \eqref{eq:tam} should coincide with the constant
$c$ appearing in question \ref{ques:fin} (in fact Peyre's construction
applies to a far more general context, including the case of
nonconstant families;  \eqref{eq:tam} is  interpreted as the
volume of an adelic space associated to $X$, with respect to a certain
Tamagawa measure; see \cite{Pey:var_drap} for more details). 
Thus we will have checked that Peyre's prediction holds when $X$ is toric.
And, still sticking to the toric case, we are going to show that the constant $c$
appearing in question \ref{ques:mot} (which is an element of the completed
Grothendieck ring) has an analogous interpretation.

\subsection{The leading term of the motivic degree zeta function}
Our task is know to settle the convergence of the series
\begin{equation}
\sum_{\bd\in \bN^I}
\muxm(\bd)\,
\bL^{-\abs{\bd}}
\end{equation}
in the completed Grothendieck ring (more precisely, we will have to get a
good control on the virtual dimension of $\muxm(\bd)$, which will be
important to deal with the error terms alluded to in section \ref{subsec:app}).
When $k$ is finite, the analogous problem was easy to handle owing to the decomposition into Euler product. 
When working over the Grothendieck ring of
varieties or its completion, there is a priori no immediate analog of the notion of Euler product. Let
us now explain how to define such a notion. Let $X$ be a
quasi-projective variety defined over $k$. Consider the motivic
Hasse--Weil zeta function
\begin{equation}
\ZHWm(X,t)=\sum_{n\geq 0} \class{\sym{X}{n}}\,t^n
\end{equation}
where $\sym{X}{n}\eqdef X^n/\grsym_n$.
When $k$ is finite, 
$\#_k \ZHWm(X,t)=\ZHW(X,t)$ is the classical Hasse--Weil zeta function attached
to $X$ and we have the decomposition into Euler product
\begin{equation}\label{eq:dec:ZHW}
\#_k \ZHWm(X,t)
=\prod_{\cP\in X^{(0)}}(1-t^{\deg(\cP)})^{-1}
\end{equation}
where $X^{(0)}$ denotes  the set of closed points
of $X$. Now, for $n\in \bN$, let 
$X^{(0)}_n$ denote the set of closed points
of $X$ of degree $n$.
Then \eqref{eq:dec:ZHW} may be rewritten as
\begin{equation}\label{eq:dec:ZHW:bis}
\ZHW(X,t)
=\prod_{n\geq 1}(1-t^n)^{-\# X^{(0)}_n}.
\end{equation}
Note that the latter equality may be seen as an immediate formal
consequence of the relations
\begin{equation}\label{eq:rel:1}
\sum \# X(k_n)\,t^n=t\,\frac{d\log}{dt} \ZHW(X,t)
\end{equation}
and
\begin{equation}\label{eq:rel:2}
\forall n\geq 1,\quad \# X(k_n)=\sum_{d|n} d\,\# X^{(0)}_d
\end{equation}
(here $k_n$ is an extension of $k$ of degree $n$).

Now we may wonder whether there is a natural `geometric incarnation'
of the family $(\# X^{(0)}_n)_{n\geq 1}$, that is, a naturally
defined family $(Y_{X,n})$ of elements in the Grothendieck ring of
varieties such that when $k$ is finite the following relation holds:
\begin{equation}
\forall n\geq 1,\quad \#_k Y_{X,n}=\# X^{(0)}_n.
\end{equation}
If we accept to work in the Grothendieck ring of varieties
with denominators (that is, tensorized with $\bQ$), there is certainly a
cheap and straightforward way of doing this. For every
quasi-projective $k$-variety $X$, mimicking the relation
\eqref{eq:rel:1} and 
\eqref{eq:rel:2}
above, define families $(\Psi_n(X))_{n\geq
  1}$ and $(\Phi_n(X))_{n\geq 1}$ of elements of $\kovark$ and
$\kovark\otimes \bQ$ respectively\footnote{In \cite{Bou:prod:eul:mot},
these two families were denoted the opposite way;  it was a somewhat unfortunate
choice since, as pointed out to me by
E. Gorsky, what we denote by $(\Psi_n(X))$ in this text is a
formal analog of the so-called Adams operations, and the
letter $\Psi$ is commonly used to denote the latter.} by the relations
\begin{equation}
\sum_{n\geq 1}\Psi_n(X)\,t^n=t\,\frac{d\log}{dt} \ZHWm(X,t)
\end{equation}
and
\begin{equation}
\forall n\geq 1,\quad \Psi_n(X)=\sum_{d|n} d\,\Phi_d(X).
\end{equation}
For example, $\Psi_1(X)=\Phi_1(X)=\class{X}$, 
$\Psi_2(X)=2\,\class{\sym{X}{2}}-\class{X^2}$, and
$\Phi_2(X)=\class{\sym{X}{2}}-\frac{1}{2}(\class{X^2}-\class{X})$.
\begin{lemma}
\begin{enumerate}
\item
There are unique group morphisms $\Psi_n\,:\,\kovark\to \kovark$
and $\Phi_n\,:\,\kovark\to \kovark\otimes \bQ$ such that for every
quasi-projective variety $X$ we have $\Psi_n(\class{X})=\Psi_n(X)$ 
and $\Phi_n(\class{X})=\Phi_n(X)$.
\item
Assume that $k$ is finite. For every quasiprojective $k$-variety $X$,
every $n\geq 1$, and every finite extension $L$ of $k$ we have
\begin{equation}
\#_{L} \Psi_n(X)=X(L_n)\quad\text{and}\quad \#_{L}\Phi_n(X)=\# X^{(0)}_{L,n}.
\end{equation}
\item
For every $n\geq 1$ and $k\geq 0$, we have $\Psi_n(\bL^{k})=\bL^{k\,n}$.
\item
For every $n\geq 1$, we have
\begin{equation}\label{eq:form_phi_n}
\Psi_n(X)=
\sum_{k=1}^n
(-1)^{k+1}
\,
\frac{n}{k}
\,
\sum_
{\substack{
(m_1,\dots,m_k)\in (\bN_{>0})^k
\\ ~\\
m_1+\dots+m_k=n
}}
\,
\prod_{i=1}^k
\class{\sym{X}{m_i}}.
\end{equation}
\item
For every $n\geq 1$, $\Psi_n(X)$ and $\Phi_n(X)$ are in
$\fil^{-n\,\dim(X)}\mk\otimes \bQ$.
\end{enumerate}
\end{lemma}
\begin{rem}
We do not claim that $\Psi_n$ and $\Phi_n$ are ring morphisms.
In fact, by considering for example  the image of $\bL$,
it is straightforward to check that for $n\geq 2$, $\Phi_n$ is not a ring morphism. 
And anyway, over a finite field, it is clear that the composition of
$\Phi_n$ with $\#_k$ is not a ring morphism.
On the other hand, the composition of $\Psi_n$ with $\#_k$ is a ring
morphism (this amounts to the relation $\# (X\times
Y)(k_n)=\#X(k_n)\#Y(k_n)$), as well as its restriction to $\bZ[\bL]$
when $k$ is arbitrary. Nevertheless, it is not true that $\Psi_n$ is a
ring morphism, but the only demonstration I know relies on a rather subtle construction of
Larsen and Lunts, who proves in fact that the motivic
Hasse--Weil zeta function of $X$ is not rational in general for
$\dim(X)\geq 2$, contrarily to the intuition that the specialization
over a finite field might support (see \cite{LaLu:motivic_birational,LaLu:rationality_criteria}
and \cite[Remarque 2.7]{Bou:NYJM}).
This phenomenon may be seen as an incarnation of the fact
that the Grothendieck ring of varieties is definitively too big. By
contrast, the specializations of $\{\Psi_n\}$ to the Grothendieck ring of
motives are ring morphisms, as we will see below (and the
specialization of the motivic Hasse--Weil zeta function to the
Grothendieck ring of motives is conjectured to
be 
always
rational).
\end{rem}

Now it is easy to give a motivic counterpart of
\eqref{eq:dec:ZHW:bis}, since by
the very definition of $\Phi_n$, we have for every quasiprojective
variety $X$
\begin{equation}\label{eq:dec:zhwm}
\ZHWm(X,t)=\prod_{n\geq 1}(1-t^n)^{\,-\Phi_n(X)}
\end{equation}
where for every element $x$ of $\kovark\otimes \bQ$, $(1-t)^x$ denotes the series
\begin{equation}
\exp(x\,\log(1-t)).
\end{equation}
Note that \eqref{eq:dec:zhwm} holds in $1+(\kovark\otimes \bQ)[[t]]^+$
(for any commutative ring $1+A[[t]]^+$ denotes the set of formal
series with coefficients in $A$ and constant term $1$)
and that  more generally for any element $P(t)\in 1+(\kovark\otimes
\bQ)[[t]]^+$, $P(t)^x=\exp(x\,\log(1-P(t)))$ makes sense, as makes sense
the `motivic Euler product'
\begin{equation}
\prod_{n\geq 1}P(t^n)^{\,-\Phi_n(X)}.
\end{equation}
Now we see that an hypothetic motivic counterpart of the formula
\begin{multline}\label{eq:sumcardkmuxm:bis}
\sum_{\bd\in \bN^I}
\#_k\muxm(\bd)\,\prod t_i^{d_i}
=
\prod_{\cP\in (\bP^1_k)^{(0)}}
\sum_{\bn\in  \bN^I}\mux^0(\bn)\,\prod t_i^{\deg(\cP)\,n_i}
\\
=
\prod_{n\geq 1}
(\sum_{\bn\in  \bN^I}\mux^0(\bn)\,\prod t_i^{\,n.n_i})^{\# X^{(0)}_n}
\end{multline}
could be the (yet to be proved !) relation
\begin{equation}\label{eq:rel:to:be:proved}
\sum_{\bd\in \bN^I}
\muxm(\bd)\,\prod t_i^{d_i}
=
\prod_{n\geq 1}
\left(\sum_{\bn\in  \bN^I}\mux^0(\bn)\,\prod t_i^{\,n.n_i}\right)^{\,\Phi_n(\bP^1)}.
\end{equation}
\begin{rem}\label{rem:conv}
If the latter relation holds, 
it follows easily 
that the LHS of \eqref{eq:rel:to:be:proved} converges in the completed Grothendieck
ring at $t_i=\bL^{-1}$, and that the limit is nonzero: indeed we have $\Phi_n(\bP^1)\in
\fil^{-n}\mk$, hence 
thanks to point \ref{item:conv} of proposition \ref{prop:mu}
the series 
\begin{equation}
\left(
\sum_{\bn\in  \bN^I}
\mux^0(\bn)\,
\prod t_i^{\,n.n_i}
\right)^{\,\Phi_n(\bP^1)}
\end{equation} converges 
in $t_i=\bL^{-1}$ and its limit lies
in 
$1+\fil^{2\,n-n}\hmk=1+\fil^{n}\hmk$.

Moreove, still assuming that \eqref{eq:rel:to:be:proved} holds, using lemma 
\ref{lm:form_expl_prod_eul} below, 
point \ref{item:conv} of proposition \ref{prop:mu} and 
the fact that $\Phi_n(\bP^1)\in\fil^{-n}\mk$, one obtains  
(\cf \cite[proof of corollary 2.23]{Bou:prod:eul:mot})
the following bound on the virtual dimension of $\muxm(\bd)$:
\begin{equation}
\forall \bd\in \bN^I,
\quad 
\dim(\muxm(\bd))\leq \frac{\abs{\bd}}{2}.
\end{equation}
\end{rem}
\begin{notas}\label{notas:bbf}
Let $r\geq 1$ and 
$\bbf=(f_1,\dots,f_r)\in (\bN_{>0})^r$
such that 
\begin{equation}
f_1=f_2=\dots=f_{i_1}<f_{i_1+1}=f_{i_1+2}=\dots=f_{i_2}<f_{i_2+1}=\dots<f_{i_{k-1}+1}=\dots=f_r
\end{equation}
Then for any sequence $(x_n)$  with values in a $\bQ$-algebra $A$ 
 we set
\begin{equation}
\left(x_{\bbf}\right)
\eqdef
\prod_{1\leq \ell \leq k} 
\frac
{x_{f_{i_{\ell}}}(x_{f_{i_{\ell}}}-1)\dots(x_{f_{i_{\ell}}}-i_{\ell}+i_{\ell-1})}
{(i_{\ell}-i_{\ell-1})!}
\end{equation}
(where $i_0=0$ and $i_k=r$).
\end{notas}
We have the following elementary combinatorial lemma.
\begin{lemma}\label{lm:form_expl_prod_eul}
Let $A$ be a $\bQ$-algebra, $E$ a non empty finite set and  $P=1+\sum_{\bn\in \bN^E\setminus \{0\}}a_{\bn}\,\bt^{\,\bn}$
an element of $A[[(t_e)_{e\in E}]]$.
Then for every sequence $(x_n)\in A^{\bN}$ the following relation holds
\begin{multline}
\exp\left(\sum_{n\geq 1}\,a_n\,\log(P(\bt^n)_{e\in E})\right)
\\
=1+\sum_{\bm\in \bN^E\setminus \{0\}}
\left(
\sum_{r\geq 1}
\quad
\sum_{
\substack{
\bbf\in (\bN_{>0})^r
\\~\\
f_1\leq \dots \leq f_r
}
}
(x_{\bbf})                                                                                          
\quad
\sum_{
\substack{
(\bn_1,\dots,\bn_r)\in (\bN^E\setminus \{0\})^r
\\~\\
\sum \bn_i\,f_i=\bm
}
}
\quad
\prod_{i=1}^r a_{\bn_i}
\right)
\,
\bt^{\,\bm}.
\end{multline}
\end{lemma}

For every $\bd\in \bN^I$, denote by $\muxmt(\bd)$ the
element
\begin{equation}
\sum_{r\geq 1}
\quad
\sum_{
\substack{
\bbf\in \bN_{>0}^r
\\~\\
f_1\leq \dots \leq f_r
}
}
(\Phi_{\bbf}(\bP^1))                                                                                          
\quad
\sum_{
\substack{
(\bn_1,\dots,\bn_r)\in (\{0,1\}^I\setminus \{0\})^r
\\~\\
\sum \bn_{\ell}\,f_{\ell}=\bd
}
}
\quad
\prod_{\ell=1}^r \mux^0(\bn_{\ell}).
\end{equation}
Thus, by the above lemma, establishing \eqref{eq:rel:to:be:proved} amounts to proving the
following identities in $\kovark\otimes \bQ$:
\begin{equation}
\forall \bd\in \bN^I,\quad \class{\bP^{\bd}_{X}}
=\sum_{0\leq \bd'\leq \bd}\muxmt(\bd') \class{\bP^{\bd-\bd'}}.
\end{equation}
Except in some particular simple situations, including the case where $X$ is a
projective space, we do not know how to prove these relations in
$\kovark\otimes \bQ$, and we are not even sure that they indeed hold.
Nevertheless, under the additional hypothesis that the characteristic
of the base field is zero,
we are going to explain how to prove a similar relation in the Grothendieck ring of
Chow motives, using a device forged by Denef and Loeser in the
context of their theory of arithmetic motivic integration.

The idea goes basically as follows: when $k$ is finite the relation \eqref{eq:rel:to:be:proved}
certainly holds after specialization by $\#_k$ (this is because
\eqref{eq:sumcardkmuxm} is true !).
We show that the involved equalities may be derived from
`algebraic $d$-cover of formulas', which in turn allows, thanks to
Denef and Loeser's construction, to do `motivic counting' instead of
`classical counting'. This motivic couting leads to a proof of
\eqref{eq:rel:to:be:proved} (in the Grothendieck ring of motives) along
exactly the same way that classical counting allows to proof \eqref{eq:rel:to:be:proved}
after specialization by $\#_k$.

To illustrate the notions of $d$-cover and motivic couting, we begin by a
very basic example, postponing the precise definitions to a little later.
We refer to \cite{Ha:whatis} for a very nice introduction to these concepts.

Let $k$ be a finite field of cardinality $q$, with $q$ odd. The
elementary fact that there are exactly
$\frac{q}{2}$ nonzero squares in $k$ may be seen as follows:
let $f\,:\,\bG_m\to \bG_m$ the morphism $x\mapsto x^2$; then for
every finite extension $L$ of $k$, the morphism $f$ induces a $2$-to-$1$
map from $\bG_m(L)$ onto the set of squares in $\bG_m(L)$, which in turn may be seen as the
set of elements $x$ in $\bA^1(L)$ satisfying the intepretation of the
first order logic formula \begin{equation}\scF\,:\,'(\exists y,\,x=y^2) \wedge (x\neq 0)'.\end{equation}
We say that $f$ induces an algebraic $2$-cover of the formula
$\scF$ by $\bG_m$. From this derives the counting formula
\begin{equation}
\# \scF(L)=\frac{1}{2}\# \bG_m(L)
\end{equation}
where $\scF(L)=\{x\in L,\,(\exists y\in L,\,x=y^2)\wedge x\neq 0\}$.

Now Denef and Loeser's construction allows to deduce from the fact
that $\bG_m$ is a $2$-cover of $\scF$ not only the `classical
counting' result above but far more generally a `motivic counting'
result, that is, 
\begin{equation}\label{eq:motcount}
\class{\scF}=\frac{1}{2}\class{\bG_m}
\end{equation}
where $\class{.}$ denotes the class in the Grothendieck ring of
motives (here the class of our formula $\scF$ may in fact be defined by relation
\eqref{eq:motcount}; in general, one has of course to define the class of an
arbitrary formula in the Grothendieck ring of motives, which is far
from trivial). In fact the more precise hypothesis
under which one is able to deduce \eqref{eq:motcount}
is that the property that $f$ induces a $2$-to-$1$ map from $\bG_m(L)$
onto $\scF(L)$ does not hold only when $L$ is finite but also when $L$
is a so-called pseudo-finite field. 
In one word, pseudo-finite fields are infinite fields
satisfying any model theoretic property which holds for the finite fields.
In the next section we review briefly first order logic formula, pseudo-finite fields
and the construction of Denef and Loeser.

\subsection{Pseudo-finite fields and the virtual motive of a formula}
A pseudo-finite field is a perfect infinite pseudo algebraically closed
field (\ie every geometrically irreducible $k$-variety has a
$k$-point) which has the following property: once  an algebraic closure
$\sep{k}$ of $k$ is fixed, for every $n\geq 1$ there is exactly one
$k$-extension of degree $n$ in $\sep{k}$. 

One can show that every field $k$ admits a pseudo-finite extension.
Pseudo-finite fields share many properties with finite fields. For
example, let $k$ be a pseudo-finite field, $\sep{k}$ an a algebraic closure
and $k_n$ the unique extension of $k$ of degree $n$ in $\sep{k}$.
One can show that $k_n/k$ is cyclic and that $k_n\subset k_m$ if and
only if $n$ divides $m$.

A first order ring formula  with coefficients
in $k$ (which from now will simply be called a $k$-formula)
is a logical formula built from boolean combinations of polynomial
equalities over $k$ and quantifiers; for example
\begin{equation}
'\exists y, \forall
x, x^2+y^2=z^2\,{}',\quad 'x^2+1=0{} ',\quad '\forall z,\,x=y ',\quad 'x^2=x^3+x+1 \wedge
x\neq 0 '\dots
\end{equation}

Let $\varphi$ be a $k$-formula with $n$ free variables. For every $k$-extension $L$, we can define a
subset $\varphi(L)\subset L^n$ (the set of `$L$-points of
$\varphi$') consisting of all the elements in $L^n$
satisfying the interpretation of the formula $\varphi$ in $L^n$. Note that
this defines in fact a functor $(k-\text{extension})\to (\text{Sets})$.
For example if $\varphi='(\exists y,\,x=y^2) \wedge (x\neq 0)'$ then $\varphi(L)$
will be the set of nonzero squares in $L$. Note also that if $\varphi$
is quantifier free, there exists a  constructible subset $F$ of $\bA^n$
such that for every $k$-extension $L$ we have $\varphi(L)=F(L)$.

Let $\varphi$ and $\psi$ two $k$-formulas with free variables
$x_1,\dots,x_n$ and $y_1,\dots,y_m$ respectively. We say that
$\varphi$ and $\psi$ are equivalent if there exists a formula $\theta$
with free variables $x_1,\dots,x_n,y_1,\dots,y_m$ such that for every
pseudo-finite $k$-extension $K$, $\theta(K)$ is the graph of 
a bijection between $\varphi(K)$ and $\psi(K)$. Substituting in the previous
definition 
'$d$-to-$1$
map from $\varphi(K)$ onto $\psi(K)$'
to 
`bijection between $\varphi(K)$ and $\psi(K)$', 
we obtain the definition of 
`$\varphi$ is a $d$-cover of $\psi$'.
For example the formula $'y\neq 0'$ is a $2$-cover of the formula
$'\exists y,\,(x=y^2 \wedge y\neq 0)'$; here the formula $\theta$ is given
by $'y=x^2{}'$.

A very important class of formulas is given by the so-called Galois
formula. Let $X$ be
a normal, affine, irreducible variety defined over $k$, and $\pi\,:\,Y\to X$ be an
unramified Galois cover with group $G$. Let $L$ be a $k$-extension and
$x$ be an element of $X(L)$. Recall that the decomposition subgroups
of $x$ with respect to $\pi$ are the stabilizers of the action of $G$
on the $\Gal(\sep{L}/L)$-orbits of the geometric fiber over $x$. You
may then check that being given a subgroup $D$ of $G$, $x$ admits $D$
as a decomposition subgroup if and only if $x$ lifts to an $L$-point
of $Y/D$ but does not lift to an $L$-point of $D'$ for every strict
subgroup $D'$ of $D$. Hence we see that there exists a $k$-formula
$\varphi_{Y,X,D}$ whose $L$-points, for every $k$-extension $L$, are the $L$-points of $X$ admitting
$D$ as a decomposition subgroup. You may check that the morphism
$Y/D\to X$ makes the formula $\varphi_{Y,Y/D,D}$ a $\frac{\#
  N_G(D)}{\# D}$-cover of the formula $\varphi_{X,Y,D}$. Galois
formulas are the key tool for eliminating quantifiers in the theory of
pseudo-finite fields, see \cite{FrJar:FA} and \cite{Nic:rel:motive}.

Let $\kpff$ denote the Grothendieck ring of the theory of pseudo-finite
fields over $k$: as a group, it is generated by the symbols
$\class{\varphi}$, where $\varphi$ is a $k$-formula, modulo the relations
$\class{\varphi}=\class{\psi}$ whenever $\varphi$ and $\psi$ are
equivalents and the `scissor relations'
$
\class{\varphi\vee\psi}+\class{\varphi\wedge \psi}
=
\class{\varphi}+\class{\psi}
$
whenever $\varphi$ and $\psi$ have the same set of free variables.
We endow it with a ring structure by defining the product of $\class{\varphi}$  by $\class{\psi}$
to be $\class{\varphi\vee \psi}$ if $\varphi$ and $\psi$ have disjoint
sets of free variables (which of course we may always assume, by
considering equivalent formulas). Now we are ready to state the result
of Denef and Loeser. Their motivation for it was the construction of a
motivic incarnation of their theory of arithmetic motivic integration
(see \cite{DeLo:def_sets_motives} and \cite{DeLo:grot_pff}).

Recall that when the field $k$ has characteristic zero, there exists a
unique morphism $\chim\,:\,\kovark\to \kochmk$ which maps the class of
a smooth projective variety to the class of its Chow motive.
\begin{thm}\label{thm:DL}
Let $k$ be a characteristic zero field. There is a unique ring morphism
\begin{equation}\label{eq:chif}
\chif\,:\,\kpff\longto \kochmk\otimes \bQ
\end{equation}
wich maps the class of a quantifier free formula to the
image by $\chim$ of the class of the associated constructible subset and which 
satisifies for every formulas $\varphi,\psi$ such that $\varphi$ is a
$d$-cover of $\psi$ the relation\footnote{In fact Denef and Loeser proved the
  existence and unicity of the
morphism \eqref{eq:chif} under the hypothesis that it satisfies the relation \eqref{eq:rel:d:covering}
only for a particular type of $d$-covers, those induced by Galois
formulas. The fact that such a morphism satisfies \eqref{eq:rel:d:covering}
for every $d$-cover is stated without proof by Hales in \cite{Ha:whatis}, and
proved by Nicaise in \cite{Nic:rel:motive}.}
\begin{equation}\label{eq:rel:d:covering}
\chif(\varphi)=d\,\chif(\psi).
\end{equation}
\end{thm}

Recall that the reader who may not feel comfortable with motives could
as well consider that the Grothendieck ring of motives is nothing else
that the Grothendieck of varieties localized at the class of the
affine line.

We would like to use  Denef-Loeser machinery to
give an other characterization of the image the family $\{\Phi_n(X)\}$
in $\kochmk\otimes \bQ$ by the morphism $\chim$.
By rather straightforward cut-and-paste arguments, we reduce to the case $X$ affine, normal and irreducible. 

What we have in mind is that $\Phi_n(X)$ 
should be the class of a formula such that for 
every pseudo-finite extension $K$ of $k$, the $K$-points of this
formula are in natural $1$-to-$1$ correspondence with the closed points of degree
$n$ of $X_K$. Now closed points of degree $n$ are particular instances
of effective divisors of degree $n$, so they form a
subset of the set of $K$-points of $\sym{X}{n}$ and in fact of
$(\sym{X}{n})^0$, where $(\sym{X}{n})^0$ is the image of the open set
$(X^n)^0$ consisting of those $n$-uples whose coordinates are pairwise
distinct. Now the morphism $(X^n)^0\to (\sym{X}{n})^0$ is plainly an unramified Galois
cover with Galois group $\mfS_n$. And we may describe the subset of
$(\sym{X}{n})^0$ of closed points of degree $n$ exactly as those
elements of $(\sym{X}{n})^0(k)$ having  a decomposition
subgroup cyclic of order $n$ with respect to the above Galois cover. 
There is therefore a Galois formula $\wt{\Phi}_n(X)$ whose $K$-points identifies naturally with
the set of closed points of degree $n$ of $X_K$ for every pseudo-finite
$k$-extension $K$. It is easy to see that its
equivalence class is uniquely determined (that is, does not depend on
the choice of an affine embedding of $X$), and we define $\phinm(X)$ to be
the image of the class of this formula by the morphism $\chif$.
\begin{prop}
Let $X$ be a quasi-projective variety defined over $k$
For every $n$, we have 
\begin{equation}
\chim(\Phi_n(X))=\phinm(X).
\end{equation}
In other words, we have the relation
\begin{equation}
\sum_{n\geq 1}\chim(\sym{X}{n})\,t^n=\prod_{n\geq 1}(1-t^n)^{\,-\phinm(X)}.
\end{equation}
\end{prop}
\begin{proof}
As before, we easily reduce to the case $X$ affine, normal, irreducible.
For every positive integer $r$, $m$ and every $\bbf\in \bN_{>0}^r$,
denote by $\cA_{r,\bbf,m}$ the set
\begin{equation}
\left\{
(n_1,\dots,n_r)\in (\bN_{>0})^r
,\quad
\sum_{\ell=1}^r n_{\ell}\,f_{\ell}=m
\right\}.
\end{equation}
By lemma \ref{lm:form_expl_prod_eul}, we have to show for every positive integer $m$ the relation
\begin{equation}\label{eq:rel_xm}
\chim\left(\class{\sym{X}{m}}\right)
=
\sum_{r\geq 1}
\quad
\sum_{
\substack{
\bbf=(f_1,\dots,f_r)\in \bN_{>0}^r
\\~\\
f_1\leq \dots \leq f_r
}
}
\left(\Phi_{\bbf,\text{\textnormal{\scriptsize{mot}}}}(X)\right)
\,
\# \cA_{r,\bbf,m}.
\end{equation}
The latter formula may be seen as the motivic counterpart of the
following relation, valid over a finite field $k$:
\begin{equation}\label{eq:rel_xm_kfini}
\# \sym{X}{m}(k)
=
\sum_{r\geq 1}
\quad
\sum_{
\substack{
\bbf=(f_1,\dots,f_r)\in \bN_{>0}^r
\\~\\
f_1\leq \dots \leq f_r
}
}
\left(\# X^{(0)}_{\bbf}\right)
\,
\# \cA_{r,\bbf,m}.
\end{equation}
Of course the latter relation follows immediatly from 
the decomposition of the Hasse--Weil zeta function into Euler product,
but the reader may check that it can also be recovered by a direct counting argument.

Now we can apply the strategy described above: we show that this counting argument can be derived from $d$-covers of
formulas, and apply the result of Denef and Loeser to transform the
'classical counting' argument into a 'motivic counting' argument.

Let $m\geq 1$, $r\geq 1$ and
$\bbf\in (\bN_{>0})^r$ 
such that  $f_1\leq \dots \leq f_r$.
We use notations \ref{notas:bbf}.
There is a natural action of 
$\grsym_{\bf }
\eqdef
\underset{\ell=1}{\prod}^{k}
\grsym_{i_{\ell}-i_{\ell-1}}$
on 
$\cA_{r,\bbf,m}$ 
and on
$\underset{i=1}{\overset{r}{\prod}} \left(\sym{X}{f_i}\right)_0$.
Let $Z_{\bbf}$ denote the  $\grsym_{\bbf}$-invariant open set of 
$\underset{i=1}{\overset{r}{\prod}} \sym{X}{f_i}_0$
defined by
\begin{equation}
Z_{\bbf}\eqdef \prod_{1\leq \ell\leq k }(\sym{X}{f_i}_0)^{i_{\ell}-i_{\ell-1}}_0
\end{equation}
(recall that $Y^n_0$ denotes the open set of $Y$ consisting of $n$-uples 
whose coordinates are pairwise distincts, and $\sym{Y}{n}_0$ the image of
$Y^n_0$ by $Y^n\to \sym{Y}{n}$).

Let $\varphi_{\bbf}$ be a formula 
whose set of $K$-points, for every pseudo-finite $k$-extension $K$, 
is $Z_{\bbf}(K)\cap \produ{1\leq i\leq r} \wt{\Phi}_{f_i}(X)(K)$.
One easily check the following relation in $\kpff$:
\begin{equation}\label{eq:psifimi1}
\class{\varphi_{\bbf}}
=
\prod_{1\leq \ell\leq k}
\,
\prod_{j=0}^{i_{\ell}-i_{\ell-1}-1}
\,
\left(\class{\wt{\Phi}_{f_{i_{\ell}}}(X)}-j\right)
=\class{\wt{\Phi}_{\bbf}(X)}.
\end{equation}
Let $\bn\in \cA_{r,\bbf,m}$.
Denote by $\grsym_{\bn}$ the stabilizator of $\bn$
under the action of $\grsym_{\bbf}$, and by 
$
\pi_{\bbf,\bn}
$
the $k$-morphism
$
Z_{\bbf}
\longto 
\sym{X}{m}
$
wich maps the $r$-uple of zero-cycles $(\cC_1,\dots,\cC_r)$
to $\sum_{\ell} n_{\ell}\,\cC_{\ell}$. It factors through
 $Z_{\bbf}/\grsym_{\bn}$.
Let $\psi_{\bbf,\bn}$ be a formula on $\sym{X}{m}$ 
whose set of $K$-points, for every pseudo-finite $k$-extension $K$,
is $\pi_{\bbf,\bn}(\varphi_{\bbf}(K))$.
Thus $\psi_{\bbf,\bn}(K)$
is the set of $K$-rationals zero-cycles which can be written 
$\cC=\sum_{i=1}^r n_{i}\,\cP_{i}$
where $\cP_{i}$ is a closed point of degree $f_{i}$ on $X_K$ and
$\cP_{i}\neq \cP_{j}$ whenever $f_{i}=f_{j}$.
Note that $\pi_{\bbf,\bn}^{-1}(\cC)$ is then a
$\grsym_{\bn}$-orbit. Therefore 
$\varphi_{\bbf}$ is a $\# \grsym_{\bn}$-covering of $\psi_{(\bbf,\bn)}$
and the motivic counting formula \eqref{eq:rel:d:covering} yields
\begin{equation*}
\chif\left(
\symb{\psi_{\bbf,\bn}}
\right)
=
\frac{1}{\#{\grsym_{\bn}}}
\chif\left(
\symb{\varphi_{\bbf}}
\right).
\end{equation*}

Let $\cA^0_{r,\bbf,m}\subset \cA_{r,\bbf,m}$ denote a system of representatives 
of $\cA_{r,\bbf,m}/\grsym_{\Gammaf}$.
We have
\begin{equation}
\sum_{\bn\in \cA^0_{r,\bbf,m}}
\chif\left(
\class{\psi_{\bbf,\bn}}
\right)
=
\Big(
\sum_{\bn\in \cA^0_{r,\bbf,m}}
\frac{1}{\#{\grsym_{\bn}}}
\Big)
\,
\chif\left(
\class{\varphi_{\bbf}}
\right)
=
\frac{\# \cA_{r,\bbf,m}}
{\# {\grsym_{\bbf}}}
\chif\left(
\class{\varphi_{\bbf}}
\right).
\end{equation}
Thus from \eqref{eq:psifimi1} we deduce the relation
\begin{equation}\label{eq:psifimi2}
\sum_{\bn\in \cA^0_{r,\bbf,m}}
\chif\left(
\class{\psi_{\bbf,\bn}}
\right)
=
\left(\Phi_{\bbf,\text{\textnormal{\scriptsize{mot}}}}(X)\right)\,\# \cA_{\bbf,m}.
\end{equation}

Moreover the above description of $\psi_{\bbf,\bn}(K)$ shows
immediatly that every element of 
$\sym{X}{m}(K)$ is in $\psi_{\bbf,\bn}(K)$ for
a unique $\bbf$ and a $\bn\in \cA_{r,\bbf,m}$ unique modulo the action of
$\grsym_{\bbf}$.
Thus the formulas
\begin{equation}
\left(
\psi_{r,\bbf,\bn}
\right)_{
\substack{
r\geq 1, 
\\ \\
\bbf\in \bN_{>0}^r,
\\ \\
f_1\leq \dots \leq f_r,
\\ \\
\bn\in \cA^0_{r,\bbf,m}.
}
}
\end{equation}
form a partition of  $\sym{X}{m}$.
This concludes the proof of the 
relation
\eqref{eq:rel_xm}.

\end{proof}
Now we return to the case of our initial smooth projective toric
variety $X$. In order to show the validity of the relation
\begin{equation}\label{eq:rel:to:be:proved:bis}
\sum_{\bd\in \bN^I}
\muxm(\bd)\,\prod t_i^{d_i}
=
\prod_{n\geq 1}
\left(\sum_{\bn\in  \bN^I}\mux^0(\bn)\,\prod t_i^{\,n\,n_i}\right)^{\,\phinm(\bP^1)}
\end{equation}
in the Grothendieck ring of motives (tensorized with $\bQ$), we apply
exactly the same strategy that in the proof of the preceding proposition.
Since the proof is very similar and the only real novelty consists in 
dealing with more intricate notations, it will not be given in these
notes and we refer to \cite{Bou:prod:eul:mot} for more details.

In the next section, \eqref{eq:rel:to:be:proved} will allow us to show that
the answers to questions \ref{ques:mot} and \ref{ques:fin:bis}(2) are
affirmative for a toric variety $X$,
with a constant $c$ which may be expressed as
\begin{align}\label{eq:def:cmot}
c_{\text{mot}}(X)
&
\eqdef
\frac{\bL^{\dim(X)}}{(1-\bL^{-1})^{\rk(\Pic(X))}}
\sum_{\bd\in \bN^I}
\muxm(\bd) 
\,
\bL^{-\abs{\bd}}
\\
&=
\frac{\bL^{\dim(X)}}{(1-\bL^{-1})^{\rk{\Pic(X)}}}
\prod_{n\geq 1}
\left(\sum_{\bn\in  \bN^I}\mux^0(\bn)\,\bL^{-n\,\abs{\bn}}\right)^{\,\phinm(\bP^1)}.
\end{align}
But an argument analogous to the one used to establish \eqref{eq:mot:rel:toric}
shows that for every $n\geq 1$ we have
\begin{equation}
\sum_{\bn\in  \{0,1\}^I}
\mux^0(\bn)\,\bL^{\,n(\# I-\bn)}
=
(\bL-1)^{\rk(\Pic(X))}\,\psinm{X}
\end{equation}
where $\psinm{X}$ denote the image of $\psin{X}$ by $\chim$. We use
the fact that, contrarily to $\psin{.}$, $\psinm{.}$ is multiplicative, \ie
satisfies $\psinm{Y\times Z}=\psinm{Y}\psinm{Z}$. One can prove this
by motiving counting, see \cite{Bou:prod:eul:mot}. It is also an immediate consequence of the fact,
proved by F.Bittner in \cite{Hei:func_eq},
that the $\lambda$-structure on $\kochmk$ defined by the Hasse--Weil
zeta function is special (see \cite{Gor:Adam}).

Thus the constant $c_{\text{mot}}(X)$ may be rewritten as
\begin{equation}
\frac{\bL^{\dim(X)}}{(1-\bL^{-1})^{\rk{\Pic(X)}}}
\prod_{n\geq 1}
\left(
(1-\bL^{-1})^{\rk(\Pic(X))}\,\frac{\psinm{X}}{\bL^{\,n\dim(X)}}
\right)^{\,\phinm(\bP^1)}
\end{equation}
and the latter may be seen as a motivic analog of 
\eqref{eq:tam} in the case of a toric variety $X$.

\subsection{The error terms}\label{subsec:error}
In this section, we show that questions 
\ref{ques:fin}, \ref{ques:mot} and \ref{ques:fin:bis} have an
affirmative answer for smooth projective toric varieties. 
Having at our disposal the results on the M\"obius inversion function
discussed in the previous sections, it is
essentially a matter of controling the error terms.

Let us begin by the study of 
\begin{equation}\label{eq:limhomubltanticanX}
\lim_{
\substack{
~
\\
y\in \ceff(X)^{\vee}\cap \Pic(X)^{\vee}
\\
\dist(y,\partial \ceff(X)^{\vee})\to +\infty
}
}
\class{\HOM_{U}(\bP^1,X,y)}\,\bL^{-\acc{y}{\antican{X}}}.
\end{equation}
The involved quantity was previously shown to equal
\begin{equation}\label{eq:expr:for:L^yanticanX}
\frac{\bL^{\dim(X)}}{(1-\bL^{-1})^{\rk(\Pic(X))}}
\sum_{0\leq \bd\leq y}
\muxm(\bd) 
\,
\bL^{-\abs{\bd}}
\,
\prod_{i\in I}(1-\bL^{-\acc{y}{D_i}+d_i-1}).
\end{equation}
Let us write the latter expression as 
$n(y)_{\text{main}}+n(y)_{\text{error}}$
where
\begin{equation}
n(y)_{\text{main}}
\eqdef
\frac{\bL^{\dim(X)}}{(1-\bL^{-1})^{\rk(\Pic(X))}}
\sum_{0\leq \bd\leq  y}
\muxm(\bd) 
\,
\bL^{-\abs{\bd}}.
\end{equation}
Recall that the condition  $\bd\leq  y$ may be rewritten 
$d_i\leq \acc{y}{D_i}$ for all $i$, where the
$D_i$'s are the boundary divisors of the toric variety $X$.
And since the $D_i$'s generate $\ceff(X)$, the condition 
\begin{equation}
\dist(y,\partial \ceff(X)^{\vee})\to +\infty
\end{equation}
is equivalent to 
\begin{equation}\label{eq:cond:yDi}
\forall i\in I,\quad \acc{y}{D_i}\to +\infty.
\end{equation}
Thus we have (see remark \ref{rem:conv} and \eqref{eq:def:cmot})
\begin{equation}
\lim_{
\substack{
~
\\
y\in \ceff(X)^{\vee}\cap \Pic(X)^{\vee}
\\
\dist(\bd,\partial \ceff(X)^{\vee})\to +\infty
}
}
n(y)_{\text{main}}
=
\frac{\bL^{\dim(X)}}{(1-\bL^{-1})^{\rk(\Pic(X))}}
\sum_{\bd\in \bN^I}
\muxm(\bd) 
\,
\bL^{-\abs{\bd}}
=
c_{\text{mot}}(X).
\end{equation}

Let us turn to the study of the term 
$n(y)_{\text{error}}$.
From the above expressions and the exclusion-inclusion principle 
it is straightforward that it may be
written as an alternating sum  of the terms
\begin{equation}
n_{J_1,J_2}(y)
\eqdef
\frac{\bL^{\dim(X)-\# J_2}}
{(1-\bL^{-1})^{\rk{\Pic(X)}}}\,
\bL^{-\sum_{i\in J_2}\acc{y}{D_i}}
\sum_{
\substack{\bd\in \bN^I
\\ 
\forall i\in J_1,\quad \acc{y}{D_i}< d_i
\\
\forall i\in J_2,\quad \acc{y}{D_i}\geq d_i
}}
\muxm(\bd)
\bL^{-\sumu{i\notin J_2} d_i}
\end{equation}
where $(J_1,J_2)$ runs over all the pair of subsets of $I$ with $J_2$
non empty and $J_1\cap J_2$ empty. We are going to show
that for every such pair $(J_1,J_2)$ one has
\begin{equation}\label{eq:lim:nJ}
\lim_{
\substack{
~
\\
y\in \ceff(X)^{\vee}\cap \Pic(X)^{\vee}
\\
\dist(y,\partial \ceff(X)^{\vee})\to +\infty
}
}
n_{J_1,J_2}(y)
=
0.
\end{equation}
Note that strictly speaking we should first show that
$n_{J_1,J_2}(y)$ is indeed well defined, since it involves an
infinite summation over $(d_i)\in \bN^{J_1}$ whose convergence is not
a priori clear; the reader may
check that all the necessary arguments are given below.

We will exploit the  fact 
(already used  in section \ref{subsec:sens}) that every polyedral rational cone 
may be written as the
support of a regular fan (the support of a fan is the union of its cones),
see \cite[Théorème 11]{Bry:toric}; the geometric
significance of this result is the existence of equivariant resolution
of singularities for toric varieties.
Nevertheless, the reader may check that we could easily avoid the use of
this result when dealing with \eqref{eq:limhomubltanticanX} (or 
\eqref{eq:limhomubltanticanX:fin}); all that we need to make the
arguments given below work is a finite family of generators of
$\ceff(X)^{\vee}\cap \Pic(X)^{\vee}$. 
But when we will study the
degree zeta functions, it will be important to work with regular cones
(see the remark following \eqref{eq:dec:zj}).

So let $\Delta$ be a regular fan of $\Pic(X)^{\vee}$ whose support is
$\ceff(X)^{\vee}$ (which will be assumed to be fixed for the remainder
of the section).
If $\delta$ is a cone of $\Delta$, let $\delta(1)$ denote the set of its rays, and let $\delta(1)_{J_1,J_2}$ denote 
the subset of $\delta(1)$ consisting of those elements $\rho$ satisfying
\begin{equation}
\forall i\in J_1\cup J_2,\quad \acc{y_{\rho}}{D_i}=0
\end{equation}
(where $y_{\rho}$ denotes the generator of $\Pic(X)^{\vee}\cap \rho$).
In particular, if $\delta(1)_{J_1,J_2}\neq \delta(1)$, we have
\begin{equation}
\lim_{
\substack{
y=\sum_{\rho\in \delta(1)}n_{\rho}y_{\rho},\quad
n_{\rho}\in \bN^{\delta(1)}
\\
~
\\
\forall i\in I,\quad \acc{y}{D_i}\to +\infty
}}
\quad
\sum_{\rho\in \delta(1)\setminus \delta(1)_{J_1,J_2}} n_{\rho}= +\infty.
\end{equation} 
Since the maximal cones of $\Delta$ cover
$\ceff(X)^{\vee}$ and $\Delta$ consists of finitely many regular cones, it is
straightforward to convince oneself that \eqref{eq:lim:nJ} will be
proven once we have established the following: for every maximal cone $\delta\in \Delta$, one has
\begin{equation}\label{eq:lim:nJ:delta}
\lim_{
\substack{
(n_{\rho})\in \bN^{\delta(1)}
\\
~
\\
\sumu{\rho\notin \delta(1)_{J_1,J_2}}
n_{\rho}
\to +\infty
}
}
n_{J_1,J_2}\left(\sum_{\rho\in \delta(1)}n_{\rho} y_{\rho}\right)
=
0.
\end{equation}
Note that since $\delta$ is maximal and $J_2$ is non empty,
$\delta(1)_{J_1,J_2}$ is necessarily a proper subset of $\delta(1)$.
The equality 
\begin{equation}
n_{J_1,J_2}\left(\sum_{\rho\in \delta(1)}n_{\rho} y_{\rho}\right)
=n_{J_1,J_2}\left(\sum_{\rho\notin \delta(1)_{J_1,J_2}}n_{\rho} y_{\rho}\right),
\end{equation}
shows that to prove \eqref{eq:lim:nJ:delta} it suffices to prove that the
series
\begin{equation}\label{eq:series:nrho}
\sum_{
(n_{\rho})\in \bN^{\delta(1)\setminus \delta(1)_{J_1,J_2}}
}
n_{J_1,J_2}\left(\sum_{\rho\notin \delta(1)_{J_1,J_2}}n_{\rho} y_{\rho}\right)
\end{equation}
converges. But up to a constant factor it equals
\begin{equation}
\sum_{
\substack{
(n_{\rho})\in \bN^{\delta(1)\setminus \delta(1)_{J_1,J_2}}
\\
\bd\in \bN^I
\\
\forall i\in J_1,\quad \sum n_{\rho}\acc{y_{\rho}}{D_i}< d_i
\\
\forall i\in J_2,\quad \sum n_{\rho}\acc{y_{\rho}}{D_i}\geq d_i
}
}
\muxm(\bd)\bL^{-\sumu{i\notin J_2} d_i-\sumu{i\in J_2}\sum n_{\rho}\acc{y_{\rho}}{D_i}}
\end{equation}
and thus may be rewritten as
\begin{equation}\label{eq:rewriting:J1:J2}
\sum_{
\bd\in \bN^I
}
\muxm(\bd)
\,
\bL^{-\abs{\bd}} 
\,
R(\bd)
\end{equation}
where
\begin{equation}
R(\bd)
\eqdef
\sum_{\be\in \bN^{J_2}}
\bL^{-\sum_{i\in J_2}e_i}
N(\bd,\be)
\end{equation}
and $N(\bd,\be)$ is the cardinality of the set of elements 
$(n_{\rho})\in \bN^{\delta(1)\setminus \delta(1)_{J_1,J_2}}$
satisfying
\begin{equation}\label{eq:nde:1}
\forall i\in J_2,\quad
\sum n_{\rho}\acc{y_{\rho}}{D_i}=d_i+e_i
\end{equation}
and
\begin{equation}\label{eq:nde:2}
\forall i\in J_1,\quad
\sum n_{\rho}\acc{y_{\rho}}{D_i}<d_i.
\end{equation}
We postpone the (easy) proof of the finiteness of $N(\bd,\be)$
and will get back to it in a minute when dealing with the
finite field case, where we will need an explicit bound for $N(\bd,\be)$.
Once we know that $N(\bd,\be)$ is finite, it is straightforward to
check that the series $R(\bd)$ converges in $\mhat$ to an element
lying in $\fil^0 \mhat$, hence the convergence of
\eqref{eq:rewriting:J1:J2}. This completes the proof 
of the fact that question \ref{ques:fin:bis}(2) has an affirmative answer for
smooth toric varieties (recall however once again that the
characteristic of the base field
has to be assumed to be zero and that we have to work
in the completed Grothendieck ring of motives).

Assuming now that the base field $k$ is a finite field with $q$ elements,
we turn to the study of 
\begin{equation}\label{eq:limhomubltanticanX:fin}
\lim_{
\substack{
y\in \ceff(X)^{\vee}\cap \Pic(X)^{\vee}
\\
\dist(y,\partial \ceff(X)^{\vee})\to +\infty
}
}
\#\HOM_{U}(\bP^1,X,y)(k)\,q^{-\acc{y}{\antican{X}}}
\end{equation}
and show that it equals the constant $c_{\text{fin}}(X)$ defined in 
\eqref{eq:def:cfinX}.
Roughly speaking, the mere thing to do is to 'specialize' the previous
proof by applying the morphism $\#_k$. Of course, to be fully rigorous, one
must be careful with convergence issues. One may check however
that the only extra needed argument  is to show the convergence of the series
\begin{equation}\label{eq:rewriting:J1:J2:bis}
\sum_{
\bd\in \bN^I
}
\abs{\#_k[\muxm(\bd)]}q^{-\abs{\bd}} 
\sum_{\be\in \bN^{J_2}}
q^{-\sum_{i\in J_2}e_i}N(\bd,\be)
\end{equation}
This is why we need an explicit bound for $N(\bd,\be)$.
Let $(n_{\rho})$ be an element of $\bN^{\delta(1)\setminus \delta(1)_{J_1,J_2}}$
satisfying \eqref{eq:nde:1} and \eqref{eq:nde:2}. For $\rho \in \delta(1)\setminus \delta(1)_{J_1,J_2}$,
there exists by definition an element $i\in J_1\cup J_2$ such that
$\acc{y_{\rho}}{D_i}\geq 1$, thus from \eqref{eq:nde:1} and
\eqref{eq:nde:2} we deduce the inequality
\begin{equation}\label{eq:maj:nrho}
n_{\rho}\leq \Max(e_i+d_i+1,d_i)\leq 
\prod_{i\in  J_1}(d_i+1)
\prod_{i\in J_2}(e_i+d_i+1)
\end{equation}
from which we infer
\begin{equation}\label{maj:accyantican}
\acc{\sum n_{\rho}y_{\rho}}{\antican{X}}
\leq 
\Supu{
\substack{\rho\in \Delta
\\
\dim(\rho)=1
}
}
(\acc{y_{\rho}}{\antican{X}}) 
(\sum_{i\in J_2} e_i+d_i+\sum_{i\in J_1} d_i).
\end{equation}
Actually, the latter inequality is not necessary for our current reasoning, but
will be used later when dealing with the anticanonical degree zeta function.

We also deduce from \eqref{eq:maj:nrho} that $N(\bd,\be)$ is finite and bounded from above by 
\begin{equation}
\prod_{i\in J_1\cup J_2}(d_i+1)^{\rk(\Pic(X))}\prod_{i\in J_2}(e_i+1)^{\rk(\Pic(X))}.
\end{equation}

But the series
$\sumu{\be\in \bN^{J_2}}\,\,\produ{i\in J_2}(e_i+1)^{\rk(\Pic(X))}q^{-\sum_{i\in
    J_2}e_i}$ is convergent, and one easily deduces from proposition
\ref{prop:mu} that the series
\begin{equation}
\sum_{
\bd\in \bN^I
}
\abs{\#_k \muxm(\bd)}
\prod_{i\in J_1\cup J_2}(d_i+1)^{\rk(\Pic(X))}
q^{-\abs{\bd}}
\end{equation}
converges too.

Now let us explain how one can deal with questions
\ref{ques:fin} and \ref{ques:mot} in case $X$ is a smooth projective
toric variety, that is, how to study the anticanonical degree zeta
functions. 
The backbone of the argument is basically the same as before.
One first writes the geometric degree zeta function as an alternating sum of
the series 
\begin{equation}
Z_{J_1,J_2}(t)
=
\frac{\bL^{\dim(X)-\# J_2}}
{(1-\bL^{-1})^{\rk{\Pic(X)}}}\,
\sum_{\bd\in \bN^I}
\muxm(\bd)
\bL^{-\sumu{i\notin J_2} d_i}
\sum_{
\substack{
y\in \Pic(X)^{\vee}\cap \ceff(X)^{\vee}
\\
\forall i\in J_1,\quad \acc{y}{D_i}< d'_i
\\
\forall i\in J_2,\quad \acc{y}{D_i}\geq d'_i
}
}
\bL^{\acc{y}{\sum_{i\notin J_2}D_i}}\,t^{y}
\end{equation}
where $(J_1,J_2)$ runs over all the pairs of subsets $(J_1,J_2)$ of $I$ such that
$J_1\cap J_2=\vide$. 
Now 
\begin{equation}
\text{sp}_{\antican{X}}Z_{\vide,\vide}(t)
=c_{\text{mot}}(X).
\text{sp}_{\antican{X}}Z(\Pic(X)^{\vee},\ceff(X)^{\vee})(\bL\,t).
\end{equation}
thus corresponding to the second
term appearing in \eqref{eq:series}, and  the terms
$\text{sp}_{\antican{X}}Z_{J_1,J_2}$ 
for $(J_1,J_2)\neq (\vide,\vide)$
must be shown to be $(\bL^{-1},\rk(\Pic(X))-1)$-controlled (but first,
strictly speaking, they must be shown to be well-defined).
But by decomposing the summation over 
$y\in \Pic(X)^{\vee}\cap \ceff(X)^{\vee}$ as an alternating sum of
summation over $y\in \delta\cap \Pic(X)^{\vee}$, where $\delta$ runs
over the cones of the regular fan $\Delta$, 
one sees easily that $Z_{J_1,J_2}$
may be written as an alternating sum of the terms
\begin{equation}\label{eq:dec:zj}
\left(
\sum_{
y\in \sumu{\rho\in  \delta(1)\setminus \delta(1)_{J_1,J_2}}\!\!\bN \,y_{\rho}
}
n_{J_1,J_2}(y)
\bL^{\acc{y}{\antican{X}}}
t^y
\right)
\prod_{\rho\in \delta(1)_{J_1,J_2}}
\frac{
1}
{1-\bL^{\acc{y_{\rho}}{\sum_{i\in J_2}D_i}}\,
\bL^{\acc{y_{\rho}}{\antican{X}}}\,t^{y_{\rho}}}.
\end{equation}
Here it is important to work with a regular cone $\delta$, to ensure that every
element of $\delta$ may be written in a unique way as the sum of 
an element of 
$\sum_{\rho\in  \delta(1)\setminus \delta(1)_{J_1,J_2}}\bN y_{\rho}$
and an element of 
$\sum_{\rho\in  \delta(1)_{J_1,J_2}}\bN y_{\rho}$.
Note that the condition $(J_1,J_2)\neq (\vide,\vide)$ implies that the
cardinality of $\delta(1)_{J_1,J_2}$ is less than $\rk(\Pic(X))$.
Thus to show that $\text{sp}_{\antican{X}}Z_{J_1,J_2}$ is 
$(\bL^{-1},\rk(\Pic(X))-1)$-controlled, it remains to show that the
series
\begin{equation}
\text{sp}_{\antican{X}}\left(
\sum_{
y\in \sumu{\rho\in  \delta(1)\setminus \delta(1)_{J_1,J_2}}\!\!\bN \,y_{\rho}
}
n_{J_1,J_2}(y)\bL^{\acc{y}{\antican{X}}}t^y\right)
\end{equation}
converges at $t=\bL^{-1}$. But the latter point is nothing else 
than the already established convergence of \eqref{eq:series:nrho}.
Thus the answer to question \ref{ques:mot} is positive for a smooth
toric variety (over a field of characteristic zero, after specialization to the Grothendieck ring of motives).

Concerning the classical anticanonical degree zeta function, we are
going to show that \eqref{eq:diff} 
\begin{equation}
Z^{\#_{k}}_U(X,\antican{X},t)-c_{\text{fin}}(X).Z(\Pic(X)^{\vee},\ceff(X)^{\vee},\class{\antican{X}},q\,t)
\end{equation}
is strongly 
$(q^{-1},\rk(\Pic(X))-1)$-controlled. 
Using the same decomposition
argument as before (formally, we just apply the morphism $\#_k$ to the
previously used decomposition of the geometric degree zeta function, though as
always there are convergence issues to be taken into account), one reduces to proving that there exists a
positive real number $\epsilon$ such that
\begin{equation}
\text{sp}_{\antican{X}}\left(
\sum_{
y\in \sumu{\rho\in \delta(1)\setminus \delta(1)_{J_1,J_2}}
\!\!\!\bN\,y_{\rho}
}
\#_k [n_{J_1,J_2}(y)]\,q^{\acc{y}{\antican{X}}}t^y
\right)
\end{equation}
converges absolutely\footnote{The reader will have of course noticed
  that strictly speaking $\#_k [n_{J_1,J_2}(y)]$ does not make sense;
 it is to be taken in a formal sense and actually designates
\[
\frac{q^{\dim(X)-\# J_2}}
{(1-q^{-1})^{\rk{\Pic(X)}}}\,
q^{-\sum_{i\in J_2}\acc{y}{D_i}}
\sum_{
\substack{\bd\in \bN^I
\\ 
\forall i\in J_1,\quad \acc{y}{D_i}< d_i
\\
\forall i\in J_2,\quad \acc{y}{D_i}\geq d_i
}}
\#_k[\muxm(\bd)]
q^{-\sumu{i\notin J_2} d_i}.
\]
}. 
For $\eta\leq 0$ this convergence follows directly from the (previously discussed)
convergence of the series
\begin{equation}\label{eq:previously:studied:series}
\sum_{
y\in \sumu{\rho\in \delta(1)\setminus \delta(1)_{J_1,J_2}}
\!\!\!\bN\,y_{\rho}
}
\#_k [n_{J_1,J_2}(y)].
\end{equation}
But now we have to show the convergence of
\begin{equation}\label{eq:currently:studied:series}
\sum_{
y\in \sumu{\rho\in \delta(1)\setminus \delta(1)_{J_1,J_2}}
\!\!\!\bN\,y_{\rho}
}
\abs{\#_k [n_{J_1,J_2}(y)]}q^{-\eta \acc{y}{\antican{X}}}
\end{equation}
for evey sufficiently small positive $\eta$.
This is here that  \eqref{maj:accyantican} is useful; 
by a reasoning analogous to the one
used to establish the convergence of 
\eqref{eq:previously:studied:series}, we see that 
\eqref{eq:currently:studied:series}
is bounded
from above by 
\begin{equation}\label{eq:expr:prod:ser}
\left(
\sum_{
\bd\in \bN^I
}
\abs{\#_k \muxm(\bd)}
\prod_{i\in J_1\cup J_2}(1+d_i)
q^{-(1-\eta\,M)\abs{\bd}} 
\right)
\left(
\sum_{\be\in \bN^{J_2}}
\prod_{i\in J_2}(1+e_i)
q^{-(1-\eta\,M)\sum_{i\in J_2}e_i}
\right)
\end{equation}
where we have set 
$M\eqdef \Supu{\substack{\rho\in \Delta
\\
\dim(\rho)=1
}
}
\acc{y_{\rho}}{\antican{X}}
$; and the two series appearing in \eqref{eq:expr:prod:ser} are obviously convergent for $\eta$
sufficiently small (again, we use the properties of the M\"obius
function $\mux$ described in proposition \ref{prop:mu}).

To finish the section, we are going to show on an example 
why one could not expect for a general toric variety the existence of
the limits
\begin{equation}
\lim_{
\substack{
~
\\
y\in \Pic(X)^{\vee}\cap \ceff(X)^{\vee}
\\
~
\\
\acc{y}{\antican{X}}\to +\infty
}
}
\class{\HOM_{U}(\bP^1,X,y)}\,\bL^{-\acc{y}{\antican{X}}}
\end{equation}
and (when $k$ is a finite field with $q$ elements)
\begin{equation}
\lim_{
\substack{
~
\\
y\in \Pic(X)^{\vee}\cap \ceff(X)^{\vee}
\\
~
\\
\acc{y}{\antican{X}}\to +\infty
}
}
\#\HOM_{U}(\bP^1,X,y)(k)\,q^{-\acc{y}{\antican{X}}}.
\end{equation}

We take for $X$ the projective plane $\bP^2$ blown-up at $(0:0:1)$. 
We denote by $D_0$, $D_1$, $D_2$ the strict transform of the
coordinate hyperplane and by $E$ the exceptional divisor. A toric
structure on $X$, as well as the corresponding fan, were described in  
section \ref{subsec:tor:gemo}.
We denote by $(D_0^{\vee},E^{\vee})$ the dual basis of the basis
$(D_0,E)$ of $\Pic(X)$ and use it to identify $\Pic(X)^{\vee}$ with
$\bZ^2$. The coordinate of $y\in \Pic(X)^{\vee}$ in this basis will be
denoted by $(y_0,y_E)$.

A very pleasing feature of $X$ is that the
M\"obius fonction $\muxm$ is explicitely computable: let us define the
function $\mum\,:\,\bN\to \kovark$ by the relation 
\begin{equation}
\sum \mum(d)t^d=\frac{1}{\ZHWm(\bP^1,t)}.
\end{equation}
Thus one immediatly computes
\begin{equation}\label{eq:expr:mum}
\mum(0)=1,\quad \mum(1)=-(1+\bL),\quad \mum(2)=\bL,
\quad 
\forall d\geq 3,\quad \mum(d)=0.
\end{equation}
Moreover one shows (see \cite{Bou:prod:eul:mot}) 
\begin{multline}
\forall (d_0,d_1,d_2,d_E)\in \bN^4,\quad 
\\
\muxm(d_0,d_1,d_2,d_E)
=
\left\{
\begin{array}{rl}
0&\text{ if }d_0\neq d_1\text{ or }d_2\neq d_E\\
\mum(d_0)\mum(d_E)&\text{ otherwise.}
\end{array}
\right.
\end{multline}
From this we see that for 
$y=(y_0,y_E)\in\Pic(X)^{\vee}\cap \ceff(X)^{\vee}$ 
the quantity 
\begin{equation}
\bL^{-\acc{y}{\antican{X}}}\class{\HOM_{U}(\bP^1,X,y)}
\end{equation}
(recall the expression \eqref{eq:expr:for:L^yanticanX})
may be rewritten as
\begin{multline}\label{eq:form:p2:blownup:1:point}
\frac{\bL^{2}}{(1-\bL^{-1})^{2}}
\sum_{
\substack{
0\leq  d_E \leq \Min(2,y_E)
\\
0\leq  d_0 \leq \Min(2,y_0)
}
}
\mum(d_0)\mum(d_E)\,\bL^{-2\,d_E-2\,d_1}
(1-\bL^{-1+d_E-y_E})
\\
\times
(1-\bL^{-1+d_E-y_E-y_0})(1-\bL^{-1+d_0-y_0})^2.
\end{multline}
thus allowing, using \eqref{eq:expr:mum}, to give a completely
explicit expression of 
\begin{equation}
(1-\bL^{-1})^2\,\bL^{-\acc{y}{\antican{X}}}\class{\HOM_{U}(\bP^1,X,y)}
\end{equation}
as an element of $\bZ[\bL^{-1}]$.

We have
\begin{equation}
c_{\text{mot}}(X)=\frac{\bL^{2}}{(1-\bL^{-1})^{2}}\left(\sum_{d\in \bN}\mum(d)\bL^{-2d}\right)^2
=\bL^2(1-\bL^{-2})^2
\end{equation}
but one easily checks using the above expression that 
\begin{equation}
\lim_{n\to +\infty}
\bL^{-\acc{n\,E^{\vee}}{\antican{X}}}\class{\HOM_{U}(\bP^1,X,n\,E^{\vee})}
=
\bL^{2}(1-\bL^{-1})(1-\bL^{-2})\neq c_{\text{mot}}(X)
\end{equation}
(for the inequality, see remark \ref{rem:L:trans}).

If $k$ is a finite field with $q$ elements, one checks similarly that
\begin{equation}
\lim_{n\to +\infty}
q^{-\acc{n\,E^{\vee}}{\antican{X}}}\# \HOM_{U}(\bP^1,X,n\,E^{\vee})(k)
=
q^{2}(1-q^{-1})(1-q^{-2})\neq c_{\text{fin}}(X)=q^2(1-q^{-2})^2
\end{equation}
Note however that one can show, using again \eqref{eq:form:p2:blownup:1:point}
that one has
\begin{equation}
\lim_{
\substack{
~
\\
y\in \Pic(X)^{\vee}\cap \ceff(X)^{\vee}
\\
~
\\
\acc{y}{\antican{X}}\to +\infty
\\
\acc{y}{E}\geq 2
\\
\acc{y}{D_0}\geq 2
}
}
\class{\HOM_{U}(\bP^1,X,y)}\,\bL^{-\acc{y}{\antican{X}}}=c_{\text{mot}}(X)
\end{equation}
and the analogous statement if $k$ is a finite field.

\section{The general case}

In this section, we want to explain how the use of homogeneous
coordinates in the study of the degree zeta function of a 
smooth projective toric variety might be generalized to other
varieties. First of all of course we have to explain 
the  notion of homogeneous coordinate rings for a non toric variety.
Motivated by the work of Cox in the toric case, 
it has been intensively studied during
the last ten years. The terms \termin{Cox rings} 
or \termin{total coordinate rings} are
often found in the literature to designate homogeneous coordinate
rings\footnote{Though `Cox ring' is probably the most commonly used,
  I will stick to the terminology `homogeneous coordinate ring'
  which I find more appealing, even though there might be confusion
  with the homogeneous coordinate ring associated to one particular
  projective embedding. Note that what is called an homogeneous
  coordinate ring in \cite{BerHau:hom} is in fact the ring we discuss
  here equipped with an extra structure}.
The topic is tightly connected with the so-called notion of universal torsors,
introduced by Colliot-Thélène and Sansuc in the 1970's in order to
study weak approximation and Hasse principle on rational varieties (see \eg \cite{CTS:desc,Sko:book}).
One owes to Salberger the idea of using universal torsors in the context of Manin's conjecture
on rational points of bounded height. He showed in \cite{Sal:tammes} that this approach
was indeed fructuous for toric varieties (defined over $\bQ$) and 
the first non toric example of a succesful application of the method
is due to de la Bretèche (\cite{dlB:duke}).
Since then, the use of universal torsors/homogeneous coordinate rings has allowed to settle
the arithmetic version of Manin's conjectures for a certain number of non toric varieties
(especially in dimension $2$), see \eg \cite{Bro:manin_dim_2}.

In the arithmetic setting, the use of homogeneous coordinate rings reduces the
counting of rational points of bounded height to the counting of integral
points of an affine space satisfying certain algebraic relations, coprimality conditions
and norm inequalities.
In the geometric setting, we will explain below how it similarly reduces the
counting of morphism $\courbe\to X$ of bounded degree to the counting of global
sections of line bundles of $\courbe$ satisfying certain algebraic relations,
non degeneracy conditions, and degree conditions. This will generalize
the case of a toric variety $X$, for which there are indeed
{\em no} algebraic relations. For the sake of simplicity we will
limit ourselves to the case $\courbe=\bP^1$.

For more about homogeneous coordinate rings and examples of
computations, see \eg \cite{BerHau:hom,BerHau:Cox,Bri:TCR,Has:eq:ut:cox:rings,HaTs:UTCR}.

\subsection{A brief survey of the theory homogeneous coordinate rings}

Let $k$ be a perfect field and $X$ be a smooth projective variety. 
We hereby assume that the Picard group of $X$ coincides with its
geometric Picard group and that it is free of finite rank 
(the theory of homogeneous coordinate rings can be developed in 
a more general context, see \eg \cite{ElKaWa,BerHau:hom}).

Very roughly, the idea behind the theory of homogeneous coordinate
rings is that instead of working with a particular choice of
coordinates coming from a morphism from $X$ to a projective space, which
in turn corresponds to a subspace of the space of global sections of a
particular invertible sheaf on $X$, we could as well work
with the spaces of global sections
of all the invertible sheaves on $X$ considered simultaneously. 

Let $\cL_1,\dots,\cL_r$ be a basis of $\Pic(X)$. 
We define the
homogeneous coordinate ring of $X$ by 
\begin{equation}
\HCR(X)
\eqdef 
\bigoplusu{\bn\in \bZ^r}
H^0(X,\cL_1^{n_1}\otimes\dots \otimes \cL_r^{n_r}).
\end{equation}
This is a $k$-algebra naturally graded by $\Pic(X)$: 
just impose that $H^0(X,\cL_1^{n_1}\otimes\dots \otimes \cL_r^{n_r})$
is homogeneous of degree the class of $\cL_1^{n_1}\otimes\dots \otimes \cL_r^{n_r}$.
The degree of the nonzero graded pieces are precisely the effective
classes in $\Pic(X)$.
The definition depends of course on a particular choice of a basis of $\Pic(X)$.
Nevertheless, one can easily show that two different choices give
rise to isomorphic $\Pic(X)$-graded $k$-algebras.
\begin{ex}
Let $n\geq 4$ be an integer and $X\subset \bP^n_k$ be a smooth
projective hypersurface of degree $d\leq n+1$; then $\HCR(X)$ is the
homogeneous coordinate ring of $X$ in the classical sense, that is,
the affine coordinate ring of the cone over $X$ in $\bA^{n+1}_k$.
\end{ex}
\begin{ex}[Cox]
Let $X$ be a smooth toric variety and let
$\{D_i\}_{i\in I}$ be the irreducible divisors of the boundary. For
$i\in I$ let $s_i$ be the canonical section of $\str{X}(D_i)$. Then
the $s_i$'s generate $\HCR(X)$, and there are no nontrivial relation
between them, thus $\HCR(X)$ is a polynomial ring in $\#I$ variables
in this case (this is essentially the content of remark \ref{rem:hcr:toric}).
\end{ex}
\begin{ex}[Hasset]\label{ex:p2blownup3pts}
Let $X$ be the projective plane blown up at three collinear points, $D_0$ be the strict transform of
the line $L$ joining the points, $D_1$, $D_2$ and $D_3$ the exceptional
divisors and $D_4$, $D_5$, and $D_6$ the strict transform of the lines
joining a point not lying on $L$ to the blown up points. Let $s_i$ be the
canonical section of $\str{X}(D_i)$. Then one can show that the $s_i$
generate $\HCR(X)$, and that the kernel of the morphism $k[X_i]\to
\HCR(X)$ mapping $X_i$ to $s_i$ is generated (after a suitable
normalization of the $s_i$'s) by $X_1\,X_4+X_2\,X_5+X_3\,X_6$
(see \cite{Has:eq:ut:cox:rings} and \cite{Der:sdp:ut:hyp}).
\end{ex}
\begin{ex}[Skorobogatov]
Let $X$ be the projective plane blown up at four points $(P_i)_{1\leq
  i\leq 4}$ in general position; 
then $\HCR(X)$ may be identified with the
homogeneous coordinate rings of the Plücker embedding of the
Grassmannian variety $Gr(3,5)$ in $\bP(\Lambda^3 k^5)\isom \bP^{10}_k$.
More explicitely,
let $(E_i)_{1\leq i\leq 4}$ be the
exceptional divisors and $(L_{i,j})_{1\leq i<j\leq 4}$ be the strict
transform of the lines joining the $P_i$'s; let $z_{i,5}$ be the
canonical section of $E_i$ and $z_{i,j}$ be the canonical section of
$L_{i,j}$; then the morphism $k[X_{i,j}]\to \HCR(X)$ mapping
$X_{i,j}$ to $s_{i,j}$ is surjective with kernel generated by
the five elements
\begin{align}
X_{1,2}X_{3,4}-X_{1,3}X_{2,4}+X_{1,4}X_{2,3},\\
X_{1,2}X_{3,5}-X_{1,3}X_{2,5}+X_{1,5}X_{2,3},\\
X_{1,2}X_{4,5}-X_{1,4}X_{2,5}+X_{1,5}X_{2,4},\\
X_{1,3}X_{4,5}-X_{1,4}X_{3,5}+X_{1,5}X_{3,4},\\
\text{and}\quad X_{2,3}X_{4,5}-X_{2,4}X_{3,5}+X_{2,5}X_{3,4}.
\end{align}
\end{ex}
\begin{ex}[Batyrev, Derenthal, Laface, Popov, Stillman, Sturmfels, Testa, Varilly-Alvarado, Velasco, Xu]
Let $1\leq r\leq 4$ be an integer an $X_r$ be a smooth del Pezzo surface of degree $r$;
recall that it is isomorphic to the projective plane blown up at $9-r$
points in general position. Then $\HCR(X_r)$ is generated by the
sections of the $(-1)$-curves, and the ideal of relations is generated
by quadratic relations\footnote{Note that for $6\leq r\leq 9$, $X_r$
  is toric and in the case $r=5$ we have a similar result by the previous example}.
\end{ex}

In all the above examples, the homogeneous coordinate ring happens
to be finitely generated. 
The relevance of the property of finite
generation of the homogeneous coordinate ring was stressed by Hu and
Keel in the context of Mori theory. In \cite{hukeel:mori}, they call
those varieties with finitely
generated homogeneous coordinate rings {\em Mori dream spaces},
showing in particular that they behave very well with respect to the
minimal model program.

The question of deciding whether the homogeneous coordinate ring of a
variety is finitely generated is difficult. 
A recent and very deep result of Birkar, Cascini, Hacon and McKernan is that the
homogeneous coordinate ring of a Fano variety is finitely generated (\cite[Corollary 1.3.2.]{BCHM}).
On a surface, it is easy to show that a necessary condition for finite
generation is that there are only finitely many curves with negative self-intersection.

Another difficult issue is to compute explicitely
generators and relations for the homogeneous coordinate ring. 
Such an explicit expression is a priori required for applications in the
context of Manin's conjectures.

\subsection{Homogeneous coordinate rings and universal torsors}

In the following, we will denote by $X$ a smooth projective variety defined over a perfect field
$k$ such that the Picard group is free of finite rank, coincide with
the geometric Picard group, and such that $\HCR(X)$ is generated by 
a finite number of sections invariant under the action of the absolute
Galois group (the reader may assume that $k$ is algebraically closed
if he likes). Under these assumptions, one can construct a
$\TNS{X}$-torsor over $X$ with properties generalizing the one of the
torsor constructed in subsection \ref{subsec:hom:cor:tor} when $X$ is
toric (recall that $\TNS{X}=\Hom(\Pic(X),\bG_m)\isom \bG_m^{\rk(\Pic(X))}$).

A first version of the result is due to Hu and Keel.
\begin{thm}[Hu,Keel]\label{thm:hukeel}
Let $D$ be an ample class in $\Pic(X)$. It corresponds to a character
of $\TNS{X}$, hence to a $\TNS{X}$-linearization of the trivial bundle
on $\Spec(\HCR(X))$. The GIT quotient of the open set 
$\Spec(\HCR(X))^{\text{ss}}$
of semi-stable points by the action of $\TNS{X}$ is a geometric
quotient isomorphic to $X$.
\end{thm}

We refer to \cite[Proposition 2.9]{hukeel:mori} for a proof of this theorem.
We will not review here the tools of Geometric Invariant Theory
necessary to understand the statement and its proof (see \eg
\cite{MFK:GIT,Dol:GIT}). But following Hasset and Tschinkel, we are
going to explain, 
by a GIT-free approach,
why the geometric quotient of theorem \ref{thm:hukeel} is the so-called
universal torsor over $X$.
First we will review some basic properties of torsors under algebraic
tori.

\subsubsection{Torsors under split algebraic tori}
We still restrict ourselves to the case of split tori, which
allows us to work only with the crude Zariski topology (otherwise, finer
Grothendieck topologies, \eg étale topology, would be needed). So let $T$ be a
split algebraic torus and $X$ a variety. The trivial $X$-torsor under $T$
is the $T$-equivariant morphism $\pr_X\,:\,X\times T\to X$ where $T$
acts trivially on $X$ and by translation on itself.
An $X$-torsor under $T$ is the datum of a variety 
$\ecT$ equipped with an algebraic action of $T$ 
and a morphism $\pi\,:\,\ecT\to X$ which is locally isomorphic to the trivial
torsor, that is to say 
there exists a finite open covering $(U_i)_{i\in I}$ of $X$ 
and $T$-equivariant $X$-isomorphisms 
$\psi_i\,:\,\pi^{-1}(U_i)\to U_i\times T$. 
By abuse of terminology, 
we will often say that the variety $\ecT$ is an
$X$-torsor under $T$.

For $i,j\in I$, the morphism
\begin{equation}
(\psi_j\circ \psi_i^{-1})_{|(U_i\cap U_j)\times T}:\,U_i\cap U_j\times
T\to U_i\cap U_j\times T
\end{equation}
induces a morphism $\lambda_{i,j}\,:\,U_i\cap U_j\to T$, that is, an
element of $T(U_i\cap U_j)$. 
Recall that the latter has a natural group
structure, for which it is isomorphic to
$\left(\Gamma(U_i\cap U_j)^{\inv}\right)^{\dim(T)}$. 
It is
straightforward to check that the $\{\lambda_{i,j}\}_{i,j\in I}$
satisfy the cocycle conditions, that is $\lambda_{j,k}\lambda_{i,j}=\lambda_{i,j}$
and $\lambda_{i,i}=1$.

Conversely, the datum of a finite open covering $\{U_i\}_{i\in I}$ of
$X$ and a family $\{\lambda_{i,j}\in T(U_i\cap U_j)\}_{i,j\in I}$ satisfying the cocycle conditions
determines an $X$-torsor under $T$: just glue the trivial torsors $U_i\times
T\to U_i$ along the $(U_i\cap U_j)\times T\to U_i\cap U_j$ using the $\lambda_{i,j}$ as transition morphisms. 

Two $X$-torsors under $T$ are said to be isomorphic if there exists a
$T$-equivariant $X$-isomorphism between them.
Denote by $H^1(X,T)$ the set of isomorphism classes of $X$-torsors
under $T$. It is naturally equipped with an abelian group structure:
if two torsors are represented by cocycles $(\{U_i\},\{\lambda_{i,j}\})$
and $(\{U_i\},\{\lambda'_{i,j}\})$ respectively, the class
of their 
product is represented by the cocycle $(\{U_i\},\{\lambda_{i,j}\lambda'_{i,j}\})$.
The unit element is the class of the trivial torsor.

\subsubsection{Torsors under $\bG_m$}\label{subsec:torsors:under:Gm}

In case $T=\bG_m$, the datum of an isomorphism class of
cocycle  $(\{U_i\},\{\lambda_{i,j}\in \Gamma(U_i\cap U_j)^{\inv}\})$ is
equivalent to the datum of an isomorphism class of invertible sheaf on
$X$; in other words we have a natural bijection $H^1(X,\bG_m)\isom
\Pic(X)$ which is clearly seen to be a group isomorphism. 
If $\cT\to X$ is a torsor under $\bG_m$ the corresponding
class of $\Pic(X)$ will be called the type of the torsor (\cf below
for a generalization). 
The pull-back of a torsor under $\bG_m$ of type $\cL$ by a morphism
$\varphi\,:\,Y\to X$ is easily seen to be a $Y$-torsor under $\bG_m$
of type $\varphi^{\ast}\cL$.

Let $\cL$ be an invertible sheaf on $X$ and
$V(\cL)\eqdef \SPEC(\oplusu{n\in \bN}\cL^{n})$.
The natural affine morphism $V(\cL)\to X$ is a line bundle on $X$; we
denote by $\mathbf{0}_{V(\cL)}$ its zero section.
Then a representant of the class of  $X$-torsors under $\bG_m$ of type $\cL$
is the morphism $V(\cL)^{\inv}\eqdef V(\cL)\setminus \mathbf{0}_{V(\cL)}\to X$; note
that $V(\cL)^{\inv}$ is naturally isomorphic
to $\SPEC(\oplusu{n\in \bZ}\cL^{n})$.

When $\cL$ is ample,
we explain now how to construct a torsor under $\bG_m$ with type
$\cL$ as an open subset of an affine variety.
We consider the $\bZ$-graded $k$-algebras
\begin{equation}
R(X,\cL)=\oplusu{n\in \bN}H^0(X,\cL^{  n})
\end{equation}
(which is finitely generated since $\cL$ is ample)
and the associated affine scheme
\begin{equation}
\cC(X,\cL)\eqdef \Spec(R(X,\cL)).
\end{equation}
We denote by $0_{X,\cL}$ the closed point of $\cC(X,\cL)$
defined by the ideal 
\begin{equation}
R(X,\cL)^+\eqdef \oplusu{n\geq 1}H^0(X,\cL^{ n}).
\end{equation}
There is a natural morphism
\begin{equation}
\pi_{\cL}\,:\,V(\cL)=\SPEC(\oplusu{n\in \bN}\cL^{ n})\to \cC(X,\cL).
\end{equation}
Since $\cL$ is ample, one checks, using that $\cL^n$ is
generated by global sections for $n$ large enough that the set
theoretic inverse image of
$0_{X,\cL}$ is the zero section $\mathbf{0}_{V(\cL)}$.
Hence $\pi_{\cL}$ induces a morphism
\begin{equation}
\pi'_{\cL}\,
:\,
V(\cL)^{\inv}
\to
\cC(X,\cL)\setminus 0_{X,\cL}.
\end{equation}
Assume moreover that $\cL$ is very ample. 
Then $\pi'_{\cL}$ is an
isomorphism. This may be seen by using the cartesian diagram
\begin{equation}
\xymatrix{
\cC(X,\cL)\setminus 0_{X,\cL}
\ar@{^{(}->}[r]\ar[d]
&
\bA(X,\cL)\setminus \{0\}
\ar[d]
\\
X\isom \Proj(R(X,\cL))
\ar@{^{(}->}[r]^>>>>>>>>{\iota}&\bP(X,\cL)
}
\end{equation}
where $\bP(X,\cL)$
(respectively $\bA(X,\cL)$) denotes the $\Proj$ (respectively the
$\Spec$ of the symmetric algebra) of $H^0(X,\cL)$, 
and both horizontal
arrows are closed immersions.
The left vertical arrow
is the pullback 
of the $\bG_m$-torsor 
$\bA(X,\cL)\setminus \{0\}\to \bP(X,\cL)$ 
which is of type $\str{\bP(X,\cL)}(1)$: thus its type is
$\iota^{\ast}\str{\bP(X,\cL)}(1)=\cL$.

Moreover $\pi'_{\cL}$
is still an isomorphism when $\cL$ is only assumed to be ample. Indeed, let $d\geq 1$
such that $\cL^{\otimes d}$ is very ample. 
We have a commutative diagram
\begin{equation}
\xymatrix{
\SPEC(\oplusu{n\in \bZ}\cL^{n})
\ar[r]
\ar[d]^{\pi'_{\cL}}
&
\SPEC(\oplusu{n\in \bZ}\cL^{d\,n})
\ar[d]^{\pi'_{\cL^{d}}}
\\
\cC(X,\cL)\setminus 0_{X,\cL}
\ar[r] ^{\iota}&
\cC(X,\cL^{d})\setminus 0_{X,\cL^d}
}
\end{equation}
The upper horizontal arrow is induced by the inclusion of
$\str{X}$-algebra 
$\oplusu{n\in \bZ}\cL^{d\,n}\subset \oplusu{n\in \bZ}\cL^{n}$
and is thus a finite morphism.
Since $\pi'_{\cL^{d}}$ is an isomorphism, $\pi'_{\cL}$ is finite,
hence affine. But by the very definition of the affine scheme
$\cC(X,\cL)$, one has
$(\pi_{\cL})_{\ast}\,\str{V(\cL)}=\str{\cC(X,\cL)}$, 
hence $(\pi'_{\cL})_{\ast}\,\str{V(\cL)^{\inv}}=\str{\cC(X,\cL)\setminus 0_{X,\cL}}$, 
and $\pi'_{\cL}$ is an isomorphism.

\subsubsection{Type and universal torsors}

Let $T$ be a split algebraic torus, $\pi=(\{U_i\},\{\lambda_{i,j}\})$ an
$X$-torsor under $T$, and
$\varphi\,:\,T\to T'$ a morphism of algebraic torus. Then 
$(\{U_i\},\{\varphi(\lambda_{i,j})\})$ is an $X$-torsor under $T'$,
denoted by $\varphi_{\ast}\pi$.

To an (isomorphism class) of $X$-torsor under $T$ one associates its
type $\tau(\cT)$, 
which is an element of $\Hom(\carac{T},\Pic(X))$ defined as follows:  let $\chi\in \carac{T}$; then
$\chi_{\ast}\cT$ is an $X$-torsor under $\bG_m$, hence determinates a
class in $\Pic(X)$, which is by definition $\tau(\cT)(\chi)$. It is
easy to check that the map $\cT\to \tau(\cT)$ induces 
an isomorphism $H^1(X,T)\isom \Hom(\carac{T},\Pic(X))$ (using the fact
that $T$ is isomorphic to $\bG_m^r$, one reduces to the case $T=\bG_m$).

Now assume that $\Pic(X)$ is free of finite rank (with a trivial
Galois action). A universal $X$-torsor is an $X$-torsor under
$\TNS{X}$ whose type is $\Id_{\Pic(X)}\in \End(\Pic(X))$. 
Note that there is only one isomorphism class of universal torsors
over $X$. 

Let $\pi\,:\,\ecT\to X$ be a universal torsor.
Being given an arbitrary torus $T$ and 
a torsor $\pi'\,:\,\ecT\to X$ under $T$, one
sees immediatly that there exists a unique morphism of algebraic group
$\varphi\,:\,\TNS{X}\to T$ such that $\varphi_{\ast}\ecT$ 
and $\ecT'$ are isomorphic: $\varphi$ is the dual morphism of $\tau(\cT')\in \Hom(\carac{T},\Pic(X))$.
Thus, every $X$-torsor under a torus can be recovered from a universal
torsor
and, in some sense, universal torsors are the most interesting torsors
among the $X$-torsors under tori, those which are maximal in terms of
complexity; hence lifting objects from $X$ to a universal torsor
should reveal itself interesting.

Choose a basis $\cL_1,\dots,\cL_r$ of $\Pic(X)$. Let $\cL_{\ell}$ be
described by the cocycle $(\{U_i\},\{\lambda^\ell_{i,j}\})$.
A representant of the class of universal torsors may be described,
according to taste,
as  
\begin{equation}
V(\cL_1)^{\inv}\times_X \dots\times_X V(\cL_r)^{\inv}\to X,
\end{equation}
\begin{equation}
\SPEC\left(
\bigoplus_{\bn\in \bZ^r}
\cL_1^{n_1}\otimes \dots \otimes
\cL_r^{n_r}\right)
\to X
\end{equation}
or
\begin{equation}
(\{U_i\},\{(\lambda^1_{i,j},\dots,\lambda^r_{i,j})\}).
\end{equation}
\subsubsection{Universal torsors and homogeneous coordinate rings}
We explain why the universal torsor embeds naturally as an open
subset of the affine scheme $\Spec(\HCR(X))$. 
We begin with a simple remark: if  $\cL$, $\cL'$ are two ample
classes, the ideals of $\HCR(X)$ generated by $R(X,\cL)^+$ and 
$R(X,\cL')^+$ respectively have the same radical. Indeed,
since the ample cone is open, one can find a very ample $\ecM$ and
positive integers $n$ and $m$ such that $(\cL')^m\otimes \ecM=\cL^n$ and
$(\cL')^m$ is very ample. This shows that for every $s\in
H^0(X,\cL)$, $s^n$ is in the ideal generated by $R(X,\cL')^+$.

The irrelevant ideal $\Irr(X)$ of $\HCR(X)$ is by definition the radical of the
ideal generated by $R(X,\cL)^+$, for $\cL$ an ample class.
\begin{thm}[Hassett-Tschinkel]\label{thm:hass:tsc}
There is a natural $\TNS{X}$-equivariant morphism $\tors_X\to \Spec(\HCR(X))$
which induces an isomorphism 
\begin{equation}
\tors_X \isom \Spec(\HCR(X)) \setminus \scZ(\Irr(X)).
\end{equation}
\end{thm}
When $\Pic(X)$ is of rank $1$ (hence necessarily generated by an ample
class) this is exactly what was shown in subsection \ref{subsec:torsors:under:Gm}.
If case the effective cone of $X$ is simplicial and generated
by ample classes, the result follows easily (essentially, just take the fibre product). 
In general, one can always find ample classes $\cL_1,\dots,\cL_r$
which form a basis of $\Pic(X)$.
Let 
\begin{equation}
R(X,\cL_1,\dots,\cL_r)
\eqdef 
\bigoplusu{\bn\in \bN^r} 
H^0(X,\cL_1^{n_1}\otimes\dots \otimes \cL_r^{n_r})
\end{equation}
We consider the natural $\TNS{X}$-equivariant morphisms
\begin{equation}\label{eq:tpicx:equiv}
\SPEC\left(
\oplusu{\bn\in \bN^r}
\cL_1^{n_1}
\otimes
\dots
\otimes
\cL_r^{n_r}
\right)
\to 
\Spec(\HCR(X))
\to 
\Spec(R(X,\cL_1,\dots,\cL_r)).
\end{equation}
As already seen, the composition of these two morphisms is an isomorphism
\begin{equation}
\SPEC\left(
\oplusu{\bn\in \bN^r}
\cL_1^{n_1}
\otimes
\dots
\otimes
\cL_r^{n_r}
\right)
\longisom
\Spec(R(X,\cL_1,\dots,\cL_r))\setminus \scZ(R(X,\cL_1)^+).
\end{equation}
We will show just below that the right arrow in \eqref{eq:tpicx:equiv} is
birational: this concludes the proof of theorem, since then one deduces
easily that the left arrow in \eqref{eq:tpicx:equiv} induces an isomorphism
\begin{equation}
\SPEC\left(
\oplusu{\bn\in \bN^r}
\cL_1^{n_1}
\otimes
\dots
\otimes
\cL_r^{n_r}
\right)
\longisom
\Spec(\HCR(X))\setminus \scZ(R(X,\cL_1)^+).
\end{equation}
We have to show that $\HCR(X)$ and its subring
$R(X,\cL_1,\dots,\cL_r)$ have the same fraction field. 
Take positive integers 
$e_1,\dots,e_r$ such that 
$\ecM=\cL_1^{e_1}\otimes \dots \otimes \cL_r^{e_r}$
is very ample.
Let $\ecM'=\cL_1^{d_1}\otimes \dots \otimes\cL^{d_r}$ be an effective line bundle
(the $d_i$'s are in $\bZ^n$).
For any sufficiently large integer $N$ 
there exists positive integers $f_1,\dots,f_r$ such that
\begin{equation}
\ecM^{\otimes N}\otimes \ecM'=\cL_1^{f_1}\otimes \dots \otimes \cL_r^{f_r}=\ecM''
\end{equation}
Thus if $s$ is a nonzero section of  $\ecM^{\otimes N}$, every section
of $\ecM'$ may be written as $s''/s$ where $s''$ is a section of
$\ecM''$, hence lies in $\Frac(R(X,\cL_1,\dots,\cL_r))$.

\subsubsection{Explicit embedding of the universal torsor}
We explain how, using theorem \ref{thm:hass:tsc}, the knowledge of a presentation of $\HCR(X)$ 
together with the incidence relations between the divisors of the
chosen set of generating sections,
lead to a very explicit description of the universal torsor $\tors_X$ as
a locally closed subvariety of an affine space.
Let $\{s_i\}_{i\in I}$ denote a finite family of global (non constant)
sections generating $\HCR(X)$. 
They induce an isomorphism of
$\Pic(X)$-graded $k$-algebras
$k[x_i]_{i\in I}/\scI_X\isom \HCR(X)$ 
where
$\scI_X$ is a $\Pic(X)$-homogeneous ideal, and an
$\TNS{X}$-equivariant embedding
$\Spec(\HCR(X))\hookrightarrow \bA^I$.

For $i\in I$, let $D_i$ denote the divisor of $s_i$.
Let $U$ denote the complement of the union of the $D_i$.
Since the $s_i$'s generate $\HCR(X)$, the class of the $D_i$'s generate
$\Pic(X)$ as a group and $\ceff(X)$ as a cone,
and $\Pic(U)$ is trivial. It is moreover known that $\HCR(X)$ is an UFD
(\cite{ElKaWa,BerHau:hom}), thus we may assume that the
$s_i$ are irreducible elements of $\HCR(X)$, and that no two of them
are associate.

Therefore we obtain an exact sequence of
free modules of finite rank:
\begin{equation}\label{eq:exsq:general}
0\to k[U]^{\inv}/k^{\inv}\to \oplusu{i\in I} \bZ\,D_i\to \Pic(X)\to 0
\end{equation}
which is a generalization of \eqref{eq:exsq}, valid in the toric case.

For an ample class $D$ denote by $\cI_D$ the class of subset $J$ of $I$ such that there exists
$\lambda_i\in \bN_{>0}^I$ and $m\in \bN_{>0}$ satisfying 
$\classe{\sum \lambda_i\,D_i}=\class{m\,D}$.
Then the ideals 
$\langle\prod_{i\in J}s_i\rangle_{J\in \cI_D}$ and
$\langle R(X,D)^+\rangle$
have the same radical, and 
thanks to theorem \ref{thm:hass:tsc},  $\tors_X$ may be described as the open subset of the variety
$\Spec(\HCR(X))$ given by the union over $J\in \cI_D$ of the trace of the open subset
$\prod_{i\in J}x_i\neq 0$. Setting
\begin{equation}
\wt{\cI_D}=\{J\subset I,\quad\forall K\in \cI_D,\quad J\cap K\neq \vide\},
\end{equation}
we have therefore
\begin{equation}
\tors_X
=
\Spec(\HCR(X))\setminus \cupu{
\substack{
J\subset I
\\
J\in \wt{\cI_D}
}} \,\,\capu{i\in J}\{x_i=0\}.
\end{equation}
Moreover one may check that, denoting by $\pi$ the quotient morphism
$\tors_X\to X$, the divisor $\pi^{\ast}D_i$ is the trace of the
hyperplane $\{x_i=0\}$ on $\tors_X$. 

From this one deduces the relation 
\begin{equation}\label{eq:descr:torsx}
\tors_X=\Spec(\HCR(X))\setminus \cupu{\substack{J\subset I\\\capu{i\in J} D_i=\vide}} \,\,\capu{i\in J}\{x_i=0\}.
\end{equation}
Indeed, first notice that
if $J\in \wt{\cI_D}$, then every point of $\cap_{i\in J}D_i$ is a base
point of $\abs{m\,D}$ for any $m\geq 1$. Since $D$ is ample,  
$\cap_{i\in J}D_i$ must be empty, and the RHS of
\eqref{eq:descr:torsx} is contained in $\tors_X$. And conversely, if for a $J\subset
I$ one has  $\tors_X\cap \capu{i\in J}\{x_i=0\}\neq \empty$, then 
$\pi^{\ast}(\cap_{i\in J}D_i)$, hence $\cap_{i\in J}D_i$, are non empty.
\begin{ex}
For a toric variety $X$, we thus recover the previous construction
\eqref{eq:def:torsX:toric} of $\tors_X$.
\end{ex}
\begin{ex}
For the plane blown up at three collinear points, we have, retaining
the notations of example \ref{ex:p2blownup3pts},
\begin{multline}
\tors_X=\Spec(k[x_0,\dots,x_6]/(x_1x_4+x_2x_5+x_3x_6))
\setminus 
\\
\cupu{4\leq i\neq j\leq 6} 
\{x_i=0\}\cap \{x_0=0\} 
\cup
\cupu{1\leq i\neq j\leq 3} 
\{x_i=0\}\cap \{x_j=0\}
\cup
\cupu{\substack{1\leq i\leq 3,\\ 4\leq j\leq 6,\\ j\neq i+3}}
\{x_i=0\}\cap \{x_j=0\}.
\end{multline}
\end{ex}

\subsection{Description of the functor of points of a variety whose
  homogeneous coordinate ring is finitely generated}

Retain all the notations of the previous section. We  want to describe
the functor of points of $X$ in terms of its homogeneous coordinate
ring, more precisely in terms of a presentation of the ring and the
incidence relations of the divisors of the chosen set of generating sections.
We follow very closely the approach described in the toric case. 
The novelty in the nontoric case is the nontrivial relations satisfied by the generators,
but it is rather easily dealt with.

Similarly to the
toric case, thanks to exact sequence \eqref{eq:exsq:general}, 
every element $m$ of $k[U]^{\inv}/k^{\inv}$ 
 determines an isomorphism
$c_m\,:\,\otimesu{i\in I} \str{X}(D_i)^{\otimes v_{D_i}(m)}\isom
\str{X}$ (where $v_{D_i}(m)$ is the order of annulation of the
rational function $m$ along $D_i$), and we have
$c_m\otimes c_{m'}=c_{m+m'}$.

Let $f\,:\,S\to X$ be a morphism from a $k$-scheme $S$ to $X$. 
Let $\cL_i\eqdef f^{\ast} \str{X}(D_i)$, $u_i\eqdef f^{\ast}s_i$
and for $m\in k[U]^{\inv}/k^{\inv}$, $d_m\eqdef f^{\ast} c_m$.
The datum 
\begin{equation}
(\{(\cL_i,u_i)\}_{i\in I},\{d_m\}_{m\in k[U]^{\inv}/k^{\inv}})
\end{equation}
is then an
\termin{$X$-collection on $S$} in the following sense:
\begin{defi}\label{defi:coll:gen}
An $X$-collection on a $k$-scheme $S$ is the datum of:
\begin{enumerate}
\item a family of pairs $\{(\cL_i,u_i)\}_{i\in I}$ where $\cL_i$ is a line bundle
on $S$  and $u_i$ a global section of $\cL_i$
\item
a family of isomorphisms 
$\{d_m\,:\,\otimes \cL_i^{\otimes v_{D_i}(m)}\isom \str{S}\}_{m\in k[U]^{\inv}/k^{\inv}}$
\end{enumerate}
satisfying the following conditions:
\begin{enumerate}
\item
for all $m,m'$ one has $d_m\otimes d_{m'}=d_{m+m'}$;
\item
for every $J\subset I$ such that $\cap_{i\in J}D_i=\vide$
the sections $\{u_i\}_{i\in J}$ do
not vanish simultaneously;
\item
For every homogeneous element $F$ of $\scI_X$, the section 
$F(u_i)_{i\in I}$ is the zero section.
\end{enumerate}
\end{defi}
Note that the datum of the trivializations $\{d_m\}$ allows to give a
sense to the latter condition, more precisely it allows to interpret 
$F(u_i)_{i\in I}$ as the section of a line bundle on $S$.

We have a canonical $X$-collection $C_X$ on $X$ given by $(\{(\str{X}(D_i),s_i)\},\{c_m\})$
and similarly to the toric case one shows that the maps
\begin{equation}\label{eq:iso:fop}
\map{\Hom(S,X)}{\Coll_{X,S}}{f}{f^{\ast}C_X}
\end{equation}
define an isomorphism between the functor of points of $X$ and the
functor which associates to a $k$-scheme $S$ the set $\Coll_{X,S}$ of
isomorphism classes of $X$-collections on $S$. Moreover \eqref{eq:iso:fop}
induces a bijection between the element of $\Hom(S,X)$ which do not
factor through the boundary $\cup D_i$ and the non-degenerate
$X$-collections on $S$ (those for which no one of the sections $u_i$ is
the zero section).

Now we should examine the functor $\Hom(\bP^1,X)$, or more precisely
the open subfonctor given by morphisms who do not factor through 
the boundary\footnote{As in the toric case, one could by the same kind
of arguments study the full functor, but for the sake of simplicity
this will be omitted in these notes}.
Such a morphism is entirely determined by an equivalence class of
non-degenerate $X$-collections on $\bP^1$.
Let $y\in \Pic(X)^{\vee}\cap \ceff(X)^{\vee}=\bN^I\cap \Pic(X)^{\vee}$ 
(here of course we view $\Pic(X)^{\vee}$ as a
subgroup of $\bZ^I$ through the dual of the exact sequence
\eqref{eq:exsq:general}). 
Denote by 
$
\wt{\scZ}^{y}_{X}
$
the $\TNS{X}$-invariant closed subscheme of $\Homogs_{y}\isom \prod_{i\in I}\bA^{y_i+1}\setminus \{0\}$
defined by the equations
\begin{equation}
F(P_i)=0
\end{equation}
where 
$F$ varies along the homogeneous elements of $\scI_X$. Denote by
$\scZ^{y}_X$ the image of $\wt{\scZ}^{y}_X$ in $\bP^{y}$.

Denote by $\Homogs_{y,X}$ the open subset of $\Homogs_{y}$
consisting of $I$-uple $(P_i)$ such that for every $J\subset I$ such
that $\cap_{i\in I}D_i=\vide$, the $\{P_i\}_{i\in J}$ are coprime.

Then one can show that the variety
$(\Homogs_{y,X}\cap \wt{\scZ}^{y}_X)/\TNS{X}$ 
is isomorphic to $\HOM_{U}(\bP^1,X,y)$. 
Hence, if $T_X$ denotes the torus $\Hom(k[U]^{\inv}/k^{\inv},\bG_m)$,
$\HOM_{U}(\bP^1,X,y)$ is a torsor under $T_X$ over $\bP^{y}_X\cap \scZ^{y}_{X}$.

\subsection{Application to the degree zeta function}
Let us know explain how this description of $\Hom(\bP^1,X)$ gives rise
to an expression of the degree zeta function similar to the one we
obtained in the toric case. 
We will assume that the base field $k$ is a finite field of
cardinality $q$ and restrict ourselves to the case of the classical
degree zeta function.
We have, for $y\in \ceff(X)^{\vee}\cap \Pic(X)^{\vee}$,

\begin{align}
\frac{\# \HOM_U(\bP^1,X,y)(k)}{\left(q-1\right)^{\dim(T_X)}}
&=
\# \left(\bP^{y}_{X}\cap \scZ^{y}_{X}\right) (k)
\\
&
=
\sum_{\becD\in \bP^{y}(k)}
\ind_{\bP^{y}_{X}(k)}(\becD)\ind_{\scZ^{y}_{X} (k)}(\becD)
\\
&=
\sum_{\becD\in \bP^{y}(k)}
\,\,\,\,
\left(
\sum_{0\leq \becD'\leq \becD}
\mux(\becD')
\right)
\ind_{\scZ^{y}_{X} (k)}(\becD)
\end{align}
where $\mux$ is the function determined by the relation
\begin{equation}
\forall \,\bd\in \bN^I, \quad 
\forall \,\becD\in \bP^{\bd}(k),\quad
\sum_{\becD'\leq \becD}\mux(\becD')
=
\ind_{\bP^{\bd}_{X}(k)}(\becD),
\end{equation}
for which proposition \ref{prop:mu} remains valid.
After a straightforward change of variables, the previous expression becomes
\begin{equation}\label{eq:expr}
\sum_{
\substack{
\becD\in \diveff{\bP^1}^I
\\
\forall i\in I,\quad \acc{y}{D_i}\geq \deg(\ecD_i)
}
}
\mux(\becD)
\sum_{\becD'\in \bP^{y-\deg(\becD)}}
\ind_{\scZ^{y}_{X} (k)}(\becD+\becD').
\end{equation}
For $\becD\in \diveff{\bP^1}^I$ 
such that $\acc{y}{D_i}\geq \deg(\ecD_i)$ 
let us denote by $\scN_X(\becD,y)$
the cardinality of the set
\begin{equation}
\{(P_i)\in \Homogs_{y-\deg(\becD)}(k),\,\,\,\forall F\in \scI_X^{\text{homog}},\,\,\,F(P_i.P_{\ecD_i})=0\}
\end{equation}
(where $P_{\ecD_i}\in \Homogs_{\deg(\ecD_i)}(k)$ denotes a representative of $\ecD_i\in \bP^{\deg(\ecD_i)}(k)$).
Then
$\# \HOM_U(\bP^1,X,y)(k)$
may be expressed as
\begin{equation}\label{eq:expr:bis}
\frac{1}{(q-1)^{\rk(\Pic(X))}}
\sum_{
\substack{
\becD\in \diveff{\bP^1}^I,
\\
\forall i\in I, \quad \acc{y}{D_i}\geq \deg(\ecD_i)
}
}
\mux(\becD)\,\,
\scN_X(\becD,y).
\end{equation}

This expression generalizes the one we obtained in the toric
case: apply the morphism $\#_k$ to relation \eqref{eq:expr:gen}
and use \eqref{eq:cardkmuxm}; in the toric case, the ideal $\scI_X$ is the zero ideal and $\scN_X(\becD,y)$ is nothing
else than the cardinality of $\Homogs_{y-\deg(\becD)}$.

Since the behaviour of the
M\"obius function $\mu_X$ is easily understood whether the variety $X$ is
toric or not,
the fundamental difference between the toric and non toric case in the
study of the degree zeta function is that we have to deal with
the non trivial relations satisfied by the generators of the
homogeneous coordinate ring.
Thus $\scN_X(\becD,y)$ is really the hard part to undersand
in the above expression; as far as I know, there is yet no general
procedure to handle these kind of relations; every succesful attempt
to settle Manin's conjecture using this method is highly dependent on the particular shape of the
equations defining the homogeneous coordinate ring of the involved
variety or family of varieties.

\begin{rem}
It is not clear (at least to me) what could be a sensible analog of expression
\eqref{eq:expr:bis} for the class of $\HOM_{U}(\bP^1,X,y)$ in the
Grothendieck ring of varieties.
\end{rem}

\subsection{Application to the projective plane blown up at three
  collinear points}

I will now describe very sketchly how expression \eqref{eq:expr:bis} leads to the
expected estimates for the anticanonical classical degree zeta function in a very particular
case, namely the case of the projective plane blown up at three
collinear points (see \cite{Bou:fam} for a generalization). We retain the notations of example \eqref{ex:p2blownup3pts}.
Note that $(D_0,D_1,D_2,D_3,D_4)$ is a basis of $\Pic(X)$ and that we
have the linear equivalence relations
\begin{equation}\label{eq:lineqrel}
D_4\sim D_0+D_2+D_3,\quad
D_5\sim D_0+D_1+D_3,\quad
D_6\sim D_0+D_1+D_2.
\end{equation}
Moreover an anticanonical divisor is easily computed as
$3\,D_0+2\,D_1+2\,D_2+2\,D_3$. Note that its class coincide with the
class of the sum of the boundary divisors minus the class of the degree of the
relation defining $\HCR(X)$; this is in fact a special case of a
generalized adjunction formula, see \cite[proposition 8.5]{BerHau:Cox}.

Now let $\becD\in \diveff{\bP^1}^7$ 
and let $\bd\in \ceff(X)^{\vee}\cap \Pic(X)^{\vee}\subset \bZ^7$ such
that $\bd\geq \deg(\becD)$; note
that according to \eqref{eq:lineqrel} 
the condition $\bd\in \ceff(X)^{\vee}\cap \Pic(X)^{\vee}$
means here that $\bd$ satisifies $d_i\geq 0$ for 
$0\leq i\leq 7$ and 
\begin{equation}
d_4=d_0+d_2+d_3,\quad
d_5=d_0+d_1+d_3,\quad
d_6=d_0+d_1+d_2.
\end{equation}
Let $Q_i\in \Homogs_{\deg(\ecD_i)}$ be a representative of $\ecD_i$.
We have to estimate the number of elements $(P_0,\dots,P_7)\in \Homogs_{\bd-\deg(\becD)}$
satisfying 
\begin{equation}
P_1\,P_4\,Q_1\,Q_4+P_2\,P_5\,Q_2\,Q_5+P_3\,P_6\,Q_3\,Q_6=0.
\end{equation} 
We make a first `approximation' by allowing $P_4$, $P_5$ and $P_6$ to be zero
and use the following elementary lemma.
\begin{lemma}
Let $D$ be a nonnegative integer, $e_1$, $e_2$ and $e_3$ be nonnegative
integers such that $e_i\leq D$. 
Moreover we assume that
$e_i+e_j\leq D$ holds whenever $i\neq j$.
Let $(R_1,R_2,R_3)$ be an element of $\Homogs_{(e_1,e_2,e_3)}(k)$.
Then the dimension of the subspace
set 
\begin{equation}
\{(R'_1,R'_2,R'_3)\in \Homog_{(D-e_1,D-e_2,D-e_3)},\quad R_1\,R'_1+R_2\,R'_2+R_3\,R'_3=0\}
\end{equation}
is 
\begin{equation}
2+2\,D-(e_1+e_2+e_3)+\deg(\gcd(P_1,P_2,P_3)).
\end{equation}
\end{lemma}
We apply this lemma to the above situation, setting $R_i=P_i\,Q_i\,Q_{i+3}$
and $R'_i=P_{i+3}$ (hence $e_i=d_i+\deg(\ecD_{i+3})$ and $D=d_i+d_{i+3}=d_0+d_1+d_2+d_3$), and we find that under the conditions
\begin{equation}\label{eq:cond:deg}
\deg(\ecD_i)+\deg(\ecD_j)\leq d_0+d_k\,\quad \{i,j,k\}=\{1,2,3\}
\end{equation}
we have 
\begin{multline}\label{eq:scnx}
\scN_X(\bd,\becD)
=
q^{2+2\,d_0+d_1+d_2+d_3-\deg(\ecD_4)-\deg(\ecD_5)-\deg(\ecD_6)}
\\
\times
\sum_{\becE\in \bP^{(d_i-\deg(\ecD_i))_{0\leq i\leq 3}}}
q^{\deg(\gcd(\cE_1+\cD_1+\cD_4,\cE_2+\cD_2+\cD_5,\cE_3+\cD_3+\cD_6))}.
\end{multline}
Our second `approximation' will be to assume that 
\eqref{eq:scnx} holds regardless \eqref{eq:cond:deg} are satisfied or not.

Now for $\bd\in \bN^4$ and $\becD\in \diveff{\bP^1}^7$ we want to estimate 
the quantity
\begin{equation}
\sum_{\becE\in \bP^{\bd}}
q^{\deg(\gcd(\cE_1+\cD_1+\cD_4,\cE_2+\cD_2+\cD_5,\cE_3+\cD_3+\cD_6))}.
\end{equation}
We consider the generating series
\begin{multline}
\sum_{\bd\in \bN^4} \sum_{\becE\in \bP^{\bd}}
q^{\deg(\gcd(\cE_1+\cD_1+\cD_4,\cE_2+\cD_2+\cD_5,\cE_3+\cD_3+\cD_6))}
\prod_{0\leq i\leq 3} t_i^{d_i}
\\
=
\sum_{\becD\in \diveff{\bP^1}^4}
q^{\deg(\gcd(\cE_1+\cD_1+\cD_4,\cE_2+\cD_2+\cD_5,\cE_3+\cD_3+\cD_6))}
\prod_{0\leq i\leq 3} t_i^{\deg(\cE_i)}
\end{multline}
wich decomposes into an Euler product
\begin{equation}\label{eq:eul:prod}
\prod_{\cP\in (\bP^1_k)^{(0)}} \sum_{\bn\in
  \bN^4}q^{\deg(\cP)\,
\Minu{1\leq i\leq 3}(n_i+\ord_\cP(\cD_i)+\ord_\cP(\cD_{i+3}))}
\,\prod_{0\leq i\leq 3}t_i^{\deg(\cP)\,n_i}.
\end{equation}
Let us explain what happens in the case $\becD=(0,\dots,0)$.
It is rather easy to check the identity
\begin{equation}
\sum_{\bn\in  \bN^4}\theta^{\,\Min(n_1,n_2,n_3)}\,\prod_{0\leq i\leq 3}t_i^{\,n_i}
=
\frac{1-t_1\,t_2\,t_3}{1-\theta\,t_1\,t_2\,t_3}\prod_{1\leq i\leq 3}\frac{1}{1-t_i}.
\end{equation}
Thus \eqref{eq:eul:prod} may be rewritten as
\begin{equation}
\prod_{\cP\in (\bP^1_{k})^{(0)}}
\frac{1-(t_1\,t_2\,t_3)^{\deg(\cP)}}{1-(q\,t_1\,t_2\,t_3)^{\deg(\cP)}} \prod_{0\leq i\leq 3}\ZHW(\bP^1_k,t)
\end{equation}
(recall that $\ZHW(\bP^1_k,t)=\frac{1}{(1-t)(1-q\,t)}$ is the
Hasse--Weil zeta function of $\bP^1_k$). Now the first factor of
the above expression defines a holomorphic function $F$ in the polydisc
$\prod\{\abs{t_i}\leq q^{-1+\eps}\}$ for sufficiently small $\eps>0$.
Using Cauchy estimates, one obtains the approximation
\begin{equation}
\sum_{\becE\in \bP^{\bd}}
q^{\deg(\gcd(\cE_1,\cE_2,\cE_3))}
\sim
F(q^{-1},\dots,q^{-1})\,q^{d_0+d_1+d_2+d_3}
\end{equation}
In case $\becD\neq (0,\dots,0)$, an analogous reasoning shows the approximation
\begin{equation}
\sum_{\becE\in \bP^{\bd}}
q^{\deg(\gcd(\cE_1+\cD_1+\cD_4,\cE_2+\cD_2+\cD_5,\cE_3+\cD_3+\cD_6))}
\sim
F_{\becD}(q^{-1},\dots,q^{-1})\,q^{d_0+d_1+d_2+d_3}
\end{equation}
where $F_{\becD}(q^{-1},\dots,q^{-1})$ has an explicit expression 
as an Euler product $\prod_{\cP}\wt{F}_{\becD}(q^{-\deg(\cP)})$, 
$\wt{F}_{\becD}$ being a rational function, 
 depending only on the $7$-uple of integers $(\ord_P(\ecD_i))$.

As a third 'approximation' we will assume that the above
estimation is in fact an equality, thus obtaining
\begin{equation}
\scN_X(\bd,\becD)=F_{\becD}(q^{-1},\dots,q^{-1})\,q^{\,2+3 d_0+2\,d_1+2\,d_2+2\,d_3-\sum_{0\leq
    i\leq 6}\deg(\ecD_i)}.
\end{equation}
Recalling that the anticanonical class is given by
$3\,D_0+D_1+D_2+D_3$, this may be rewritten as
\begin{equation}
\scN_X(\bd,\becD)=F_{\becD}(q^{-1},\dots,q^{-1})\,q^{\,\dim(X)+\acc{\bd}{\antican{X}}
-\sum_{0\leq    i\leq 6}\deg(\ecD_i)}.
\end{equation}
Our last 'approximation' will be to drop the conditions 
$\acc{\bd}{D_i}\geq \deg(\ecD_i)$
appearing in the summation in expression \eqref{eq:expr:bis}.

Modulo all the previous approximations, the classical anticanonical
degree zeta function 
may be now  written as
\begin{multline}\label{eq:expr:fin}
q^{\dim(X)}
\sum_{\becD\in \diveff{\bP^1}^I}\mux(\becD)
F_{\becD}(q^{-1},\dots,q^{-1})
q^{-\sum_{0\leq i\leq 6}\deg(\ecD_i)}
\\
\times
\sum_{
\bd\in \ceff(X)^{\vee}\cap \Pic(X)^{\vee}
}
(q\,t)^{\acc{\bd}{\antican{X}}}
\end{multline}
The second
factor is exactly
$\text{sp}_{\antican{X}}Z(\Pic(X)^{\vee},\ceff(X)^{\vee})(q\,t)$.

Now the main task we are left with in order to show that the answer to
question \ref{ques:fin} is indeed positive, is to 
establish that all the above 'approximations' can be justified more
rigorously through the introduction of error terms which
are indeed $(q^{-1},\rk(\Pic(X))-1)$ controlled. Roughly, this can be
done using a regular decomposition of the effective cone analogous to the one
used in the toric case, but there is a certain amount of technical subtelties
that will not be discussed here (see \cite{Bou:compt,Bou:fam}).

Regarding Peyre's refinement of Manin's conjecture
discussed at the end of section \ref{subsec:dzf:tv},
another task is to show
that the constant given by the first factor of \eqref{eq:expr:fin} may
be expressed as the Tamagawa number
\begin{equation}\label{eq:tam:bis}
\frac{q^{\dim(X)}}{(1-q^{-1})^{\rk{\Pic(X)}}}
\prod_{\cP\in (\bP^1_k)^{(0)}}
(1-q^{-\deg(\cP)})^{\rk(\Pic(X))}\,
\frac
{\# X(\kappa_{\cP})}
{q^{\,\deg(\cP)\,\dim(X)}}.
\end{equation}
But using properties of $\mux$ and $F_{\becD}$, 
the first factor of \eqref{eq:expr:fin}
may be rewritten as the Euler product
\begin{equation}
\prod_{\cP\in (\bP^1_k)^{(0)}}
\sum_{\bn\in \{0,1\}^7}
\mux^0(\bn)
\wt{F}_{\bn}(q^{-\deg(\cP)})
q^{-\deg(\cP)\sum n_i}
\end{equation}
Hence we must check, for every $\cP\in (\bP^1_k)^{(0)}$, the following identity
\begin{equation}\label{eq:equality:to:check}
(1-q^{-\deg(\cP)})^{\rk(\Pic(X))}\,
\frac{\# X(\kappa_{\cP})}
{q^{\,\deg(\cP)\,\dim(X)}}
=
\sum_{\bn\in \{0,1\}^7}
\mux^0(\bn)
\wt{F}_{\bn}(q^{-\deg(\cP)})
q^{-\deg(\cP)\sum n_i}.
\end{equation}
Note that $\# X(\kappa_{\cP})=1+4\,q^{\deg(\cP)}+q^{2\,\deg(\cP)}$, hence
\eqref{eq:equality:to:check} may be seen as a formal identity between
two rational functions in the variable $q^{\deg(\cP)}$, which may be
checked in a finite amount of time (recall that we have an
explicit expression for the rational functions $\wt{F}_{\bn}$; of
course a 
computer algebra system
 may be helpful...). One can also try to exploit the
following relation, which holds for every finite $k$-extension $L$. 
This is a generalization of proposition
\ref{prop:rel:tor} to the nontoric case, valid for every $k$-variety
$X$ having a finitely generated homogeneous coordinate ring:
\begin{equation}\label{eq:rel:mux:tors} 
\sum_{\bn\in  \{0,1\}^I}
\mux^0(\bn)\,\frac{\# \tors_{X,\bn}(L)}{(\# L)^{\dim(\tors_X)}}
=
(1-\# L)^{\rk(\Pic(X))}\frac{\# X(L)}{(\# L)^{\dim(X)}}
\end{equation}
Here we denote by $\tors_{X,\bn}$ the intersection of $\tors_X\subset \bA^I$ with
the subspace $\capu{i,\,n_i=1}\{x_i=0\}$. 
The proof goes along the same line that the proof of proposition \ref{prop:rel:tor} 
and from \eqref{eq:rel:mux:tors} one may derive a slightly more conceptual proof
of \eqref{eq:equality:to:check} (see \cite{Bou:compt}). 
But to our mind this still does not explain in a satisfactory way why
\eqref{eq:equality:to:check} holds, and it would be nice to find a
genuine conceptual explanation.

It is interesting to note how very similar arguments provide an answer
to question \ref{ques:bat} for $X$ (though here the obtained result is
also a consequence of \cite{KimLeeOh:rat}).
Let us sketch very roughly how this is done: the idea is to study
\begin{equation}\label{eq:lim}
\lim_{r\to +\infty} p^{-r(\dim(X)+\acc{y}{\antican{X}})}\# \HOM_U(\bP^1,X,y)(\bF_{p^r}).
\end{equation}
Using the same kind of approximations for \eqref{eq:expr:bis} as before, one obtains
that $\# \HOM_U(\bP^1,X,y)(\bF_{p^r})$
may be estimated by
\begin{equation}
p^{\,r(\dim(X)+\acc{y}{\antican{X}})}
\sum_{\becD\in \diveff{\bP^1_{\bF_{p^r}}}^7}
F_{\becD}(p^{-r},\dots,p^{-r})
\,
p^{-r\left(\sum_{0\leq i\leq 6}\deg(\ecD_i)\right)}.
\end{equation}
Decomposing the sum as an Euler product and using dominated
convergence and the properties of $\mu_X$ and $F_{\bn}$, 
one shows that the sum appearing in the previous
expression converges to $1$ when $r\to +\infty$.
Hence (after, of course, having rigorously justified the
approximations) the limit in \eqref{eq:lim} is $1$, and standard arguments
invoking Weil conjectures show that this implies that 
$\HOM_U(\bP^1,X,y)$ is geometrically irreducible, of dimension
$\dim(X)+\acc{y}{\antican{X}}$
(this holds for any $y\in \ceff(X)^{\vee}\cap \Pic(X)^{\vee}$).

One of the key ingredient in the above (sketch of) proof of the
geometric Manin's conjecture for the plane blown up at three collinear
points was the property that the homogeneous coordinate ring 
has only one relation and that there exists $I_0\subset I$ such that the classes of
$\{D_{i}\}_{i\in I_0}$ form a basis of $\Pic(X)$ and the relation is
linear with respect to the variables $\{s_i\}_{i\in I\setminus I_0}$.
In some sense, in the context of the approach of our counting problem via
homogeneous coordinate ring, this situation might be considered as the simplest one once
the case of toric varieties (for which there are no relations) has
been excluded. Note that along varieties for which the hypotheses hold one finds a lot of generalized del
Pezzo surfaces whose homogeneous coordinate ring has one relation (see
\cite{Der:sdp:ut:hyp} for their complete classification).
One might hope that the techniques employed may lead to
a kind of uniform proof of Manin's conjecture for varieties
satisfying the above requirements (see \cite{Bou:fam} for a beginning of
justification), though even under the mere above hypotheses the control of the
error terms seems to be a very hard task in general. 
One of the main problem is that what was designated in the above sketch of proof by the 'second approximation'
does not seem to lead in general to controllable error terms; a
somewhat hidden crucial point in the
specific case considered above is that the classes of the divisors 
$(D_i)_{i\in I\setminus  I_0}$ are, in some sense, 'sufficiently large' with respect to the
degree of the relation defining $\HCR(X)$.

Of course, one could also try to draw inspiration from works dealing
with Manin's conjecture for generalized del Pezzo surfaces in the
arithmetic case, such as \eg \cite{dlBBD,BroDer:DP4A4}, for which such
large degree conditions do not intervene. But one should
notice that in these works the base field is almost always the field
of rational numbers and that extending the methods to arbitrary
number fields seems to be a quite delicate task. 
On the other hand, though we 
have not given the details in this survey, the above sketched method
generalizes rather easily when $\bP^1$ is replaced by an arbitrary smooth
projective curve.

\end{document}